\documentclass[12pt]{amsart}
\usepackage{amsmath,amsthm,amsfonts,amssymb,xcolor}
\usepackage{graphicx}
\usepackage{mathrsfs}
\usepackage{enumitem}
\usepackage{etoolbox}
\apptocmd{\sloppy}{\hbadness 10000\relax}{}{}
\apptocmd{\sloppy}{\vbadness 10000\relax}{}{}
\usepackage{hyperref}
\usepackage[letterpaper,margin=1.1in]{geometry}

\numberwithin{equation}{section}
\theoremstyle{plain}
\newtheorem*{atst}{Analyst's Traveling Salesman Theorem}

\newtheorem{thm}{Theorem}[section]
\newtheorem{prop}[thm]{Proposition}
\newtheorem{cor}[thm]{Corollary}

\newtheorem{lem}[thm]{Lemma}
\newtheorem{claim}[thm]{Claim}

\theoremstyle{definition}
\newtheorem{rem}[thm]{Remark}
\newtheorem{definition}[thm]{Definition}
\newtheorem{ex}[thm]{Example}

\def\R{\mathbb{R}}
\def\Z{\mathbb{Z}}
\def\N{\mathbb{N}}

\def\W{\mathcal{W}}

\def\e{\epsilon}
\def\g{\gamma}
\def\d{\delta}
\def\D{\Delta}

\def\G{\Gamma}
\def\a{\alpha}

\newcommand{\Flat}{\mathsf{Flat}}

\newcommand{\diam}{\mathop\mathrm{diam}\nolimits}
\newcommand{\dist}{\mathop\mathrm{dist}\nolimits}
\newcommand{\side}{\mathop\mathrm{side}\nolimits}

\newcommand{\Lip}{\mathop\mathrm{Lip}\nolimits}
\newcommand{\res}{\hbox{ {\vrule height .22cm}{\leaders\hrule\hskip.2cm} }}

\newcommand{\card}{\operatorname{card}}

\newcommand{\Haus}{\mathcal{H}}
\newcommand{\Length}{\mathop\mathrm{length}\nolimits}

\newcommand{\excess}{\mathop\mathrm{excess}\nolimits}

\begin{document}

\title[H\"older curves and parameterizations]{H\"older curves and parameterizations in the Analyst's Traveling Salesman theorem}

\author{Matthew Badger \and Lisa Naples \and Vyron Vellis}
\thanks{M.~Badger was partially supported by NSF DMS grants 1500382 and 1650546. L.~Naples was partially supported by NSF DMS grant 1650546. V.~Vellis was partially supported by NSF DMS grant 1800731}
\date{April 3, 2019}

\address{Department of Mathematics\\ University of Connecticut\\ Storrs, CT 06269-1009}
\email{matthew.badger@uconn.edu}
\address{Department of Mathematics\\ University of Connecticut\\ Storrs, CT 06269-1009}
\email{lisa.naples@uconn.edu}
\address{Department of Mathematics\\ University of Connecticut\\ Storrs, CT 06269-1009}
\email{vyron.vellis@uconn.edu}

\begin{abstract} We investigate the geometry of sets in Euclidean and infinite-dimensional Hilbert spaces. We establish sufficient conditions that ensure a set of points is contained in the image of a $(1/s)$-H\"older continuous map $f:[0,1]\rightarrow l^2$, with $s>1$. Our results are motivated by and generalize the ``sufficient half" of the Analyst's Traveling Salesman Theorem, which characterizes subsets of rectifiable curves in $\R^N$ or $l^2$ in terms of a quadratic sum of linear approximation numbers called Jones' beta numbers. The original proof of the Analyst's Traveling Salesman Theorem depends on a well-known metric characterization of rectifiable curves from the 1920s, which is not available for higher-dimensional curves such as H\"older curves. To overcome this obstacle, we reimagine Jones' non-parametric proof and show how to construct parameterizations of the intermediate approximating curves $f_k([0,1])$. We then find conditions in terms of \emph{tube approximations} that ensure the approximating curves converge to a H\"older curve. As an application to the geometry of measures, we identify conditions that guarantee fractional rectifiability of pointwise doubling measures in $\R^N$.\end{abstract}

\maketitle

\tableofcontents

\renewcommand{\thepart}{\Roman{part}}

\section{Introduction}\label{sec:intro}

The ubiquitous Traveling Salesman problem \cite{lawler,gutin,applegate}  is to find a tour of edges on a finite graph that returns to the initial vertex and has the shortest possible length. The Analyst's Traveling Salesman problem \cite{Jones-TST,Schul-survey} is to find a rectifiable curve that contains a finite or infinite, bounded set of points in a metric space that has the shortest possible length. The former problem always has a solution yet is computationally hard, while the latter problem may or may not have any solution at all. A sophisticated example from Geometric Measure Theory of a bounded point set that is not contained in any rectifiable curve is a \emph{Besicovitch irregular} set \cite{Bes28} (see \S \ref{sec:gmt} below); a trivial example is a solid square in the plane. Tests to decide which sets are contained in a rectifiable curve have been found in $\R^2$ \cite{Jones-TST}, $\R^N$ \cite{Ok-TST}, $l^2$ \cite{Schul-Hilbert}, the first Heisenberg group \cite{Li-Schul1,Li-Schul2}, Carnot groups \cite{Carnot-TSP,Li-TSP}, Laakso-type spaces \cite{Guy-Schul}, and in general metric spaces \cite{Hah05,Hah08,DS-metric}. Applications of Jones' and Okikiolu's solution of the Analyst's TSP in $\R^N$ have been given in Complex Analysis \cite{BJ,Bishop-dyn, Bishop-QC}, Dynamics and Probability \cite{wiggly, frontier}, Geometric Measure Theory \cite{BS1,BS3}, Harmonic Analysis \cite{Tolsa-P}, and Metric Geometry \cite{Naor-Peres,shortcuts}.

Let $E\subset\R^N$ be a nonempty set and let $Q\subset \R^N$ be a bounded set of positive diameter (such as a ball or a cube). Following \cite{Jones-TST}, the \emph{Jones beta number} $\beta_{E}(Q)$ is defined by
\begin{equation*} \beta_{E}(Q) := \inf_{\ell} \sup_{x\in E\cap Q}\frac{\dist(x,\ell)}{\diam{Q}}\in[0,1],\end{equation*}
where $\ell$ ranges over all straight lines in $\R^N$, if $E\cap Q\neq\emptyset$, and by $\beta_E(Q)=0$, if $E\cap Q=\emptyset$. Let $\D(\R^N)$ denote the family of dyadic cubes in $\R^N$,
\begin{equation*} \D(\R^N) := \{ [2^{k}m_1,2^{k}(m_1+1)]\times\cdots\times [2^{k}m_N,2^{k}(m_N+1)] : m_1,\dots,m_N, k\in \mathbb{Z}\}.\end{equation*} Given a cube $Q$ and a scaling factor $\lambda>0$, we let $\lambda Q$ denote the concentric dilate of $Q$ by $\lambda$.

\begin{atst}[{\cite{Jones-TST, Ok-TST}}] A bounded set $E\subset\R^N$ is contained in a rectifiable curve $\Gamma=f([0,1])$ if and only if
\begin{equation}\label{eq:betasum}
S_E:=\sum_{Q\in\D(\R^N)}\beta_{E}(3Q)^2 \diam{Q} < \infty.
\end{equation}
More precisely, \begin{enumerate}
\item If $\Gamma$ is any curve containing $E$, then $\diam E + S_E\lesssim_N \Length(\Gamma)$.
\item If $S_E<\infty$, then there exists a curve $\Gamma\supset E$ such that $\Length(\Gamma)\lesssim_N \diam E+ S_E$.
\end{enumerate}
\end{atst}

We may refer to statements (1) and (2) as the \emph{necessary half} and \emph{sufficient half} of the Analyst's Traveling Salesman theorem, respectively. The theorem is valid if the length of a curve $\Gamma=f([0,1])$ is interpreted either as the 1-dimensional Hausdorff measure of the set $\Gamma$ or as the total variation of the parameterization $f$. A curious feature of the known proofs of the sufficient half of the Analyst's TST (see \cite{Jones-TST} or \cite{BS3}) is that a rectifiable curve $\Gamma$ containing the set $E$ is constructed as the limit of piecewise linear curves $\Gamma_k$ containing a $2^{-k}$-net for $E$ \emph{without constructing a parameterization} of $\Gamma_k$ or $\Gamma$. This aspect of the proof breaks the analogy with the classical TSP, in which one is asked to find a minimal tour of a graph.

In this paper, we provide a \emph{parametric proof} of the sufficient half of the Analyst's TST, which more closely parallels the classical TSP. Beyond its intrinsic interest, the method that we provide is important, because it allows us to establish multiscale tests to ensure that a bounded set of points in $\R^N$ is contained in a \emph{$(1/s)$-H\"older continuous curve} with $s\in(1,N)$. Rectifiable curves correspond precisely to the class of Lipschitz curves ($s=1$). Remarkably, in the H\"older Traveling Salesman theorem (see \S\S \ref{sec:statements} and \ref{sec:Holder}), we can replace approximation by lines in the definition of the Jones beta numbers with \emph{approximation by thin tubes}. For a self-contained statement of the ``parametric" Analyst's TST, see \S\ref{sec:Lip}. While our focus in this paper is primarily on sets, we are motivated by open questions about the structure of Radon measures \cite{BV,Badger-survey}. For applications of our H\"older Traveling Salesman theorems to fractional rectifiability of measures, see \S\ref{sec:gmt}.

\subsection{H\"older Traveling Salesman Theorem(s)}\label{sec:statements} A \emph{$(1/s)$-H\"older curve} $\Gamma$ in $\R^N$ is the image of a continuous map $f:[0,1]\rightarrow\R^N$ satisfying the H\"older condition, $$|f(x)-f(y)|\leq H|x-y|^{1/s}\quad\text{for all }x,y\in[0,1],$$ where $s\in[1,\infty)$ and $H$ is a finite constant independent of $x$ and $y$. A 1-H\"older curve is also called a \emph{Lipschitz curve} or a \emph{rectifiable curve}. While non-trivial rectifiable curves always have topological dimension 1 and asymptotically resemble a unique tangent line $\Haus^1$ almost everywhere, $(1/s)$-H\"older curves with $s>1$ exhibit a variety of more complicated behaviors. For example,
\begin{itemize}
\item an $m$-dimensional cube in $\R^N$ ($m\leq N$) is a $(1/m)$-H\"older curve;
\item the von Koch snowflake is a $\log_4(3)$-H\"older curve; and,
\item the standard Sierpi\'nski carpet is a $\log_8(3)$-H\"older curve.
\end{itemize}
In fact, Remes \cite{Remes} proved that any compact, connected self-similar set $K\subset\R^N$ of Hausdorff dimension $s$ that satisfies the open set condition is a $(1/s)$-H\"older curve. For related work on space-filling curves generated by graph-directed iterated function systems, see Rao and Zhang \cite{RZ16}.

Towards a H\"older version of the Analyst's Traveling Salesman theorem, the first and third authors proved in \cite{BV} as a test case that if $s>1$, $E\subset\R^N$ is bounded, and
\begin{equation*} \label{sum-no-beta} \sum_{\substack{Q\in \D(\R^N)\\ Q\cap E \neq \emptyset,\, \side Q\leq 1}}(\diam{Q})^s < \infty,\end{equation*}
then $E$ is contained in a $(1/s)$-H\"older curve. By establishing a parametric version of Jones' proof of the sufficient half of the Analyst's TST, we are able to obtain the following substantial improvement.

\begin{thm}[H\"older Traveling Salesman I]\label{thm:main2}
For all $N\geq 2$ and $s> 1$, there exists $\beta_0 \in (0,1)$ such that if $E\subset\R^N$ is bounded and
\begin{equation}\label{eq:thm2}
S^{s,+}_E:=\sum_{\substack{Q\in\D(\R^N)\\ \beta_E(3Q)\geq \beta_0}} (\diam{Q})^s < \infty,
\end{equation}
then $E$ is contained in a $(1/s)$-H\"older curve. More precisely, $E\subset \Gamma=f([0,1])$ for some $(1/s)$-H\"older map $f:[0,1]\rightarrow\R^N$ with H\"older constant $H\lesssim_{N,s} \diam E + (\diam{E})^{1-s}S_E^{s,+}.$
\end{thm}

Condition (\ref{eq:thm2}) implies that at $\Haus^s$ almost every point, the set $E$ asymptotically lies in sufficiently \emph{thin tubes}. Theorem \ref{thm:main2} provides a sufficient test that identifies all subsets of some well-known H\"older curves such as snowflakes of small dimension. However, because of the richness of H\"older geometry, a condition using Jones beta numbers alone such as \eqref{eq:thm2} cannot be expected to hold for all subsets of every H\"older curve. Indeed \eqref{eq:thm2} fails when $E$ is a carpet or a square. For expanded discussion and related examples, see \textsection\ref{sec:notnec}.

Theorem 1.1 is a simplification of our main result, which is adapted to a nested sequence of separated sets in a finite or infinite-dimensional Hilbert space. See Theorem \ref{thm:main}.

To estimate the size of the constant $\beta_0$ in Theorem \ref{thm:main2}, see Lemma \ref{lem:flat2} and Remark \ref{rem:alpha-beta}. The following variant of Theorem \ref{thm:main2} is an immediate corollary, whose hypothesis does not require knowledge of $\beta_0$.

\begin{cor}[H\"older Traveling Salesman II] \label{cor:main} Suppose that $N\geq 2$, $s>1$, and $p>0$. If $E \subset \R^N$ is bounded and \begin{equation} \label{eq:sum-p} S_E^{s,p}:=\sum_{Q\in\D(\R^N)}\beta_{E}(3Q)^{p} (\diam{Q})^s < \infty,  \end{equation}
then $E$ is contained in a $(1/s)$-H\"older curve. More precisely, $E\subset \Gamma=f([0,1])$ for some $(1/s)$-H\"older map $f:[0,1]\rightarrow\R^N$ with H\"older constant  \begin{equation*} H\lesssim_{N,s} \diam E + \beta_0^{-p}(\diam{E})^{1-s}S_E^{s,p},\end{equation*} where $\beta_0$ is the constant appearing in Theorem \ref{thm:main2}.
\end{cor}

A good exercise is to prove that any bounded set $E$ in $\R^N$ satisfying condition \eqref{eq:sum-p} with $s>1$ has zero $s$-dimensional Hausdorff measure. In \textsection\ref{sec:null}, we construct a countable, compact set $E$ (hence $\Haus^s(E)=0$) such that $E$ is not contained in any $(1/s)$-H\"older curve with $1\leq s<N$. Thus, Corollary \ref{cor:main} is nonvacuous.

\subsection{Overview of the proof of Theorem \ref{thm:main2}}\label{sec:overview}

In order to properly discuss the proof of Theorem \ref{thm:main2}, we quickly sketch the proof of the sufficient half of the Analyst's TST. The proof splits into three steps. In the first step, one uses the Jones beta numbers $\beta_E(3Q)$ (in particular, whether they are large or small) to construct a sequence of finite, connected graphs $G_k$ in $\R^N$ with straight edges that converge in the Hausdorff distance to a compact, connected set $G$ containing $E$. Each graph $G_k$ is obtained by refining $G_{k-1}$ and resembles a flat arc near points of $E$ that look flat at scale $2^{-k}$. In step two, one uses the structure of the graphs $G_k$ and the Pythagorean theorem to prove the existence of a constant $C>0$ such that
\begin{equation}\label{e:Jones-estimate} \mathcal{H}^1(G_{k+1}) \leq  \mathcal{H}^1(G_{k}) + C \sum_{\substack{Q \in \D(\R^N) \\ \side{Q}\simeq 2^{-k}}}\beta_E(3Q)^2 \diam{Q}.\end{equation}
Condition (\ref{eq:betasum}) and {G}o\l \c{a}b's  semicontinuity theorem (e.g.~see \cite{AO-curves}) ensure that $\Haus^1(G)\leq \liminf_{k\rightarrow\infty} \mathcal{H}^1(G_k) < \infty$. Thus, the first two parts of the proof yield a compact, connected set $G$ containing $E$ with $\mathcal{H}^1(G)<\infty$. The final step is to invoke Wa\.zewski's theorem to conclude \emph{existence} of a Lipschitz parameterization for $G$: if $G \subset\R^N$ is connected, compact, and $\Haus^1(G)<\infty$, then there exists a Lipschitz map $f:[0,1]\rightarrow\R^N$ such that $G=f([0,1])$ (see \cite[Theorem 4.4]{AO-curves} or \cite[Lemma 3.7]{Schul-Hilbert}). Note that the condition $\mathcal{H}^1(G)<\infty$ promotes connectedness of $G$ to local connectedness (because $G$ is a curve).

In the H\"older setting, there are at least two obstacles to following the approach above. First and foremost, a naive analogue of Wa\.zewski's theorem cannot hold for H\"older maps, since the condition $\mathcal{H}^s(G)<\infty$ does not imply a continuum is locally connected when $s>1$ (e.g.~the topologist's comb). What is more, even if $G$ is assumed to be an Ahlfors $s$-regular curve with finite $\Haus^s$ measure, we cannot conclude that $G$ is a $(1/s)$-H\"older curve; we provide examples in \textsection\ref{sec:regularcurve} using a theorem of Mart\'in and Mattila \cite{MM2000}. Another obstacle is the well-known failure of {G}o\l \c{a}b's  semicontinuity theorem for Hausdorff measures $\Haus^s$ with $s>1$. Thus, in a proof of a H\"older Traveling Salesman theorem, estimating the Hausdorff measure of approximating sets has no direct use.

To overcome these obstacles, we reimagine the proof of the Analyst's TST, and in \S\ref{sec:curve}, give a procedure to construct  a sequence of partitions $\{\mathscr{I}_k\}_{k\geq 0}$ of $[0,1]$ and a sequence of piecewise linear maps $\{f_k : [0,1] \to \R^N\}_{k\geq 0}$ that \emph{parameterize} approximating graphs $G_k$. Each map $f_k$ is built by carefully refining $f_{k-1}$ to ensure that $\|f_{k}-f_{k-1}\|_\infty\lesssim 2^{-k}$. This guarantees that the maps $f_k$ have a uniform limit $f$ whose image contains the Hausdorff limit of $G_k$ (and hence $E$). To prove that $f$ is H\"older continuous, one must estimate growth of the Lipschitz constants of the maps $f_k$ (see Appendix \ref{sec:lip-hold} for the basic method). In \textsection\ref{sec:mass}, we introduce a notion of \emph{mass} of intervals $I\in\mathscr{I}_k$, defined using the $s$-power of diameters of images $f_l(J)$ of intervals $J\subset I$, $l\geq k$. This lets us record estimates in the domain of the map rather than its image, and in \S\ref{sec:mass}, we provide a mass-centric analogue of \eqref{e:Jones-estimate} that is adapted to the H\"older setting. In turn, this lets us estimate the Lipschitz constants of the maps $f_k$ and complete the proof of the H\"older Traveling Salesman theorem in \S\ref{sec:Holder}. For completeness, we use our method to reprove and strengthen the sufficient half of the Analyst's TST in \S\ref{sec:Lip}.

\subsection{Wa\.zewski type theorem for flat continua}\label{sec:wazewski} The Hahn-Mazurkiewicz Theorem (e.g.~see \cite[Theorem 3.30]{hocking}) asserts that a set $E\subset\R^N$ is a continuous image of $[0,1]$ if and only if $E$ is compact, connected, and locally connected. The Wa\.zewski Theorem (for an attribution, see \cite{AO-curves}) asserts that $E\subset\R^N$ is a Lipschitz image of $[0,1]$ if and only if $E$ is compact, connected, and $\Haus^1(E)<\infty$. It is an easy exercise to check that every $(1/s)$-H\"older continuous image of $[0,1]$ is compact, connected, locally connected, and has $\Haus^s(E)<\infty$, but the converse fails when $s>1$ (see \S\ref{sec:regularcurve} below). This motivates the following, apparently open question: Is there a metric, geometric, and/or topological characterization of H\"older curves in $\R^N$?

The method of proof of the H\"older Traveling Salesman theorems leads to the following Wa\.zewski type theorem for flat continua. For the proof of Proposition \ref{thm:flat}, see \S\ref{sec:flatcontinua}. A set $E\subset\R^n$ is called \emph{Ahlfors $s$-regular} if there exist $0<c\leq C<\infty$ such that \begin{equation}\label{eq:regularity} cr^s \leq \Haus^s(E\cap B(x,r))\leq Cr^s\quad\text{for all $x\in E$ and $0<r\leq \diam E$}.\end{equation}
We say that $E$ is \emph{lower (upper) Ahlfors $s$-regular} if the first (second) inequality in (\ref{eq:regularity}) holds for all $x\in E$ and $0<r\leq \diam E$.

\begin{prop}\label{thm:flat}
There exists a constant $\beta_1 \in (0,1)$ such that if $s>1$ and $E\subset\R^N$ is compact, connected, $\Haus^s(E)<\infty$, $E$ is lower Ahlfors $s$-regular with constant $c$, and
\begin{equation}\label{e:mv}
 \beta_{E}\left(\overline{B}(x,r)\right) \leq \beta_1\quad \text{for all }x\in E\text{ and }0<r\leq \diam E, \end{equation}
then $E=f([0,1])$ for some injective $(1/s)$-H\"older continuous map $f:[0,1]\to\R^N$ with H\"older constant $H \lesssim_{s} c^{-1}\mathcal{H}^s(E)(\diam E)^{1-s}$.
\end{prop}

Inclusion of lower Ahlfors regularity in the hypothesis of Proposition \ref{thm:flat} is justifiable, because it holds automatically when $s=1$, i.e.~ every non-trivial connected set is lower Ahlfors $1$-regular. When $s>1$, a non-trivial $(1/s)$-H\"older curve is not necessarily lower Ahlfors $s$-regular, and, in fact, could have zero $\Haus^s$ measure. Nevertheless, Mart\'in and Mattila \cite{MM1993} proved that if $\Gamma$ is a $(1/s)$-H\"older curve in $\R^N$ with $\Haus^s(\Gamma)>0$, then  $$\liminf_{r\downarrow 0} \frac{\Haus^s(\Gamma\cap B(x,r))}{r^s}>0\quad \text{at $\Haus^s$-a.e.~$x\in\Gamma$}.$$ Even if it can be weakened, the lower regularity hypothesis in Proposition \ref{thm:flat} cannot be completely dropped: In \S\ref{sec:flatcurves-counterexample}, for any $s>1$ and $\beta_1 \in (0,1)$, we find a curve $E\subset\R^N$ with $\mathcal{H}^s(E) < \infty$ satisfying \eqref{e:mv} such that $E$ is not contained in a $(1/s)$-H\"older curve.

Sharp estimates on the Minkowski dimension of sets satisfying \eqref{e:mv} were provided by Mattila and Vuorinen \cite{MV}; for generalized \emph{Mattila-Vuorinen type sets}, see \cite{lsa}.

\subsection{Related Work} As noted above, one motivation for this paper is to develop tools to analyze the structure of Radon measures. See \S\ref{sec:gmt} for background and for an application of Corollary \ref{cor:main} to the fractional rectifiability of measures.

There is considerable interest in finding higher-dimensional analogues of the Analyst's Traveling Salesman theorem, for example finding a characterization of subsets of Lipschitz images of $[0,1]^2$. This problem is still open, but some positive steps were recently taken by Azzam and Schul \cite{AS-TST} for \emph{Hausdorff content lower regular sets}. Also see \cite{Villa}.

\subsection*{Acknowledgements} We would like to thank Mika Koskenoja for providing us with a copy of M. Remes' thesis \cite{Remes}. The first author would also like to thank Guy C. David for a useful conversation at an early stage of this project. We thank an anonymous referee for their close and careful reading of the initial manuscript.

\part{Proof of the H\"older Traveling Salesman Theorem}

In the first part of the paper, \S\S \ref{sec:prelim}--\ref{sec:Lip}, we establish several H\"older Traveling Salesman theorems, including Theorem \ref{thm:main2} and Theorem \ref{thm:main}. To start, in \S\ref{sec:prelim}, we introduce notation and essential concepts used in the proof, including nets, flat pairs, and variation excess. In \S\ref{sec:curve}, we present a refined version of Jones' Traveling Salesman construction, which takes a nested sequence $(V_k)_{k=0}^\infty$ of $\rho_kr_0$-separated sets,  approximating lines $\{\ell_{k,v}:v\in V_k\}_{k=0}^\infty$, and associated errors $\{\alpha_{k,v}:v\in V_k\}_{k=0}^\infty$ and outputs a sequence of partitions $\mathscr{I}_k$ of $[0,1]$ and piecewise linear maps $f_k$ such that $f_k([0,1])\supset V_k$. In \S\ref{sec:mass}, we define and estimate a discrete $s$-variation of the maps $f_k$, which is adapted to the partitions $\mathscr{I}_k$ of the domain. When $s>1$, the total $s$-mass $\mathcal{M}_s([0,1])$ associated to the sequence of maps $f_k$ fills the role that 1-dimensional Hausdorff measure $\Haus^1$ plays in Jones' proof of the Analyst's TST. In \S\ref{sec:Holder}, we use the algorithm of \S\ref{sec:curve} and the mass estimates of \S\ref{sec:mass} to prove our main theorem (see Theorem \ref{thm:main}). Finally, in \S\ref{sec:Lip}, we use our method to obtain a stronger version of the sufficient half of the Analyst's Traveling Salesman theorem. The construction presented below can be carried out in any finite or infinite-dimensional Hilbert space.

\section{Preliminaries}\label{sec:prelim}

Given numbers $x,y\geq 0$ and parameters $a_1,\dots, a_n$, we may write $x \lesssim_{a_1,\dots,a_n} y$ if there exists a positive and finite constant $C$ depending on at most  $a_1,\dots,a_n$ such that $x \leq C y$. We write $x \simeq_{a_1,\dots,a_n} y$ to denote $x \lesssim_{a_1,\dots,a_n} y$ and $y \lesssim_{a_1,\dots,a_n} x$. Similarly, we write $x\lesssim y$ or $x\simeq y$ to denote that the implicit constants are universal.

\subsection{Ordering flat sets}\label{sec:orderflat}

The following lemma shows that if a discrete set is sufficiently flat at the scale of separation, then there exists a natural linear ordering of its points. Estimates \eqref{e:graph1} and \eqref{e:graph2} are consequences of the Pythagorean theorem.

\begin{lem}[{\cite[Lemma 8.3]{BS3}}]\label{lem:approx}
Suppose that $V\subset\R^N$ is a $1$-separated set with $\card(V) \geq 2$ and there exist lines $\ell_1$ and $\ell_2$ and a number $\a \in (0,1/16]$ such that
\begin{equation*} \dist(v,\ell_i) \leq \a \qquad\text{for all }v\in V\text{ and }i=1,2.\end{equation*}
Let $\pi_i$ denote the orthogonal projection onto $\ell_i$. There exist compatible identifications of $\ell_1$ and $\ell_2$ with $\R$ such that $\pi_1(v) \leq \pi_1(v')$ if and only if $\pi_2(v) \leq \pi_2(v')$ for all $v,v' \in V$. If $v_1$ and $v_2$ are consecutive points in $V$ relative to the ordering of $\pi_1(V)$, then \begin{equation}\label{e:graph1} \Haus^1([u_1,u_2]) \leq (1+ 3\alpha^2)\cdot \Haus^1([\pi_1(u_1),\pi_1(u_2)]) \quad\text{for all }[u_1,u_2]\subseteq[v_1,v_2].\end{equation}
Moreover, \begin{equation}\label{e:graph2} \Haus^1([y_1,y_2]) \leq (1+12\alpha^2)\cdot\Haus^1([\pi_1(y_1),\pi_1(y_2)])\quad\text{for all }[y_1,y_2]\subseteq\ell_2.\end{equation}
\end{lem}

Suppose that $V$, $\ell_1$, and $\pi_1$ are given as in Lemma \ref{lem:approx} and let $v,v_1,v_2 \in V$. Given an orientation of $\ell$ (that is, an identification of $\ell_1$ with $\R$), we say \emph{$v_1$ is to the left of $v_2$} and \emph{$v_2$ is to the right of $v_1$} if $\pi_{1}(v_1) < \pi_{1}(v_2)$. We say \emph{$v$ is between $v_1$ and $v_2$} if $\pi_1(v_1) \leq \pi_1 (v) \leq \pi_1(v_2)$ or $\pi_1(v_2) \leq \pi_1 (v) \leq \pi_1(v_1)$.

\begin{lem}\label{lem:var}
Suppose that $V\subset\R^N$ is a $\d$-separated set with $\card(V) \geq 2$ and there exists a line $\ell$ and a number $\a\in (0,1/16]$ such that
$$\dist(v,\ell) \leq \a\d \quad\text{for all $v\in V$}.$$
Enumerate $V=\{v_1,\dots,v_n\}$ so that $v_{i+1}$ is to the right of $v_i$ for all $1\leq i\leq n-1$. Then \begin{equation} \label{e:var2} \sum_{i=1}^{n-1} |v_{i+1}-v_i|^s \leq (1+3\alpha^2)^s |v_1-v_n|^s\quad\text{for all }s\geq 1.\end{equation} Moreover, if $\card(V)\geq 3$, then
\begin{equation} \label{e:var3} \sum_{i=1}^{n-1}|v_{i+1}-v_i|^s \leq ((1+3\alpha^2)|v_1-v_n|-\d)^s + \delta^s\quad\text{for all }s\geq 1.\end{equation}
\end{lem}

\begin{proof} Let $\pi$ denote the orthogonal projection onto $\ell$ and put $x_i:=\pi(v_i)$. Then $$|x_{i+1}-x_i|\leq |v_{i+1}-v_i| \leq (1+3\alpha^2)|x_{i+1}-x_i|\quad\text{for all } 1\leq i\leq n-1,$$ where the first inequality holds since projections are 1-Lipschitz and the second inequality holds by Lemma \ref{lem:approx}. Assume $s\geq 1$ and $\card(V)\geq 3$. Then
\begin{equation*}\begin{split} \sum_{i=1}^{n-1} \frac{|v_{i+1}-v_i|^s}{(1+3\alpha^2)^s} &\leq \left(\sum_{i=1}^{n-2} |x_{i+1}-x_i|^s\right)+|x_n-x_{n-1}|^s\\
&\leq  \left(\sum_{i=1}^{n-2} |x_{i+1}-x_i|\right)^s +|x_n-x_{n-1}|^s\\
&= (|x_n-x_1|-|x_n-x_{n-1}|)^s + |x_n-x_{n-1}|^s\\
&\leq \left(|x_n-x_1| - \frac{\delta}{1+3\alpha^2}\right)^s + \left(\frac{\delta}{1+3\alpha^2}\right)^s\\
& \leq \left(|v_n-v_1|-\frac{\delta}{1+3\alpha^2}\right)^s + \left(\frac{\delta}{1+3\alpha^2}\right)^s,
\end{split}\end{equation*}
where the penultimate inequality holds because for any $M>0$, $\e \in (0,M)$, and $s\geq 1$, the function $f(t) = t^s + (M-t)^s$ defined on $[\e,M-\e]$ attains its maximum at $t=\e$. This establishes \eqref{e:var3}. Inequality \eqref{e:var2} follows from a similar (and easier) computation.
\end{proof}

\subsection{Nets, flat pairs, and variation excess} \label{sec:flat}

Let $(X,|\cdot|)$ denote the Hilbert space $l^2(\R)$ of square summable sequences  or the Euclidean space $\R^N$ for some $N\geq 2$.

Let $\mathscr{V}=\{(V_k,\rho_k)\}_{k\geq 0}$ be a sequence of pairs of nonempty finite sets $V_k$ in $X$ and numbers $\rho_k>0$. Assume that there exist $x_0 \in X$, $r_0>0$, $C^*\geq 1$, and $0<\xi_1\leq \xi_2<1$ such that $\mathscr{V}$ satisfies the following properties.
\begin{enumerate}
\item[(\textbf{V0})] When $k=0$, we have $\rho_0=1$. For all $k\geq 0$, we have $\xi_1 \rho_k \leq \rho_{k+1} \leq \xi_2 \rho_k$.
\item[(\textbf{V1})] When $k=0$, we have $V_0 \subset B(x_0,C^*r_0)$.
\item[(\textbf{V2})] For all $k\geq 0$, we have $V_{k} \subset V_{k+1}$.
\item[(\textbf{V3})] For all $k\geq 0$ and all distinct $v,v'\in V_k$, we have $|v-v'| \geq \rho_k r_0$.
\item[(\textbf{V4})] For all $k \geq 0$ and all  $v\in V_{k+1}$, there exists $v'\in V_{k}$ such that $|v-v'| < C^*\rho_{k+1} r_0$.
\end{enumerate} With $C^*$ and $\xi_2$ given, define the associated parameter $$A^*:=\frac{C^*}{1-\xi_2}>C^*.$$
In addition to (V0)--(V4), assume that for each $k\geq 0$ and $v\in V_k$ we are given a number $\a_{k,v}\geq 0$ and a straight line $\ell_{k,v}$ in $X$ such that
\begin{equation}\tag{\textbf{V5}} \sup_{x\in V_{k+1}\cap B(v,30A^*\rho_k r_0)} \dist(x,\ell_{k,v}) \leq \a_{k,v} \rho_{k+1}r_0. \end{equation}
We call the line $\ell_{k,v}$ an \emph{approximating line at $(k,v)$}.

The formulation of (V5) is motivated by \cite[Proposition 3.6]{BS3}. We remark that the number $\rho_{k+1}r_0$ appearing on the right hand side of (V5) is the scale of separation of points in $V_{k+1}$. While we allow $X=l^2(\R)$, each $V_k$ can be identified with a subset of $\R^{N_k}$ for some increasing sequence $N_k$ if convenient, because each $V_k$ is finite and $V_k\subset V_{k+1}$.

\begin{lem}\label{rem:card}
Let $v\in V_k$ be such that $\a_{k,v} \leq 1/16$ and fix an orientation for $\ell_{k,v}$.
\begin{enumerate}
\item If $v'\in V_{k}\cap B(v,14A^*\rho_k r_0)$ is the first point to the right of $v$, then there exist fewer than $2+2.2C^*$ points of $V_{k+1}$ between $v$ and $v'$ (inclusive).
\item There exist fewer than $1.1C^*$ points of $V_{k+1}\cap B(v,C^*\rho_{k+1}r_0)$ to the right of $v$.
\end{enumerate}
\end{lem}

\begin{proof} The points in $V_{k+1}\cap B(v,30A^*\rho_k r_0 )$ are $\rho_{k+1}r_0$-separated and are linearly ordered by Lemma \ref{lem:approx}. Let $v'$ be the first point in $V_k\cap B(v,14A^*\rho_k r_0)$ to the right of $v$ and let $w_1,\dots,w_m$ denote the points in $V_{k+1}$ that lie between $v$ and $v'$ (inclusive). By (V4), each point $w_i$ belongs to $B(v,C^*\rho_{k+1} r_0)\cup B(v',C^*\rho_{k+1} r_0)$. Let $\pi_{k,v}$ denote the orthogonal projection onto $\ell_{k,v}$. By (V3) and \eqref{e:graph1}, $1.1|\pi(w_i)-\pi(w_j)| > |w_i-w_j|\geq \rho_{k+1}r_0$ for all distinct $i,j$, since $(1+3(1/16)^2)<1.1$. It follows that there are fewer than $1.1C^*$ points $w_i$ in $V_{k+1}\cap B(v,C^*\rho_{k+1} r_0)$ to the right of $v$ and fewer than $1.1C^*$ points $w_i$ in $V_{k+1}\cap B(v',C^*\rho_{k+1} r_0)$ to the left of $v'$. The first claim follows. A similar argument gives the second claim.
\end{proof}

\begin{definition}[flat pairs] \label{def:flatpairs} Fix a parameter $\alpha_0\in(0,1/16]$. For all $k\geq 0$, define $\Flat(k) $ to be the set of pairs $(v,v')\in V_k\times V_k$ such that
\begin{enumerate}
\item $\rho_k r_0\leq |v-v'| < 14A^* \rho_k r_0$,
\item $\a_{k,v} <\alpha_0$ and $v'$ is the first point in $V_k\cap B(v,14A^*\rho_k r_0 )$ to the left or to the right of $v$ with respect to ordering induced by $\ell_{k,v}$.
\end{enumerate} Define the corresponding set of geometric line segments $\mathscr{L}_k=\{[v,v']: (v,v')\in \Flat(k)\}$.
\end{definition}

Note that the collection $\Flat(k)$ of flat pairs is not symmetric in the sense that $(v,v') \in \Flat(k)$ does imply $(v',v)\in\Flat(k)$, because $\alpha_{k,v}$ does not control $\alpha_{k,v'}$.

\begin{lem}\label{lem:tridents}
Let $e_1$, $e_2$, $e_3$ be distinct elements of $\mathscr{L}_{k}$ for some $k\geq 0$.
\begin{enumerate}
\item Edges $e_1$ and $e_2$ intersect at most in a common endpoint.
\item Edges $e_1$, $e_2$ and $e_3$ do not have a common point.
\end{enumerate}
\end{lem}

\begin{proof} Let $e_1=[v_1,v_1']$, $e_2=[v_2,v_2']$, and $e_3=[v_3,v_3']$ represent distinct elements of $\mathscr{L}_k$, where $(v_i,v_i')\in \Flat(k)$ for all $i\in\{1,2,3\}$. If two or more of the edges intersect in a common point, say $\{e_i:i\in I_0\}$ for some $I_0\subset\{1,2,3\}$ with $\card(I_0)\geq 2$, then those edges are contained in $B(v_j,30A^*\rho_k r_0)$ for each $j\in I_0$, since each edge has diameter at most $14A^*\rho_k r_0$. Note that $V_k$ is a $\rho_k r_0$ separated set, $\dist(v_i,\ell_{k,v_j})\leq \alpha_{k,v_j} \rho_{k+1}r_0< \alpha_{k,v_j}\rho_k r_0$ for all $i,j\in I_0$, and $\alpha_{k,v_j}\leq 1/16$. Thus, by Lemma \ref{lem:approx}, the vertices $\{v_i,v_i':i\in I_0\}$ are consistently linearly ordered according to their projections onto $\ell_{k,v_j}$ for each $j\in I_0$. Claims (1) and (2) follow immediately, since the segments in $\mathscr{L}_k$ emanating from a vertex $v\in V_k$ with $\alpha_{k,v}<\alpha_0$ are only drawn to the first vertex in $V_k \cap B(v,14A^*\rho_k r_0)$ to the left or right of $v$ with respect to the projection onto $\ell_{k,v}$.
\end{proof}

Given a pair $(v,v')\in \Flat(k)$, let $V_{k+1}(v,v')$ denote the set of all points $x \in V_{k+1}\cap B(v, 14A^*\rho_k r_0)$ such that $x$ lies between $v$ and $v'$ (including $v$ and $v'$).

\begin{definition}[variation excess] For all $s\geq 1$, for all $k\geq 0$, and for all $(v,v')\in\Flat(k)$, define the \emph{$s$-variation excess} $\tau_s(k,v,v')$ by $$\tau_s(k,v,v')|v-v'|^s = \max\left\{\left(\sum_{i=1}^{n-1} |v_{i+1}-v_i|^s\right)-|v-v'|^s,0\right\},$$  where $V_{k+1}(v,v')=\{v_1,\dots,v_n\}$ with $v_1=v$, $v_n=v'$, and $v_{i+1}$ is the first point to the right (or left) of $v_i$ for all $1\leq i\leq n-1$.\end{definition}

\begin{lem}\label{lem:monotone} For all $k\geq 0$ and $(v,v')\in\Flat(k)$, we have $\tau_1(k,v,v') \leq 3\a_{k,v}^2$.
\end{lem}

\begin{proof} Let $V_{k+1}(v,v')=\{v_1,\dots,v_n\}$, where $v_1=v$, $v_n=v'$, and $v_{i+1}$ is to the right of $v_i$ for all $1\leq i\leq n-1$. By Lemma \ref{lem:var}, with $s=1$, \begin{equation*}\sum_{i=1}^{n-1}|v_{i+1}-v_i| \leq (1+3\alpha_{k,v}^2)|v_n-v_1|=(1+3\alpha_{k,v}^2)|v-v'|.\end{equation*} Rearranging the inequality gives $\tau_1(k,v,v') \leq 3\alpha_{k,v}^2$.\end{proof}

We now demonstrate that when $s>1$, the variation excess $\tau_s(k,v,v')$ is zero whenever the set $V_{k+1}(v,v')$ lies in a sufficiently thin tube.

\begin{lem}[tube control] \label{lem:flat2} For all $s>1$, there exists $\e_{s,C^*,\xi_1,\xi_2}\in(0,1/16]$ such that if $\alpha_{k,v} \leq \e_{s,C^*,\xi_1,\xi_2}$, then $\tau_s(k,v,v')=0$ for all $(v,v')\in\Flat(k)$.\end{lem}

\begin{proof} Let $(v,v')\in \Flat(k)$ and enumerate $V_{k+1}(v,v')=\{v_1,\dots,v_n\}$ so that $v_1=v$, $v_n=v'$, and $v_{i+1}$ is to the right of $v_i$ for all $1\leq i\leq n-1$. If $n=2$, then $\sum_{i=1}^{n-1}|v_{i+1}-v_i|^s = |v-v'|^s$. Thus, suppose that $n\geq 3$. By Lemma \ref{lem:var}, with $\delta=\rho_{k+1}r_0$ and $\alpha=\alpha_{k,v}$,
\begin{align*}\sum_{i=1}^{n-1}|v_{i+1}-v_i|^s &\leq |v-v'|^s\left [ \left((1+3\alpha_{k,v}^2)-\frac{\rho_{k+1}r_0}{|v-v'|}\right)^s+\left(\frac{\rho_{k+1}r_0}{|v-v'|}\right)^s\right]\\ &\leq |v-v'|^s\left [ \left((1+3\alpha_{k,v}^2)-\frac{\xi_1}{14A^*}\right)^s+\left(\frac{\xi_1}{14A^*}\right)^s\right]=: A_{\alpha_{k,v}}|v-v'|^s,
\end{align*} because $\xi_2^{-1}\rho_{k+1} r_0 \leq |v-v'| \leq 14A^* \xi_1^{-1}\rho_{k+1}r_0$ and the function $$f(t)=((1+3\alpha_{k,v}^2)-t)^s + t^s\quad\text{on } [\xi_1/14A^*,\xi_2]$$ takes its maximum at $t=\xi_1/14A^*$. Since $s>1$, the coefficient $$A_\epsilon \rightarrow \left(1-\frac{\xi_1}{14A^*}\right)^s+\left(\frac{\xi_1}{14A^*}\right)^s<1\quad\text{as $\e\to 0$}.$$ Thus, by continuity, there exists $\epsilon'>0$ such that $A_{\epsilon'}=1$. Let $\epsilon_{s,C^*,\xi_1,\xi_2}=\min\{\epsilon',1/16\}.$ Then $\sum_{i=1}^{n-1}|v_{i+1}-v_i|^s \leq |v-v'|^s$  whenever $\alpha_{k,v}\leq \epsilon_{s,C^*,\xi_1}$.
\end{proof}

\section{Traveling Salesman algorithm}\label{sec:curve}
For the rest of \S\S \ref{sec:curve} and \ref{sec:mass}, let $X$ denote the Hilbert space $l^2(\R)$ or $\R^N$, let $(V_k)_{k\geq 0}$ be a sequence of sets in $X$ and let $\rho_k>0$ be a sequence of numbers satisfying (V0)--(V5) defined in \textsection\ref{sec:flat}. In addition, fix the parameter $\alpha_0\in(0,1/16]$ in Definition \ref{def:flatpairs}. For each integer $k\geq 0$, we will construct
\begin{enumerate}
\item two collections of pairwise disjoint, open intervals in $[0,1]$ denoted by $\mathscr{B}_k$ (called ``bridge intervals") and $\mathscr{E}_k$ (``edge intervals"),
\item two collections  of pairwise disjoint, nondegenerate closed intervals in $[0,1]$ denoted by $\mathscr{F}_k$ (``frozen point intervals") and $\mathscr{N}_k$ (``non-frozen point intervals"), and
\item a continuous map $f_k :[0,1]\to X$
\end{enumerate}
that satisfy the following properties.

\begin{enumerate}
\item[\textbf{(P1)}] The four collections $\mathscr{B}_k$, $\mathscr{E}_k$, $\mathscr{F}_k$, $\mathscr{N}_k$ are mutually disjoint and for any $x\in[0,1]$ there exists unique interval $I$ contained in their union such that $x\in I$.
\item[\textbf{(P2)}] The map $f_k| I$ is affine on each $I\in\mathscr{E}_k\cup\mathscr{B}_k$ and the map $f_k|J$ is constant on each $J \in \mathscr{F}_k\cup\mathscr{N}_k$.
\item[\textbf{(P3)}] For all $I\in \mathscr{E}_k$, we have $ \diam{f_k(I)} <  14A^*\rho_k r_0$.
\item[\textbf{(P4)}] The map $f_k| \bigcup\mathscr{E}_k$ is $2$-to-$1$; that is, for every $x\in \bigcup\mathscr{E}_k$ there exists a unique $x' \in \bigcup\mathscr{E}_k\setminus\{x\}$ such that $f_k(x) = f_k(x')$.
\item[\textbf{(P5)}] If $(v,v')\in \Flat(k)$, then there exists $I \in\mathscr{E}_k$ such that $f_k(I)$ joins $v$ with $v'$. Conversely, if $a$ and $b$ are endpoints of an interval $I\in \mathscr{E}_k$ and $\a_{k,f_{k}(a)} < \a_0$, then $(f_k(a),f_{k}(b))\in\Flat(k)$.
\item[\textbf{(P6)}] For each $I\in \mathscr{F}_k\cup \mathscr{N}_k$, the image $f_k(I) \in V_k$ and for each $v\in V_k$ there exists a unique $I\in\mathscr{N}_k$ such that $f_k(I) = v$.
\item[\textbf{(P7)}] If $J\in \mathscr{N}_k$ is such that $f_k(J)$ is an endpoint of $f_k(I)$ for some $I\in\mathscr{E}_k$, then there exists $I'\in \mathscr{E}_k$ (possibly $I'=I$) such that $f_k(I') = f_k(I)$ and $J\cap\overline{I'} \neq \emptyset$.
\end{enumerate}

\begin{lem}\label{lem:P5} Assume (P5) holds at stage $k\geq 0$. Let $a<b<a'<b'$ be such that $(a,b)$ and $(a',b')$ belong to $\mathscr{E}_{k}$, $f_k(b) = f_k(a')$, and $\a_{k,f_{k}(b)} \leq \a_0$. Then either
\begin{enumerate}
\item $f_k((a,b)) = f_k((a',b'))$ or
\item $f_{k}(b)$ lies between $f_k(a)$ and $f_k(b')$ with respect to order induced by $\ell_{k,f_k(b)}$.
\end{enumerate}
\end{lem}

\begin{proof} This is an immediate consequence of (P5) and Lemma \ref{lem:tridents}.\end{proof}

In \textsection\ref{sec:k=0}, we construct $\mathscr{E}_0$, $\mathscr{B}_0$, $\mathscr{N}_0$, $\mathscr{F}_0$, and $f_0$. In \textsection\ref{sec:hypotheses}, we formulate the inductive hypothesis. We construct the collections of intervals $\mathscr{E}_{k+1}$, $\mathscr{B}_{k+1}$, $\mathscr{N}_{k+1}$, $\mathscr{F}_{k+1}$, and $f_{k+1}$ in \textsection\textsection\ref{sec:bridges}--\ref{sec:noflatpoints}. We verify properties (P1)--(P7) in \textsection\ref{sec:concl}. Finally, in \textsection\ref{sec:choices},  we review choices in all choices of the algorithm.

\subsection{Step $0$}\label{sec:k=0}
Fix a point $v_0 \in V_0$. Let $G_0$ be the (not necessarily connected) graph with vertices $V_0$ and edges $\mathscr{L}_{0}$. Suppose that $G_0$ has components $G_0^{(1)}, \dots, G_0^{(l)}$ with $v_0\in G_0^{(1)}$.

\emph{Case 1.} Suppose that $V_0 = \{v_0\}$. Then set $\mathscr{E}_0 = \emptyset$, $\mathscr{B}_0 = \emptyset$, $\mathscr{N}_{0} = \{[0,1]\}$ and $\mathscr{F}_0 = \emptyset$. Define also $f_0:[0,1]\to X$ with $f_0(x) = v_0$  for all $x\in [0,1]$. Note that properties (P1)--(P7) are trivial in this case.

\emph{Case 2.} Suppose that $\card(V_0) \geq 2$ and that $l=1$, that is, $G_0$ is connected. We apply Proposition \ref{prop:graph} for $v_0$ with $\D=[0,1]$, $G=G_0$ and we obtain a collection of intervals $\mathcal{I}$ and a continuous map $g$. By Lemma \ref{lem:tridents}, each point $v\in V_0$ has valence at most $2$ in $G_0$ and there exists a component $J_v$ of $g^{-1}(v)$ such that if $e$ is an edge of $G_0$ that contains $v$ as an endpoint, then $e$ has a preimage $I\in\mathcal{I}$ such that $\overline{I}\cap J_v \neq \emptyset$. Let $\mathcal{N}$ be the collection of all such intervals $J_v$.

Set $\mathscr{E}_0 = \mathcal{I}$, $\mathscr{B}_0 = \emptyset$, $\mathscr{N}_0 = \mathcal{N}$, define $\mathscr{F}_0$ to be the components of $[0,1]  \setminus\bigcup (\mathscr{E}_0\cup\mathscr{B}_0\cup\mathscr{N}_0)$, and let $f_0 = g$. Properties (P1)--(P7) follow from Proposition \ref{prop:graph}.

\emph{Case 3.} Suppose that $\card(V_0) \geq 2$ and that $l\geq 2$, that is, $G_0$ is disconnected. For each $j=2,\dots,l$ fix a vertex $u_j$ of $G_{0}^{(j)}$. Let $\{ I_1, \dots, I_{2l-2}\}$
be a collection of open intervals, enumerated according to the orientation of $[0,1]$, such that their closures are mutually disjoint and are contained in the interior of $[0,1]$. Let also $\{ J_{1},\dots, J_{2l-1}\}$
be the components of $I \setminus \bigcup_{j=1}^{2l-2}I_j$ enumerated according to the orientation of $[0,1]$. Applying Proposition \ref{prop:graph} for $G = G^{(1)}_0$, $v_0$ and $\D = J_{1}$, we obtain a family of open intervals $\mathcal{I}_1$, a map $g_1 : J_{1} \to G^{(1)}_0 $ and a family $\mathcal{N}_1$ of closed intervals. Similarly, for each $j=2,\dots,l$, applying Proposition \ref{prop:graph} for $G = G^{(j)}_0$, $u_j$ and $\D = J_{2j}$, we obtain a family of open intervals $\mathcal{I}_j$, a map $g_j : J_{2j} \to G^{(j)}_0 $ and a family $\mathcal{N}_j$ of closed intervals. There exists a continuous map $g:[0,1] \to X$ that extends the maps $g_j$ such that
\begin{enumerate}
\item $g(J_{2j+1}) = v_0$ for each $j\in \{1,\dots,l-1\}$;
\item $g|I_{j}$ is affine for each $j\in \{1,\dots,2l-2\}$ and $g(I_{2j-1}) = g(I_{2j}) = [u_j,v_0]$ for each $j\in \{1,\dots,l-1\}$.
\end{enumerate}

Set $\mathscr{E}_{0} = \bigcup_{j=1}^{l}\mathcal{I}_j$, $\mathscr{B}_{0} = \{I_{1},\dots, I_{2l-2}\}$, $\mathscr{N}_{0} = \bigcup_{j=1}^{l}\mathcal{N}_j$, define $\mathscr{F}_{0}$ to be the components of $[0,1]  \setminus\bigcup (\mathscr{E}_0\cup\mathscr{B}_0\cup\mathscr{N}_0)$, and let $f_{0}|[0,1] = g$. Properties (P1)--(P7) follow from Proposition \ref{prop:graph}.

\subsection{Inductive hypothesis}\label{sec:hypotheses}
Suppose that for some $k\geq 0$ we have defined collections $\mathscr{B}_k$, $\mathscr{E}_k$ of open intervals in $[0,1]$, collections $\mathscr{F}_k$, $\mathscr{N}_k$ of nondegenerate closed intervals in $[0,1]$, and a continuous map $f_k :[0,1]\to X$, which satisfy properties (P1)--(P7).

We will define a new map $f_{k+1}:[0,1]\rightarrow X$ and new collections $\mathscr{B}_{k+1} , \mathscr{E}_{k+1}, \mathscr{F}_{k+1}, \mathscr{N}_{k+1}$,
\begin{align*}
\mathscr{B}_{k+1} &= \bigcup_{I\in\mathscr{B}_k\cup\mathscr{E}_k\cup\mathscr{F}_k\cup\mathscr{N}_k}\mathscr{B}_{k+1}(I), \quad &\mathscr{E}_{k+1} = \bigcup_{I\in\mathscr{B}_k\cup\mathscr{E}_k\cup\mathscr{F}_k\cup\mathscr{N}_k}\mathscr{E}_{k+1}(I),\\
\mathscr{F}_{k+1} &= \bigcup_{I\in\mathscr{B}_k\cup\mathscr{E}_k\cup\mathscr{F}_k\cup\mathscr{N}_k}\mathscr{F}_{k+1}(I), \quad  &\mathscr{N}_{k+1} = \bigcup_{I\in\mathscr{B}_k\cup\mathscr{E}_k\cup\mathscr{F}_k\cup\mathscr{N}_k}\mathscr{N}_{k+1}(I),
\end{align*}
where $\mathscr{B}_{k+1}(I)$, $\mathscr{E}_{k+1}(I)$, $\mathscr{F}_{k+1}(I)$, $\mathscr{N}_{k+1}(I)$ are collections of intervals in $I$ that we define below. In particular:
\begin{itemize}
\item In \textsection\ref{sec:bridges}, we define the four collections and $f_{k+1}|I$ for $I\in\mathscr{B}_k$.
\item In \textsection\ref{sec:flatedges} and \textsection\ref{sec:noflatedges}, we define  the four collections and $f_{k+1}|I$ for $I\in\mathscr{E}_k$.
\item In \textsection\ref{sec:fixed}, we define  the four collections and $f_{k+1}|I$ for $I\in\mathscr{F}_k$.
\item In \textsection\ref{sec:flatpoints} and \textsection\ref{sec:noflatpoints}, we the four collections and $f_{k+1}|I$ for $I\in\mathscr{N}_k$.
\end{itemize}

\subsection{Step $k+1$: intervals in $\mathscr{B}_k$}\label{sec:bridges}
For any $I\in \mathscr{B}_k$ we set $\mathscr{B}_{k+1}(I) = \{I\}$, $\mathscr{E}_{k+1}(I) = \emptyset$, $\mathscr{F}_{k+1}(I) = \emptyset$, $\mathscr{N}_{k+1}(I) = \emptyset$, and we define $f_{k+1}|I = f_k|I$. In other words, bridge intervals are frozen and we make no changes on them.

\subsection{Step $k+1$: intervals in $\mathscr{E}_k$ with a at least one endpoint with flat image}\label{sec:flatedges}
Here we consider those intervals $I = (a_I,b_I) \in\mathscr{E}_{k}$  such that one of the $\a_{k,f_{k}(a_I)}$, $\a_{k,f_{k}(b_I)}$ is less than $\a_0$. If no such interval exists, we move to \textsection\ref{sec:noflatedges}. Assume now that such intervals exist. By (P4) and the induction step, such intervals come in pairs $\{I,I'\}$ where $f_k(I) = f_k(I')$ and $f_k(I)\cap f_k(J) = \emptyset$ for all $J\in\mathscr{E}_k\setminus\{I,I'\}$. Fix now such a pair $\{I,I'\}$. We choose one of the two intervals $I,I'$ to start with, say $I$.

Without loss of generality, assume that $\a_{k,f_k(a_{I})} < \a_0$. Let $\ell$ be the approximating line for $(k,f_k(a_I))$, oriented so that $f_k(a_{I})$ lies to the left of $f_k(b_{I})$. Let $V_{k+1,I}$ denote the points in $V_{k+1}\cap B(f_k(a_{I}),14A^* \rho_k r_0)$ that lie between $f_k(a_{I})$ and $f_k(b_{I})$ with respect to $\ell$,  including $f_k(a_{I})$ and $f_k(b_{I})$. Enumerate $V_{k+1,I}$ from left to right,
\begin{equation*} V_{k+1,I} = \{v_1,\dots, v_l\}.\end{equation*} That is, $v_{i}$ lies to the left of $v_{i+1}$ for all $i\in\{1,\dots,l-1\}$, $v_1 = f_k(a_{I})$, and $v_l = f_k(b_{I})$.

\begin{rem}\label{rem:lleq6} By Lemma \ref{rem:card}, we have that $l<2+2.2C^*$.\end{rem}

Let $\{I_{1}, \dots, I_{l-1}\}$ be a collection of open intervals in $I$ with mutually disjoint closures, enumerated according to the orientation of $[0,1]$ so that the left endpoint of $I_1$ coincides with $a_{I}$ and the right endpoint of $I_{l-1}$ coincides with $b_{I}$. Let $\mathscr{N}_{k+1}(I)$ be the components of $[0,1]\setminus \bigcup_{i=1}^{l-1}I_i$ and $\mathscr{F}_{k+1}(I) = \emptyset$. Define
\begin{align*}
\mathscr{E}_{k+1}(I) &= \{I_i : |v_i-v_{i+1}| <14A^* \rho_{k+1}r_0\}, \\
\mathscr{B}_{k+1}(I) &= \{I_i : |v_i-v_{i+1}| \geq 14A^* \rho_{k+1}r_0\}.
\end{align*} Then define  $f_{k+1}|\overline{I}$ continuously so that
\begin{enumerate}
\item $f_{k+1}$ is affine on each $J \in \mathscr{E}_{k+1}(I) \cup \mathscr{B}_{k+1}(I)$ and constant on each $J \in \mathscr{N}_{k+1}(I)\cup\mathscr{F}_{k+1}(I)$;
\item for each $j=1,\dots,l-1$, $f_{k+1}(\overline{I_j}) = [v_{j},v_{j+1}]$ mapping the left endpoint of $I_{j}$ onto $v_{j}$ and the right endpoint of $I_{j}$ onto $v_{j+1}$.
\end{enumerate}
See Figure \ref{fig:flatedges} for the image of $f_{k+1}(I)$.

Once we have defined the four families and $f_{k+1}$ for $I$, we work as follows for $I'$. First note that $V_{k+1,I}=V_{k+1,I'}$. Define $\psi_{I',I} : I' \to I$ to be the unique orientation-reversing linear map between $I'$ and $I$. Define
\begin{equation*} \mathscr{E}_{k+1}(I') = \{ \psi_{I',I}(J) : J \in \mathscr{E}_{k+1}(I)\} \quad\text{and}\quad \mathscr{B}_{k+1}(I') = \{ \psi_{I',I}(J) : J \in \mathscr{B}_{k+1}(I)\}\end{equation*}
This time, however, we set $\mathscr{F}_{k+1}(I')$ to be the components of $I' \setminus\bigcup_{i=1}^{l-1}I_i'$ and $\mathscr{N}_{k+1}(I')  = \emptyset$. Define also $f_{k+1}|\overline{I'}$ continuously so that $f_{k+1}|I' = (f_{k+1}|I) \circ \psi_{I',I}$.
\begin{figure}
  \includegraphics[width=.8\textwidth]{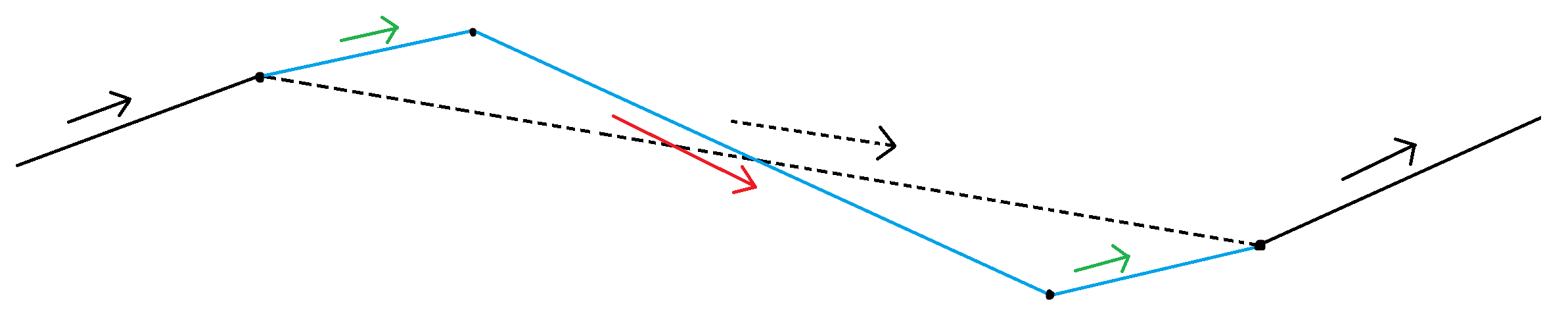}
  \caption{The image of $f_{k+1}(I)$: black arrows denote images of intervals in $\mathscr{E}_{k}\cup\mathscr{B}_k$, green arrows denote images of intervals in $\mathscr{E}_{k+1}$, and red arrows denote images of intervals in $\mathscr{B}_{k+1}$.}
  \label{fig:flatedges}
\end{figure}

\begin{lem}\label{lem:2to1(1)}
For $i=1,2$, let $I_i\in\mathscr{E}_k$ be an interval with at least endpoint having flat image and let $I_i'\in \mathscr{E}_{k+1}(I_i)$. If $f_k(I_1)\neq f_k(I_2)$, then $f_{k+1}(I_1')\cap f_{k+1}(I_2') = \emptyset$.
\end{lem}

\begin{proof} Suppose that $I_1' \in \mathscr{E}_{k+1}(I_1)$, $I_2' \in \mathscr{E}_{k+1}(I_2) $, and $f_{k+1}(I_1')\cap f_{k+1}(I_2') \neq \emptyset$. Because $f_{k+1}(I_1')$ and $f_{k+1}(I_2')$ do not include their endpoints (since intervals in $\mathscr{E}_{k+1}$ are open), we conclude that $f_{k+1}(I_1')=f_{k+1}(I_2')$ by Lemma \ref{lem:tridents}. Now, the endpoints of the four intervals $f_{k}(I_1)$, $f_k(I_2)$, $f_{k+1}(I_1')$, and $f_{k+1}(I_2')$ lie in the $30A^*\rho_k r_0$ neighborhood of any flat endpoint of $f_{k}(I_1)$ or $f_{k}(I_2)$. In particular, the endpoints of the four intervals are linearly ordered by Lemma \ref{lem:approx}, and the endpoints of $f_{k+1}(I_i')$ lie between the endpoints of $f_{k}(I_i)$ by the construction of $\mathscr{E}_{k+1}(I_i)$. Because $f_{k+1}(I_1')=f_{k+1}(I_2')$, this forces $f_k(I_1)=f_k(I_2)$.
\end{proof}

\subsection{Step $k+1$: intervals in $\mathscr{E}_k$ with no endpoints with flat image}\label{sec:noflatedges}
Suppose that $I = (a_I,b_I) \in \mathscr{E}_k$ is such that $\a_{k,f_k(a_I)} \geq \a_0$ and $\a_{k,f_k(b_I)} \geq \a_0$. Then set $\mathscr{E}_{k+1}(I) = \emptyset$, $\mathscr{B}_{k+1}(I) = \{I\}$, $\mathscr{N}_{k+1}(I) = \emptyset$, $\mathscr{F}_{k+1}(I) = \emptyset$, and $f_{k+1}|I = f_k|I$. In other words, edge intervals with no endpoints with flat image become bridge intervals and remain bridge intervals for the rest of the construction.

\subsection{Step $k+1$: intervals in $\mathscr{F}_{k}$}\label{sec:fixed}
For any $I\in \mathscr{F}_k$ we set $\mathscr{E}_{k+1}(I) = \emptyset$, $\mathscr{B}_{k+1}(I) = \emptyset$, $\mathscr{F}_{k+1}(I) = \{I\}$, $\mathscr{N}_{k+1}(I) =\emptyset$ and we set $f_{k+1}|I = f_k|I$. In other words, frozen point intervals in $\mathscr{F}_k$ remain frozen for the rest of the construction.

\subsection{Step $k+1$: intervals in $\mathscr{N}_k$ with flat image}\label{sec:flatpoints}
We now consider the intervals $I \in \mathscr{N}_k$ for which $\a_{k,f_{k}(I)} < \a_0$. If no such interval exists we proceed to \textsection\ref{sec:noflatpoints}. Assume now that such intervals exist.
Let $I$ be such an interval and let $\ell$ be the approximating line for $(k,f_k(I))$. We consider three cases.

\subsubsection{Non-terminal vertices}\label{sec:nonterm} Suppose that there exist distinct $v,v' \in V_k\setminus \{f_k(I)\}$ such that $(f_k(I),v)$ and $(f_k(I),v')$ are in $\Flat(k)$.
By (P5), Lemma \ref{lem:P5} and the induction step, there exist $J, J' \in \mathcal{E}_k$ such that $f_k(I)$ and $v$ are the endpoints of $f_k(J)$, while $f_k(I)$ and $v'$ are the endpoints of $f_k(J')$. Hence, all points of $V_{k+1}$ between $v$ and $v'$ are contained in the image of $f_{k+1}(J \cup J')$ defined in \textsection\ref{sec:flatedges}. Set $\mathscr{E}_{k+1}(I) = \emptyset$, $\mathscr{B}_{k+1}(I) = \emptyset$, $\mathscr{N}_{k+1}(I) = \{I\}$, $\mathscr{F}_{k+1}(I) = \emptyset$ and $f_{k+1}|I = f_k|I$.

\subsubsection{1-sided terminal vertices}\label{sec:1-term}  Suppose that there exists unique $v\in V_{k}\setminus \{f_{k}(I)\}$ such that $(f_k(I),v) \in \Flat(k)$. Fix an orientation for $\ell$ so that $v$ lies to the left of $f_k(I)$. As in \textsection\ref{sec:nonterm}, the points of $V_{k+1}$ that lie between $f_k(I)$ and $v$ are all contained in $f_{k+1}(J)$ for some $J\in \mathscr{E}_k$. Let $V_{k+1,I}$ denote the set that includes $f_k(I)$ and all points in $V_{k+1}\cap B(f_k(I),C^* \rho_{k+1} r_0)$ that lie to the right of $f_k(I)$. Enumerate $V_{k+1,I}= \{v_1,\dots,v_l\}$ from left to right. That is, $v_1 = f_k(I)$ and $v_l$ is the rightmost point of $V_{k+1,I}$

\begin{rem}\label{rem:lleq3}$\card{V_{k+1,I}}\leq 1+1.1C^*$ by Lemma \ref{rem:card}.\end{rem}

If $V_{k+1,I} = \{f_k(I)\}$, then set $\mathscr{E}_{k+1}(I) = \emptyset$, $\mathscr{B}_{k+1}(I) = \emptyset$, $\mathscr{N}_{k+1}(I) = \{I\}$, $\mathscr{F}_{k+1}(I) = \emptyset$ and $f_{k+1}|I = f_k|I$.

If $V_{k+1,I} \neq \{f_k(I)\}$, then let $G_{k+1,I}$ be the graph with vertices the points in $V_{k+1,I}$ and edges the segments $\{[v_1,v_2], \dots, [v_{l-1},v_l]\}$.
That is, $G_{k+1,I}$ forms a simple polygonal arc to the right of $f_k(I)$ joining $f_k(I)$ with $v_l$. Let $\mathcal{I}$ and $g$  be the collection and map given by Proposition \ref{prop:graph} for $\D = I$, $G=G_{k+1,I}$ and $v = f_k(I)$. For each $v'\in V_{k+1,I}$ fix a component of $I\setminus\bigcup\mathscr{I}_{k+1}(I)$ that is mapped onto $v'$ and let $\mathcal{N}$ be the collection of these components. Set $\mathscr{E}_{k+1}(I) = \mathcal{I}$, $\mathscr{B}_{k+1}(I) = \emptyset$, $\mathscr{N}_{k+1}(I) = \mathcal{N}$, and define $\mathscr{F}_{k+1}(I)$ to be the set of components of $I\setminus\bigcup(\mathscr{E}_{k+1}(I)\cup\mathscr{B}_{k+1}(I)\cup\mathscr{N}_{k+1}(I))$. Set  $f_{k+1}|I = g$. See the left half of Figure \ref{fig:flatpoints} for the image of $f_{k+1}(I)$.

\subsubsection{2-sided terminal vertices}\label{sec:2-term} Suppose that there exists no point in $V_{k} \setminus\{f_k(I)\}$ such that $(f_k(I),v)\in\Flat(k)$. That is, $V_k\cap B(f_k(I),14A^* \rho_k r_0) = \{f_k(I)\}$. Set
\begin{equation*} V_{k+1,I} = V_{k+1}\cap B(f_k(I),C^* \rho_{k+1} r_0).\end{equation*} Fix an orientation for $\ell$ and enumerate $V_{k+1,I} = \{v_1,\dots,v_l\}$ from left to right.

\begin{rem}\label{rem:lleq5}
$\card(V_{k+1,I})\leq 1+ 2.2C^*$ by Lemma \ref{rem:card}.
\end{rem}

If $V_{k+1,I} =\{f_k(I)\}$, then set $\mathscr{E}_{k+1}(I) = \emptyset$, $\mathscr{B}_{k+1}(I) = \emptyset$, $\mathscr{N}_{k+1}(I) = \{I\}$, $\mathscr{F}_{k+1}(I) = \emptyset$ and $f_{k+1}|I = f_k|I$. If $V_{k+1,I} \neq \{f_k(I)\}$, then let $G_{k+1,I}$ be the graph with vertices the points in $V_{k+1,I}$ and edges the segments $\{[v_1,v_2], \dots, [v_{l-1},v_l]\}$.  The remainder of the construction proceeds in the same way as in \textsection\ref{sec:1-term}. See the right half of Figure \ref{fig:flatpoints} for the image of $f_{k+1}(I)$.
\begin{figure}
\includegraphics[width=.8\textwidth]{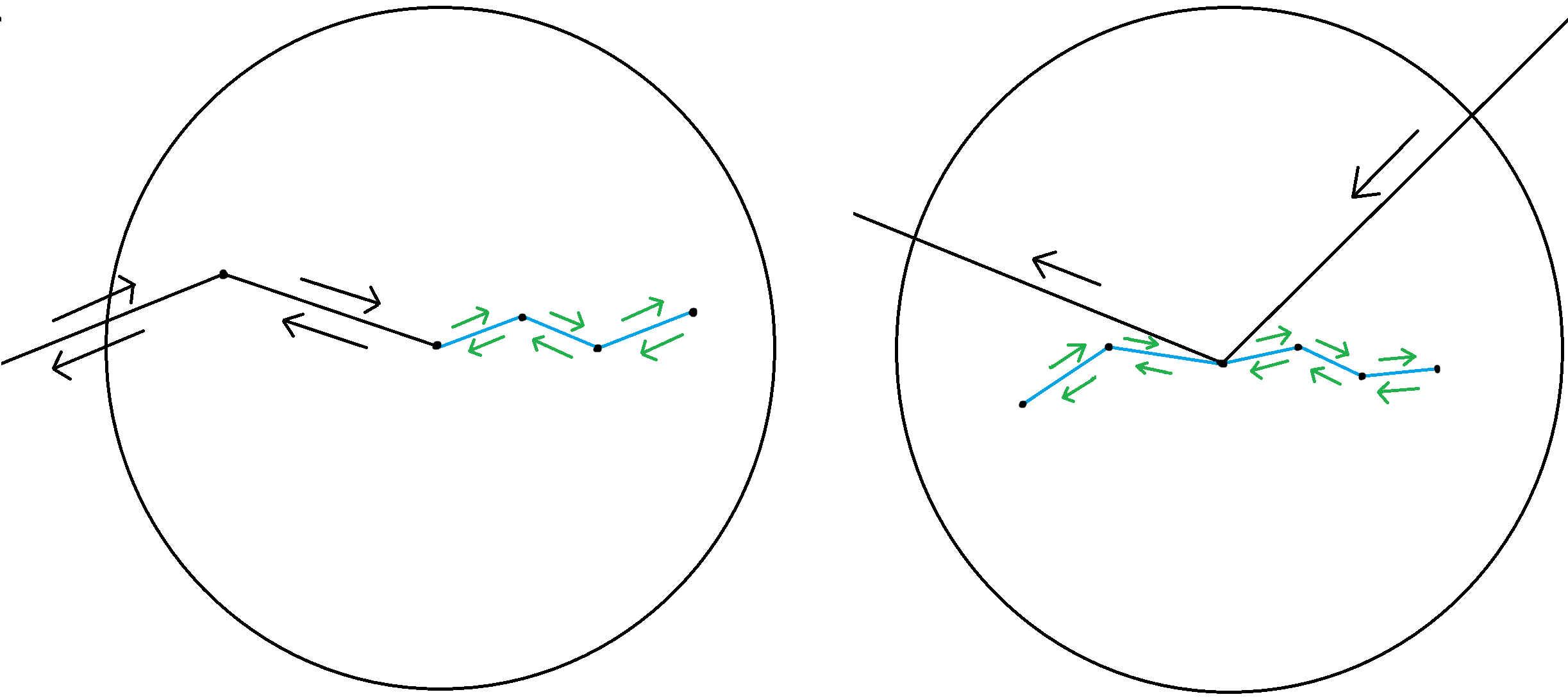}
  \caption{The image of $f_{k+1}(I)$: on the left, we have $f_{k+1}(I)$, where $I$ is as in \textsection\ref{sec:1-term}; on the right, we have $f_{k+1}(I)$, where $I$ is as in \textsection\ref{sec:2-term}.}
  \label{fig:flatpoints}
\end{figure}

\begin{lem}\label{lem:2to1(2)}
Let $I_1,I_2\in \mathscr{N}_k$ be distinct intervals as in \textsection\ref{sec:flatpoints} and let $I_3 \in \mathscr{E}_k$ be an interval as in \textsection\ref{sec:flatedges}. If $I_i' \in \mathscr{E}_{k+1}(I_i)$ for $i=1,2,3$, then the segments $f_{k+1}(I_1')$, $f_{k+1}(I_2')$, and $f_{k+1}(I_3')$ are mutually disjoint.
\end{lem}

\begin{proof} This follows from similar arguments employed in the proof of Lemma \ref{lem:2to1(1)}.
\end{proof}

\subsection{Step $k+1$: intervals in $\mathscr{N}_k$ with non-flat image}\label{sec:noflatpoints}
In this final part of the algorithm, we define $\mathscr{E}_{k+1}(I)$, $\mathscr{B}_{k+1}(I)$, $\mathscr{N}_{k+1}(I)$, $\mathscr{F}_{k+1}(I)$ and $f_{k+1}|I$ for those $I\in\mathscr{N}_{k}$ such that $\a_{k,f_k(I)} \geq \a_0$. Let $\{I_1,\dots,I_n\} $ be an enumeration of such intervals. The construction in this case resembles that in Step $0$.

We start with $I_1$. Let $V_{k+1,I_1}$ be the set of points in $V_{k+1}\cap B(f_k(I_1),C^* \rho_{k+1} r_0)$ that are not images of some $I\in \mathscr{N}_{k+1}$ defined in \textsection\ref{sec:flatedges} and \textsection\ref{sec:flatpoints}. Let $\mathscr{L}_{k+1,I_1}$ be the set of edges in $\mathscr{L}_{k+1}$ that have an endpoint in $V_{k+1,I_1}$. Then define $\tilde{V}_{k+1,I_1}$ to be the union of $V_{k+1,I}$ and the set of all endpoints of edges in $\mathscr{L}_{k+1,I_1}$. By the triangle inequality, the set $\tilde{V}_{k+1,I_1}$ is subset of $B(v,15A^*\rho_{k+1} r_0)$. Finally, let $G_{k+1,I_1}$ denote the graph with vertices $\tilde{V}_{k+1,I_1}$ and with edges $\mathscr{L}_{k+1,I_1}$. We note that the graph $G_{k+1,I_1}$ may be connected or disconnected.

If $\tilde V_{k+1,I_1} = \{f_k(I_1)\}$, then we simply set $\mathscr{E}_{k+1}(I_1) = \emptyset$, $\mathscr{B}_{k+1}(I_1) = \emptyset$, $\mathscr{N}_{k+1}(I_1) = \{I_1\}$, $\mathscr{F}_{k+1}(I_1) = \emptyset$ and $f_{k+1}|I_1 = f_k|I_1$.

For the remainder of \S\ref{sec:noflatpoints}, let us assume that $\tilde V_{k+1,I_1}$ contains at least two points. Let $G_{k+1,I_1}^{(1)},\dots,G_{k+1,I_1}^{(l_1)}$ denote
the connected components of $G_{k+1,I_1}$, labeled so that $G_{k+1,I_1}^{(1)}$ is the component containing $f_k(I_1)$. There are two cases.

\subsubsection{Connected graph}\label{sec:conngraph} Suppose that $G_{k+1,I_1}$ is connected. Apply Proposition \ref{prop:graph} for $\D=I_1$, $G=G_{k+1,I_1}$ and $v=f_k(I_1)$ to obtain a collection of intervals $\mathcal{I}$ and a continuous map $g$.
If $v' \in V_{k+1,I_1}$, then $v'$ has valence at most 2 by Lemma \ref{lem:tridents}. Hence, by Proposition \ref{prop:graph}, there exists a component $J_{v'}$ of $g^{-1}(v')$ with the following property: \begin{quotation}If $v'$ is the endpoint of some $e \in \mathscr{L}_{k+1,I_1}$, then there exists $I \in \mathcal{I}$ such that $g(I) = e$ and $\overline{I}\cap J_{v'} \neq \emptyset$.\end{quotation} Let $\mathcal{N}$ be the collection of the fixed intervals $J_{v'}$ where $v' \in \tilde{V}_{k+1,I_1}$. Now define a set $\mathcal{E} \subseteq \mathcal{I}$ with two rules:
\begin{enumerate}
\item If $e\in \mathscr{L}_{k+1,I_1}$ has both its endpoints in $V_{k+1,I_1}$ then both components of $g^{-1}(e)$ are in $\mathcal{E}$.
\item If $e\in \mathscr{L}_{k+1,I_1}$ has one endpoint $v' \in V_{k+1,I_1}$ and another in $\tilde{V}_{k+1,I_1}\setminus V_{k+1,I_1}$, then only one component of $g^{-1}(e)$ (one that intersects $J_{v'}$) is in $\mathcal{E}$.
\end{enumerate}

Set $\mathscr{E}_{k+1}(I_1) = \mathcal{E}$, $\mathscr{B}_{k+1}(I_1) = \mathcal{I}\setminus\mathcal{E}$, $\mathscr{N}_{k+1}(I_1) = \mathcal{N}$, and define  $\mathscr{F}_{k+1}(I_1)$ to be the set of components of $I_1 \setminus\bigcup(\mathscr{E}_{k+1}(I_1)\cup\mathscr{B}_{k+1}(I_1)\cup\mathscr{N}_{k+1}(I_1))$ and $f_{k+1}|I_1 = g$.

\subsubsection{Several components} Suppose that $l_1\geq 2$; that is, $G_{k+1,I_1}$ is disconnected.  See Figure \ref{fig:noflatpoints} for the image of $f_{k+1}(I)$. We will add some edges which will make the graph connected and the preimage of these edges will be bridge intervals. To this end, for each $j\in \{2,\dots,l_1\}$ fix some point $v_j \in V_{k+1,I_1}\cap G_{k+1,I_1}^{(j)}$. Let $\{I_{1,1},\dots, I_{1,2l_1-2}\}$ be a collection of open intervals, enumerated according to the orientation of $[0,1]$, such that their closures are mutually disjoint and are contained in the interior of $I_1$. Let also $\{J_{1,1},\dots,J_{1,2l_1-1}\}$ be the components of $I_1 \setminus \bigcup_{j=1}^{2l_1-2}I_{1,j}$ enumerated according to the orientation of $[0,1]$.

Working as in \textsection\ref{sec:conngraph}, we obtain a family $\mathcal{I}_1$ of open intervals in $J_{1,1}$, a subset $\mathcal{E}_1 \subset \mathcal{I}_1$, a family $\mathcal{N}_1$ containing some components of $J_{1,1} \setminus \bigcup\mathcal{I}_1$ and a continuous map $g_1:J_{1,1} \to G_{k+1,I_1}^{(1)}$. Similarly, for each $j\in \{2,\dots,l_1\}$ we obtain a family $\mathcal{I}_{j}$ of open intervals in $J_{1,2(j-1)}$, a subset $\mathcal{E}_j \subset \mathcal{I}_{j}$, a family $\mathcal{N}_j$ containing some components of $J_{1,2(j-1)} \setminus \bigcup\mathcal{I}_j$ and a continuous map $g_j:J_{1,2(j-1)} \to G_{k+1,I_1}^{(j)}$.
There exists a continuous map $g:I_1 \to X$ that extends the maps $g_j$ such that
\begin{enumerate}
\item $g(J_{1,2j+1}) = f_k(I_1)$ for each $j\in \{1,\dots,l_1-1\}$;
\item $g|I_{1,j}$ is affine for all $j\in\{1,\dots,2l_1-1\}$ and $g(I_{1,2j-1}) = g(I_{1,2j}) = [v_j,f_k(I_1)]$ for all $j\in\{1,\dots,l_1-1\}$.
\end{enumerate}

Define edge intervals $\mathscr{E}_{k+1}(I_1) = \bigcup_{j=1}^{l_1}\mathcal{E}_j$ and bridge intervals $\mathscr{B}_{k+1}(I_1) = \bigcup_{j=1}^{2l_1-2}I_{1,j} \cup \bigcup_{j=1}^{l_1}(\mathcal{I}_j\setminus\mathcal{E}_j)$. Set $\mathscr{N}_{k+1}(I_1) = \bigcup_{j\in\{1,2,4,\dots,2l_1\}}\mathcal{N}_j$ and define $\mathscr{F}_{k+1}(I_1)$ to be the set of components of $I_1\setminus \bigcup(\mathscr{E}_{k+1}(I_1)\cup\mathscr{B}_{k+1}(I_1)\cup\mathscr{N}_{k+1}(I_1))$. Also, set $f_{k+1}|I_1 = g$.
\begin{figure}
\includegraphics[width=.5\textwidth]{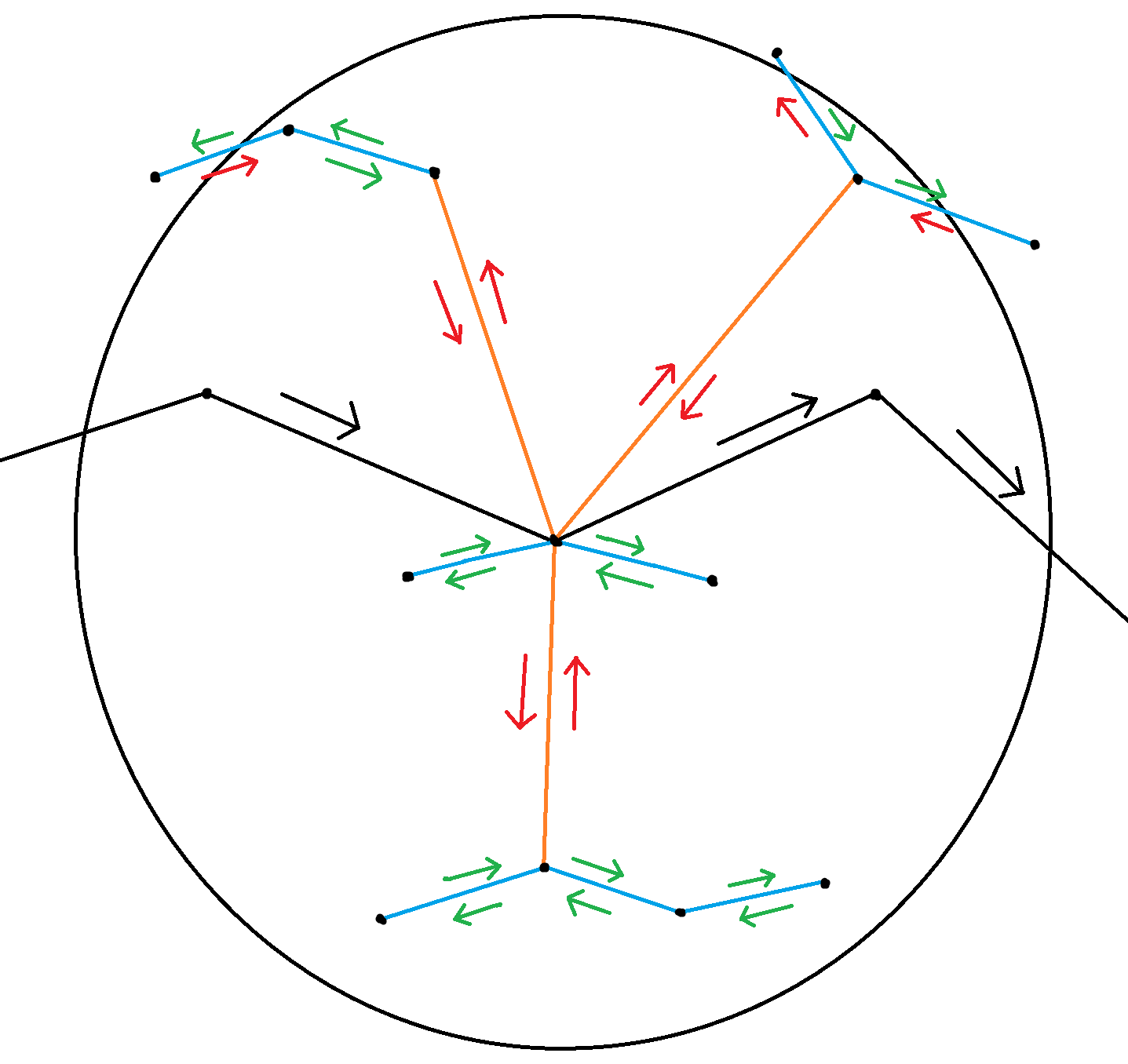}
  \caption{The image of $f_{k+1}(I)$: Blue segments represent edges in $\mathscr{L}_{k+1}$, black arrows represent images of intervals in $\mathscr{E}_k\cup\mathscr{B}_k$, green arrows represent images of intervals in $\mathscr{E}_{k+1}(I)$ and red arrows represent images of intervals in $\mathscr{B}_{k+1}(I)$.}
  \label{fig:noflatpoints}
  \end{figure}

\subsubsection{Inductive hypothesis} Inductively, suppose that for $i\in\{1,\dots,r-1\}$ we have defined $\mathscr{E}_{k+1}(I_i)$, $\mathscr{B}_{k+1}(I_i)$, $\mathscr{F}_{k+1}(I_i)$, $\mathscr{N}_{k+1}(I_i)$ and $f_{k+1}|I_i$. We work now for $I_r$. Let $V_{k+1,I_r}$ be the set of points in $V_{k+1}\cap B(f_k(I_1),C^* \rho_{k+1}r_0)$ that are not images of some $I\in \mathscr{N}_{k+1}$ defined in \textsection\ref{sec:flatedges} or in \textsection\ref{sec:flatpoints} or for the previous intervals $I_1,\dots,I_{r-1}$. Let $\mathscr{L}_{k+1,I_r}$ be the set of edges in $\mathscr{L}_{k+1}$ that have an endpoint in $V_{k+1,I_r}$ and let $\tilde{V}_{k+1,I_r}$ be the set of endpoints of edges in $\mathscr{L}_{k+1,I_r}$. Let now $G_{k+1,I_r}$ be the (not necessarily connected) graph with vertices the set $\tilde{V}_{k+1,I_r}$ and with edges the set $\mathscr{L}_{k+1,I_r}$. To continue, repeat the procedure carried out for $I_1$ \emph{mutatis mutandis}.

\begin{rem}\label{rem:neighbors}
By the choice of set $\mathcal{N}$ for $I_1$, it follows that if $I\in\mathscr{N}_{k+1}(I_1)$ and if $f_{k+1}(I)$ is the endpoint of $f_{k+1}(J)$ for some $J\in\mathscr{E}_{k+1}(I_1)$, then there exists $J'\in \mathscr{E}_{k+1}(I_1)$ (possibly $J'=J$) such that $f_{k+1}(J) = f_{k+1}(J')$ and $\overline{I}\cap\overline{J'} \neq \emptyset$. The same is true for all $I_j$.
\end{rem}

\begin{lem}\label{lem:2to1(3)}
Let $J_1\in \mathscr{N}_k$ be as in \textsection\ref{sec:noflatpoints}, let $J_2\in\mathscr{E}_k$ be as in \textsection\ref{sec:flatedges}, and let $J_3 \in \mathscr{N}_k$ be as in \textsection\ref{sec:flatpoints}. If $J_i' \in \mathscr{E}_{k+1}(J_i)$ for $i=1,2,3$, then the segments $f_{k+1}(J_1')$, $f_{k+1}(J_2')$, $f_{k+1}(J_3')$ are mutually disjoint.
\end{lem}

\begin{proof} By Lemma \ref{lem:2to1(2)} we know that $f_{k+1}(J_2')$ and $f_{k+1}(J_3')$ are disjoint. Fix an interval $J_1\in \mathscr{N}_k$ as in  \textsection\ref{sec:noflatpoints}. Suppose that either $J_2 \in \mathscr{E}_k$ is as in  \textsection\ref{sec:flatedges} or $J_2 \in \mathscr{N}_k$ is as in \textsection\ref{sec:flatpoints}. Let $J_1' \in \mathscr{E}_{k+1}(J_1)$ and $J_2' \in \mathscr{E}_{k+1}(J_2')$. By Lemma \ref{lem:tridents}, either $f_{k+1}(J_1')\cap f_{k+1}(J_2') \neq \emptyset$ or $f_{k+1}(J_1') = f_{k+1}(J_2')$. However, we have defined $\mathscr{L}_{k+1,J_1}$ as those elements in $\mathscr{L}_{k+1}$ that are not contained in $f_{k+1}(J)$, where $J \in \mathscr{E}_k$ is as in  \textsection\ref{sec:flatedges} or $J \in \mathscr{N}_k$ is as in \textsection\ref{sec:flatpoints}. Thus, $f_{k+1}(J_1') \cap f_{k+1}(J_2') = \emptyset$.
\end{proof}

\subsection{Properties (P1)--(P7) for Step $k+1$}\label{sec:concl}
We have now defined $\mathscr{B}_{k+1}$, $\mathscr{E}_{k+1}$, $\mathscr{F}_{k+1}$, $\mathscr{N}_{k+1}$ and $f_{k+1}:[0,1]\rightarrow X$.
It remains to prove that $f_{k+1}$ is continuous and that properties (P1)--(P7) are satisfied by the new collections of intervals and $f_{k+1}$.
Properties (P1), (P2), and (P3) follow immediately from the construction.

\medskip

\textsc{Continuity of $f_{k+1}$.} By design, the map $f_{k+1}$ is continuous on every point interior to an interval in $\mathscr{E}_{k}\cup\mathscr{B}_{k}\cup\mathscr{N}_{k}\cup\mathscr{F}_{k}$. If $x$ is an endpoint of some interval in $\mathscr{E}_{k}\cup\mathscr{B}_{k}\cup\mathscr{N}_{k}\cup\mathscr{F}_{k}$, then $f_{k+1}(x) = f_k(x)$. Thus, continuity of $f_{k+1}$ at $x$ follows from continuity of $f_k$ at $x$.

\medskip

\textsc{Property (P6).} The first claim of (P6), that $f_{k+1}(I) \in V_{k+1}$ for all $I \in \mathscr{N}_{k+1}\cup\mathscr{F}_{k+1}$, is immediate from the construction. To check the second claim of (P6), fix $v\in V_{k+1}$. By (V4), there exists $v'\in V_k$ such that $|v-v'|< C^* \rho_{k+1} r_0$. By the inductive step, there exists $I\in \mathscr{N}_k$ such that $f_k(I) = v'$. There are two cases.

\emph{Case 1.} Suppose that $\a_{k,v'} < \a_0$. Then, following the discussion in \textsection\ref{sec:flatpoints}, either $v = f_{k+1}(J)$ for some $J \in \mathscr{N}_k(J')$ and $J'\in\mathscr{E}_k$ as in \textsection\ref{sec:flatedges}, or $v = f_{k+1}(J)$ for some $J \in \mathscr{N}_k(I)$.

\emph{Case 2.} Suppose that $\a_{k,v'} \geq \a_0$. Following the construction of the graph $G_{k+1,I}$ and the design of the map $f_{k+1}| I$, if $v$ is not the image of some $J \in \mathscr{N}_{k+1}(J')$, where $J' \in \mathscr{E}_k\cup\mathscr{B}_k\cup\mathscr{N}_k \setminus\{J'\}$, then there exists $J\in \mathscr{N}_{k+1}(I)$ such that $v = f_{k+1}(J)$.

\medskip

\textsc{Property (P4).} Fix $I\in\mathscr{E}_{k+1}$. There are three cases.

\emph{Case 1.} Suppose that $I\in \mathscr{E}_{k+1}(I_0)$ for some $I_0 \in \mathscr{E}_k$. By the inductive hypothesis, there exists unique $I_0' \in \mathscr{E}_k\setminus\{I_0\}$ such that $f_k(I_0')= f_k(I_0)$ while $f_k(I_0)\cap f_{k}(J)= \emptyset$ for all $J \in\mathscr{E}_k\setminus\{I_0,I_0'\}$. By construction, there exists $I'\in \mathscr{E}_{k+1}(I_0')$ such that $f_{k+1}(I) = f_{k+1}(I')$. Again by construction, $f_{k+1}(I)\cap f_{k+1}(J) = \emptyset$ for all $J \in \mathscr{E}_{k+1}(I_0)\cup\mathscr{E}_{k+1}(I_0') \setminus\{I,I'\}$. By Lemma \ref{lem:2to1(1)}, Lemma \ref{lem:2to1(2)} and Lemma \ref{lem:2to1(3)}, $f_{k+1}(I)$ does not intersect any $f_{k+1}(J)$ for any $J\in \mathscr{E}_{k+1}(J')$ and $J'\in\mathscr{E}_{k}\cup\mathscr{N}_{k}\setminus \{I_0,I_0'\}$.

\emph{Case 2.} Suppose that $I\in \mathscr{E}_{k+1}(I_0)$ for some $I_0 \in \mathscr{N}_k$ as in \textsection\ref{sec:flatpoints}. By construction, there exists $I' \in \mathscr{E}_{k+1}(I_0)\setminus\{I\}$ such that $f_{k+1}(I) = f_{k+1}(I')$ while $f_{k+1}(I)\cap f_{k+1}(J) = \emptyset$ for all $J\in\mathscr{E}_{k+1}(I_0)\setminus\{I,I'\}$. Moreover, by Lemma \ref{lem:2to1(2)} and Lemma \ref{lem:2to1(3)}, $f_{k+1}(I)$ does not intersect any $f_{k+1}(J)$ for any $J\in \mathscr{E}_{k+1}(J')$ and $J'\in\mathscr{E}_{k}\cup\mathscr{N}_{k}\setminus \{I_0\}$.

\emph{Case 3.} Suppose that $I \in \mathscr{E}_{k+1}(I_0)$ for some $I_0 \in \mathscr{N}_k$ as in \textsection\ref{sec:noflatpoints}. By Lemma \ref{lem:2to1(3)}, $f_{k+1}(I) \cap f_{k+1}(I') = \emptyset$ for all $I' \in \mathscr{E}_{k+1}(J)$ and all $J\in \mathscr{E}_k$ as in \textsection\ref{sec:flatedges} or $J\in \mathscr{N}_k$ as in \textsection\ref{sec:flatpoints}. By the construction of $f_{k+1}|I_0$, there are two possibilities.

\emph{Case 3a.} Suppose that both endpoints of $f_{k+1}(I_0)$ are in $V_{k_1,I_0}$. Then there exists an interval $I' \in \mathscr{E}_{k+1}(I_0)\setminus\{I\}$ such that $f_{k+1}(I)= f_{k+1}(I')$. On the other hand, $f_{k+1}(I) \not\in \mathscr{L}_{k+1,J}$ for any $J\in\mathscr{N}_{k}\setminus\{I_0\}$. Thus, by Lemma \ref{lem:tridents}, if $J' \in \mathscr{E}_{k+1}(J)$ and $J\in\mathscr{N}_{k}\setminus\{I_0\}$, then $f_{k+1}(J) \cap f_{k+1}(J') = \emptyset$.

\emph{Case 3b.} Suppose that only one endpoint of $f_{k+1}(I_0)$ is in $V_{k_1,I_0}$. In this case, by construction, $f_{k+1}(I)\cap f_{k+1}(J) = \emptyset$ for all $J \in \mathscr{E}_{k+1}(I_0)\setminus\{I\}$. Moreover, there exists unique $I_0' \in \mathscr{N}_k\setminus\{I_0\}$ as in \textsection\ref{sec:noflatpoints} such that $V_{k+1,I_0'}$ contains the other endpoint of $f_{k+1}(I)$. As with $I_0$, there exists unique $I'\in\mathscr{E}_{k+1}(I_0')$ such that $f_{k+1}(I') = f_{k+1}(I)$ while $f_{k+1}(J)\cap f_{k+1}(I) =\emptyset$ for all $J \in \mathscr{E}_{k+1}(I_0')$. Finally, by the construction and Lemma \ref{lem:tridents}, $f_{k+1}(I) \cap f_{k+1}(J) = \emptyset$ for all $J \in \mathscr{E}_{k+1}(J')$ and all $J' \in \mathscr{N}_k \setminus\{I_0,I_0'\}$ as in \textsection\ref{sec:noflatpoints}.

\medskip

\textsc{Property (P5).} To prove the first claim in (P5), fix $(v,v')\in \Flat(k+1)$. Let $v_0$ be the point of $V_k$ closest to $v$ and let $I_0 \in \mathscr{N}_k$ be such that $f_k(I_0) = v_0$. There are four cases.

\emph{Case 1.} Suppose that $\a_{k,v_0} < \a_0$ and $v_0$ is non-terminal (see \textsection\ref{sec:nonterm}). Then either both $v$ and $v'$ lie to the left of $v_0$ (with respect to $\ell-{k,v_0}$) or both lie to the right of $v_0$. In any case, $[v,v']$ is the preimage of some $I\in \mathscr{E}_{k+1}(J)$ under $f_{k+1}$ where $J \in \mathscr{E}_k$ and $f_{k}(J)$ is an edge with endpoint $f_k(I_0)$.

\emph{Case 2.} Suppose that $\a_{k,v_0} < \a_0$ and $v_0$ is 2-sided terminal (see \textsection\ref{sec:2-term}). Then $[v,v']$ is the preimage of some $I\in \mathscr{E}_{k+1}(I_0)$ under $f_{k+1}$.

\emph{Case 3.} Suppose that $\a_{k,v_0} < \a_0$ and $v_0$ is 1-sided terminal (see \textsection\ref{sec:1-term}). Then either both $v$ and $v'$ lie to the left of $v_0$ (with respect to $\ell_{k,v_0}$) or both lie to the right of $v_0$. Depending on their position, we work as in Case 1 or Case 2.

\emph{Case 4.} Suppose that $\a_{k,v_0} \geq \a_0$. By definition of graph $G_{k+1,I}$ in \textsection\ref{sec:noflatpoints}, the segment $[v,v']$ is the image of some $J\in \mathscr{E}_{k+1}$ under $f_{k+1}$.

\smallskip

To prove the second claim of (P5), fix $(a,b) \in\mathscr{E}_{k+1}$ such that one of its endpoints has flat image. Without loss of generality, assume $\alpha_{k+1,f_{k+1}(a)}< \a_0$.

\emph{Case 1.} Suppose that $(a,b) \in \mathscr{E}_{k+1}(I)$ for some $I\in\mathscr{N}_k$ as in \textsection\ref{sec:noflatpoints}. By construction of $f_{k+1}$ on such intervals, $(f_{k+1}(a),f_{k+1}(b)) \in \Flat(k+1)$.

\emph{Case 2.} Suppose that $(a,b) \in \mathscr{E}_{k+1}(I)$ for some $I\in\mathscr{N}_k$ as in \textsection\ref{sec:flatpoints}. By construction of $f_{k+1}$ on such intervals, no point of $V_{k+1}\cap B(f_{k+1}(a),14A^* \rho_{k+1}r_0)$,  $$V_{k+1}\cap B(f_{k+1}(a),14A^* \rho_{k+1}r_0)\subset V_{k+1}\cap B(f_{k}(I),30A^*\rho_k r_0),$$ lies strictly between $f_{k+1}(a)$ and $f_{k+1}(b)$ with respect to $\ell_{k,f_k(I)}$. The same is true with respect to $\ell_{k+1,f_{k+1}(a)}$ by Lemma \ref{lem:approx}. Thus, $(f_{k+1}(a),f_{k+1}(b)) \in \Flat(k+1)$.

\emph{Case 3.} Suppose that $(a,b) \in \mathscr{E}_{k+1}(I)$ for some $I\in\mathscr{E}_k$ as in \textsection\ref{sec:flatedges}. The argument is similar to Case 2

\medskip

\textsc{Property (P7).} To check the final property, fix $J\in \mathscr{N}_{k+1}$ and choose $I\in \mathcal{E}_{k+1}$ such that $f_{k+1}(J)$ is an endpoint of $f_{k+1}(I)$. There are several cases.

\emph{Case 1.} Suppose  $J\in\mathscr{N}_{k+1}(J_0)$ for some $J_0\in\mathscr{N}_k$ as in \textsection\ref{sec:noflatpoints}. Then there exists $I' \in \mathscr{E}_{k+1}(J_0)$ such that $f_{k+1}(I') = f_{k+1}(I)$. By the construction of $\mathscr{N}_{k+1}(J_0)$ in \S\ref{sec:noflatpoints}, $\overline{I'}\cap J \neq \emptyset$.

\emph{Case 2.} Suppose $J\in \mathscr{N}_{k+1}(J_0)$ for some $J_0 \in \mathscr{E}_k$ as in \textsection\ref{sec:flatedges}. By (P4), there exists $I' \in \mathscr{E}_{k+1}(J_0)$ such that $f_{k+1}(I') =f_{k+1}(I)$. The interval $I'$ satisfies $\overline{I'}\cap J \neq \emptyset$.

\emph{Case 3.} Suppose $J\in \mathscr{N}_{k+1}(J_0)$ for some $J' \in \mathscr{N}_k$ as in \textsection\ref{sec:flatpoints}. There are three subcases.

\emph{Case 3a.} Suppose that $f_{k+1}(J) \neq f_{k}(J_0)$. Then by the choice of $\mathscr{N}_{k+1}(J_0)$, there exists $I' \in \mathscr{E}_{k+1}(J_0)$ such that $f_{k+1}(I') = f_{k+1}(I)$ and $\overline{I'}\cap\overline{J} \neq \emptyset$.

\emph{Case 3b.} Suppose that $f_{k+1}(J) = f_{k}(J_0)$ and there exists $\tilde{I}\in\mathscr{E}_{k+1}(J_0)$ such that $f_{k+1}(\tilde{I})=f_{k+1}(I)$. As in Case 3a, the claim follows from the choice of $\mathscr{N}_{k+1}(J_0)$.

\emph{Case 3c.} Suppose that $f_{k+1}(J) = f_{k}(J_0)$ and there exists no $\tilde{I}\in\mathscr{E}_{k+1}(J_0)$ such that $f_{k+1}(\tilde{I})=f_{k+1}(I)$. In this case, $f_k(J_0)$ is the endpoint of $f_k(I_0)$ for some $I_0 \in \mathcal{E}_k$. By the inductive hypothesis and (P4), there exists at least one and at most two intervals $I_0' \in \mathscr{E}_{k}$ such that $f_k(I_0) = f_k(I_0')$ and $\overline{I_0'}\cap \overline{J_0} \neq \emptyset$. On one hand, if there is only one interval $I_0'$, then $J_0$ is as in \textsection\ref{sec:nonterm} and $J = J_0$. Hence there exists $I' \in I_0'$ such that $f_{k+1}(I') = f_{k+1}(I)$ and $\overline{I'} \cap J \neq \emptyset$. On the other hand, if there are two intervals $I_0', I_0''$, then one of them has a closure which intersects $J$, say $I_0'$. Then there exists $I' \in I_0'$ such that $f_{k+1}(I') = f_{k+1}(I)$ and $\overline{I'} \cap J \neq \emptyset$.

\subsection{Choices in the Traveling Salesman algorithm}\label{sec:choices}
In \S\S \ref{sec:flat} and \ref{sec:curve}, we made a series of implicit and explicit choices.
\begin{enumerate}
\item[(C0)] The choice of $\alpha_0\in(0,1/16]$ determines the set $\Flat(k)$ of flat pairs. The constant $14$ in the definition of $\Flat(k)$ is chosen to facilitate the estimates in \S\ref{sec:mass} (see (E2)), but has not been optimized. The constant 30 in the definition of $\alpha_{k,v}$ is chosen to be larger than $(1+3(1/16)^2)\cdot 2 \cdot 14$. For example, see the proof of Lemma \ref{lem:tridents}.
\item[(C1)] If $I, I'\in \mathscr{E}_k$ satisfy $f_k(I) = f_k(I')$ and are as in \textsection\ref{sec:flatedges} (that is, $I$ has at least one endpoint $x$ with $\a_{k,f_{k}(x)} < \a_0$), then either $\mathscr{N}_{k+1}(I)=\mathscr{F}_{k+1}(I') = \emptyset$ or vice-versa.
\item[(C2)] If $I \in \mathscr{N}_k$ is as in \textsection\ref{sec:2-term} (i.e.~ $\a_{k,f_{k}(I)} < \a_0$ and $V_{k}\cap B(f_k(I), 14A^* \rho_kr_0) = \{f_k(I)\}$),  then there may exist up to two different ways to parameterize the graph $G_{k+1,I}$ therein.
\item[(C3)] If $I \in \mathscr{N}_k$ is as in \textsection\ref{sec:noflatpoints} (i.e.~ $\a_{k,f_{k}(I)} \geq \a_0$) and $G_{k+1,I}^{(1)},\dots,G_{k+1,I}^{(l)}$ are the graph components of the graph $G_{k+1,I}$ therein, then
\begin{enumerate}
\item[(C3a)] we choose the order in which we parameterize the graph components and
\item[(C3b)] in each graph component, there exists up to two choices of parameterization.
\end{enumerate}
Similar choices are made in the step $0$.
\item[(C4)] We get to choose the enumeration of intervals $I$ in $\mathscr{N}_k$ such that $\a_{k,f_k(I)} \geq \a_0$.
\end{enumerate}

The algorithm can be made more flexible by permitting four additional choices. Let $\tilde{\a}_0 \in (0,\a_0)$ and $\tilde{A}>14A^*$.

\begin{enumerate}
\item[(C5)] Suppose that $I\in\mathscr{N}_k$.
\begin{itemize}
\item If $\a_{k,f_k(I)} <\tilde{\a}_0$ then we treat $I$ as in \textsection\ref{sec:flatpoints}; i.e., we treat $f_k(I)$ as a flat vertex.
\item If $\a_{k,f_k(I)} \geq \a_0$ then we treat $I$ as in \textsection\ref{sec:noflatpoints}; i.e., we treat $f_k(I)$ as a non-flat vertex.
\item If $\a_{k,f_k(I)} \in [\tilde{\a}_0,\a_0)$ then we can either treat $I$ as in \textsection\ref{sec:flatpoints} or as in \textsection\ref{sec:noflatpoints}.
\end{itemize}
\item[(C6)] Suppose that $v\in V_k$ is chosen to be considered ``flat" by (C5). Let $\ell$ be the approximating line for $(k,v)$ and let $v' \in V_k$ be such that there exists no $v''\in V_k\cap B(v,\tilde{A} \rho_{k} r_0)$ such that $\pi_{\ell}(v'')$ is between $v$ and $v'$.
\begin{itemize}
\item If $|v-v'| < 14A^* \rho_kr_0$, then $(v,v')\in \Flat(k)$.
\item If $|v-v'| \geq \tilde{A} \rho_k r_0$, then $(v,v') \not\in \Flat(k)$.
\item If $|v-v'| \in [14A^* \rho_kr_0, \tilde{A}\rho_k r_0)$, then we are free to choose whether $(v,v')$ is contained in $\Flat(k)$ or not.
\end{itemize}
\item[(C7)] Similarly to (C6), suppose that $I\in \mathscr{E}_k$ is as in \textsection\ref{sec:flatedges}; i.e., $f_k(I)$ has at least one endpoint $x$ whose image is ``flat" by (C5). Let $\{v_1,\dots,v_l\}$ and $\{I_1,\dots,I_{l-1}\}$ be as in \textsection\ref{sec:flatedges}.
\begin{itemize}
\item If $|v_i-v_{i+1}| < 14A^* \rho_{k+1}r_0$, then we set $I_i\in\mathscr{E}_{k+1}$.
\item If $|v_i-v_{i+1}| \geq \tilde{A} \rho_{k+1}r_0$, then we set $I_i\in\mathscr{B}_{k+1}$.
\item If $|v_i-v_{i+1}| \in [14A^* \rho_{k+1} r_0, \tilde{A} \rho_{k+1}r_0)$, then we can choose in each instance whether $I_i\in \mathscr{E}_{k+1}(I)$ or $I_i\in \mathscr{B}_{k+1}(I)$.
\end{itemize}
\item[(C8)] Suppose that $\{I_1,\dots, I_n\}$ are the intervals in $\mathscr{N}_k$ that have a non-flat image. Suppose also that we have defined $f_{k+1}$ on $I_1,\dots,I_{r-1}$ and on intervals in $\mathscr{N}_k$ that have an image chosen to be flat. Let $v \in V_{k+1}$ be a point which is not the image of some $I\in\mathscr{N}_{k+1}(J)$, where $J\in\{I_1,\dots,I_{r-1}\}$ or $J\in\mathscr{N}_k$ is as in \textsection\ref{sec:flatpoints}.
\begin{itemize}
\item If $v \in B(f_k(I_r),14A^* \rho_k r_0)$, then $v\in V_{k+1,I_r}$.
\item If $v \in B(f_k(I_r),\tilde{A} \rho_k r_0)\setminus B(f_k(I_r),14A^* \rho_k r_0)$, then we may choose whether $v\in V_{k+1,I_r}$ or not.
\end{itemize}
\end{enumerate}

Note that the (C5) has subsequent implications on the treatment of intervals $I\in\mathscr{E}_k$. For instance, if both endpoints if $I$ have images chosen to be non-flat, then $I\in\mathscr{B}_{k+1}$; otherwise, we treat $I$ as in \textsection\ref{sec:flatedges}. Similarly, (C6) gives us the set $\mathscr{L}_{k}$ and together with (C8) affects the parametrization near non-flat vertices.

\begin{rem}[coherence] In the original, non-parametric Analyst's Traveling Salesman construction, Jones \cite{Jones-TST} required the coherence property (V2), i.e.~ $V_k\subset V_{k+1}$. The first author and Schul \cite{BS3} established a non-parametric Traveling Salesman construction, which replaced (V2) with the weaker property that for all $v\in V_k$, the set $$v'\in V_{k+1}\cap B(v,C^* \rho_k r_0)$$ is nonempty. This relaxation was crucial for the proof of the main result in \cite{BS3}, which characterized Radon measures in $\R^N$ that are carried by rectifiable curves. We would like to emphasize that in the parametric Traveling Salesman construction described above, we heavily rely on (V2). At this time, we do not know how to build a parameterization under the relaxed condition of \cite{BS3}.\end{rem}

\section{Mass of intervals}\label{sec:mass}
In this section, we use the construction of \textsection\ref{sec:curve}, to assign mass to intervals defined in \textsection\ref{sec:curve}. The total mass $\mathcal{M}_s$  on the domain of the maps fills the role that the Hausdorff measure $\mathcal{H}^1$ of the image plays in the proof of the sufficient half of the Analyst's TST given in \cite{Jones-TST} or \cite{BS3}. The main result of this section is Proposition \ref{prop:mass}, which bounds the total mass of $[0,1]$ by a sum involving the flatness approximation errors $\a_{k,v}$ and variation excess $\tau_{s}(k,v,v')$ defined in \S\ref{sec:flat}.
For each $k\geq 0$, set
\begin{equation*} \mathscr{I}_k := \mathscr{E}_k\cup\mathscr{B}_k\cup\mathscr{N}_k\cup\mathscr{F}_k \qquad\text{and}\qquad \mathscr{I} := \bigcup_{k\geq 0}\mathscr{I}_k.\end{equation*}
For each $I \in \mathscr{I}_k$, set
$\mathscr{I}_{k+1}(I) := \mathscr{E}_{k+1}(I)\cup\mathscr{B}_{k+1}(I)\cup\mathscr{N}_{k+1}(I)\cup\mathscr{F}_{k+1}(I)$.

\begin{rem}
If $I \in \mathscr{I}_k \cap \mathscr{I}_m$ for some $m\neq k$, then $f_k | I = f_m | I$.
\end{rem}

\subsection{Trees over intervals}\label{sec:trees}

Given $k\geq 0$ and $I \in \mathscr{I}_k$, we define a \emph{finite tree $T$ over $(k,I)$} to be a finite subset of $\bigcup_{m\geq 0}(\{m\}\times\mathscr{I}_m)$ satisfying the following three conditions.
\begin{enumerate}
\item The pair $(k,I) \in T$. If $(m,J) \in T$, then $m\geq k$ and $J\subset I$.
\item If $(m,J) \in T$ and there exists $J' \in \mathscr{I}_{m+1}(J)$ such that $(m+1,J')\cap T$, then $\{m+1\}\times\mathscr{I}_{m+1}(J)\subseteq T$.
\item If $(m,J)\in T$ for some $m>k$ and $J\in \mathscr{I}_m(J')$, then $(m-1,J')\in T$.
\end{enumerate}
The first condition says that the root of the tree is $(k,I)$ and its elements are descendants of $(k,I)$. The second condition says that if one child $(m+1,J')$ of $(m,J)$ is in $T$, then every child of $(m,J)$ is in $T$. The third condition says that if $(m,J)$ is in $T$, then all its ancestors up to $(k,I)$ are in $T$.

We extend this notion to the entire domain by defining a \emph{finite tree $T$} over $[0,1]$ to be a set of the form
\begin{equation*} T = \{[0,1]\} \cup \bigcup_{I\in\mathscr{I}_0 } T_I,\end{equation*}
where $T_I$ is a finite tree over $(0,I)$. A finite tree over $[0,1]$ may be thought to belong to step $k=-1$ of the construction.

Let $T$ be a finite tree over $(k,I)$.  The \emph{boundary $\partial T$ of $T$} is defined by
\begin{equation*} \partial T := \{(m,J) \in T :(\{m+1\}\times\mathscr{I}_{m+1}(I)) \cap T = \emptyset\}.\end{equation*}
The \emph{depth $m(T)$ of $T$} is the integer defined by
\begin{equation*} m(T) := \max\{m\geq 0 : (m,J) \in T\} = \max\{k\geq 0 : (m,J) \in \partial T\}. \end{equation*}
If $m(T)\geq 1$, the \emph{parent tree $p(T)$} is defined by
\begin{equation*} p(T) := T \setminus  (\{m(T)\}\times\mathscr{I}_{m(T)})\end{equation*}
Note that $m(p(T))=m(T)-1$.

\begin{rem}\label{rem:chainpartition}
If $T$ is a finite tree over $(k,I)$ and $\partial T = \{(k_1,J_1), \dots, (k_n,J_n)\}$, then the intervals $J_1,\dots,J_n$ partition $I$. That is, for all $x\in I$, there exists a unique $i\in\{1,\dots,n\}$ such that $x\in J_i$.
\end{rem}

\subsection{Mass of intervals}\label{sec:massinterval} For all $s\geq 1$, $k\geq 0$, and intervals $I\in\mathscr{I}_k$, define the \emph{$s$-mass $\mathcal{M}_s(k,I)$ of $(k,I)$} by
\[ \mathcal{M}_s(k,I) := \sup_{T} \sum_{(k',I')\in \partial T}(\diam{f_{k'}(I')})^s \in [0,\infty],\]
where the supremum is taken over all finite trees over $(k,I)$. This notion extends to $[0,1]$ by assigning
\[ \mathcal{M}_s([0,1]) := \sum_{I \in \mathscr{I}_0 } \mathcal{M}_s(0,I) \in [0,\infty].\]

\begin{lem}\label{lem:sum}
Let $k\geq 0$ and $I\in\mathscr{I}_k$.
\begin{enumerate}
\item If $I \in \mathscr{B}_k$, then $\mathcal{M}_s(k,I) = (\diam{f_k}(I))^s$.
\item If $I \in \mathscr{F}_k$, then $\mathcal{M}_s(k,I) = 0$.
\item $\mathcal{M}_s(k,I) \geq \sum_{I'\in \mathscr{I}_{k+1}(I)} \mathcal{M}_s(k+1,I')$.
\item If $I \in \mathscr{I}_k \cap \mathscr{I}_m$ for some $m\geq 0$, then $\mathcal{M}_s(k,I) = \mathcal{M}_s(m,I)$.
\end{enumerate}
\end{lem}

Before proving Lemma \ref{lem:sum}, we make two clarifying remarks. First, it is possible for an interval $I\in\mathscr{N}_k$ to have $\mathcal{M}_s(k,I) >0$ even though $\diam f_k(I)^s=0$. This happens whenever $I\in\mathscr{N}_k$ and $\mathscr{E}_{k+1}(I)\cup\mathscr{B}_{k+1}(I)$ is non-empty. Second, Lemma \ref{lem:sum}(4) implies that given $I \in \mathscr{I}_k$, the mass $\mathcal{M}_s(k,I)$ is defined independently of the step of the construction in which $I$ appears. Nevertheless, we include the step $k$ in definition of the mass to improve exposition of the estimates in \textsection\ref{sec:totalmass} and \textsection\ref{sec:prooflemma}.

\begin{proof}[{Proof of Lemma \ref{lem:sum}}]
For the first claim, note that if $I\in\mathscr{B}_k$ and $m\geq k+1$, then $\mathscr{I}_m(I) = \{I\}$. Therefore, if $T$ is a finite tree over $I$ of depth $m$, then $\partial T = \{(m,I)\}$. Thus, $\mathcal{M}_s(k,I) = \sup_{m\geq k} (\diam{f_m}(I))^s =  (\diam{f_k}(I))^s$.

For the second claim, note that if $I\in\mathscr{F}_k$ and $m\geq k+1$, then $\mathscr{I}_m(I) = \{I\}$ and $f_m(I)$ is a point. Therefore, if $T$ is a finite tree over $I$ of depth $m$, then $\partial T = \{(m,I)\}$. Thus, $\mathcal{M}_s(k,I) = \sup_{m\geq k}(\diam{f_m}(I))^s=0$.

For the third claim, let us first assume that $\mathcal{M}_s(k+1,J) = \infty$ for some $J\in\mathscr{I}_{k+1}(I)$. Fix $M>0$ and find a finite tree $T_J$ over $(k+1,J)$ such that
\[ \sum_{(k',J') \in T_J} (\diam{f_{k'}(J')})^s >M.\]
The collection  $T = T_J\cup \{(k,I)\} \cup (\{k+1\}\times\mathscr{I}_{k+1}(I))$ is a finite tree over $(k,I)$ and $\partial T_J \subset \partial T$. Hence
\[ \mathcal{M}_s(k,I) \geq \sum_{(k',I')\in \partial T}(\diam{f_{k'}(I')})^s \geq  \sum_{(k',J')\in \partial T_J}(\diam{f_{k'}(J')})^s >M.\]
We conclude that $\mathcal{M}_s(k,I)= \infty$.

Alternatively, assume that $\mathcal{M}_s(k+1,J)$ is finite for all $J\in\mathscr{I}_{k+1}(I)$. Fix $\epsilon>0$. For each interval $J\in\mathscr{I}_{k+1}(I)$, let $T_J$ be a finite tree over $(k+1,J)$ such that
\[ \sum_{(k',J')\in \partial T_J} (\diam{f_{k'}(J')})^s > \mathcal{M}_s(k+1,J) - \frac{\e}{\card(\mathscr{I}_{k+1}(I))}.\]
Then the collection $T = \{(k,I)\} \cup\bigcup_{J\in\mathscr{I}_{k+1}(I)} T_J$ is a finite tree over $(k,I)$ with $\partial T = \bigcup_{J\in\mathscr{I}_{k+1}(I)}\partial T_J $. Therefore,
\begin{align*}
\mathcal{M}_s(k,I) \geq \sum_{J\in\mathscr{I}_{k+1}(I)} \sum_{(k',J') \in \partial T_J} (\diam{f_{k'}(J')})^s
> \sum_{J\in\mathscr{I}_{k+1}(I)} \mathcal{M}_s(k+1,J) - \e.
\end{align*} The third claim follows by taking $\epsilon\downarrow 0$.

For the fourth claim, suppose that $I\in\mathscr{I}_k\cap\mathscr{I}_m$ for some $m$ and $k$, say without loss of generality that $m>k$. Because $I \in \mathscr{I}_k\cap\mathscr{I}_m$, we have $\mathscr{I}_n(I) = \{I\}$ for all $k\leq n \leq m$. Thus, iterating the third claim, $\mathcal{M}_s(k,I) \geq \mathcal{M}_s(m,I)$. For the opposite inequality, let $T$ be a finite tree over $(k,I)$. If $(m,I)\not\in T$, then $T = \{(k,I),\dots, (l,I)\}$ for some $k\leq l\leq m-1$ and we define $T'=\{(m,I)\}$. If $(m,I) \in T$ , then $T \supset \{(k,I),\dots,(m-1,I)\}$ and we set $T'=T\setminus \{(k,I),\dots,(m-1,I)\}$ so that $T'$ is a finite tree over $(m,I)$ and $\partial T = \partial T'$. In either case,
\[ \sum_{(k',I') \in \partial T} (\diam{f_{k'}(I')})^s = \sum_{(k',I') \in \partial T'} (\diam{f_{k'}(I')})^s\]
and it follows that $\mathcal{M}_s(k,I) \leq \mathcal{M}_s(m,I)$.
\end{proof}

When $s=1$, the 1-mass is comparable to the Hausdorff measure $\Haus^1$ of the image.

\begin{lem}\label{lem:massequiv}
For each $k\geq 0$ and each $I\in\mathscr{I}_k$,
\begin{equation*}
\mathcal{M}_1(k,I) = \lim_{m\to\infty} \sum_{\substack{J \in \mathscr{I}_m \\ J\subset I}} \diam{f_m(J)} \geq \limsup_{m\to\infty} \mathcal{H}^1(f_m(I)).
\end{equation*} If there exists $n\in\N$ such that $f_m|\bigcup\mathscr{B}_k$ is at most $n$-to-1 for all $m\geq k$, then \begin{equation*} \mathcal{M}_1(k,I) \simeq_n \liminf_{m\to\infty} \mathcal{H}^1(f_m(I))\end{equation*}
\end{lem}

\begin{proof} Fix $k\geq 0$ and $I\in \mathscr{I}_k$.
By definition of the mass and (P2),
\[ \mathcal{M}_1(k,I) \geq \limsup_{m\to\infty}\sum_{\substack{J \in \mathscr{I}_m \\ J\subset I}} \diam{f_m(J)} \geq \limsup_{m\rightarrow\infty} \mathcal{H}^1(f_m(I)). \]

To establish the other direction, let $T$ be a finite tree over $I$ of depth $m\geq k$ and enumerate with $\partial T = \{(k_1,I_1),\dots,(k_n,I_n)\}$, where each $k_i \leq m$. For each $i=1,\dots, n$, let $\mathcal{I}_i$ be the set of all intervals $J \in \mathscr{I}_{m}$ such that $J\subset I_i$. Then
\[ \sum_{\substack{J \in \mathscr{I}_{m} \\ J \subset I}} \diam{f_{m}(J)} = \sum_{i=1}^n\sum_{J\in\mathcal{I}_i} \diam{f_{m}(J)} \geq \sum_{i=1}^n\diam{f_{k_i}(I_i)} = \sum_{(l,J)\in \partial T}\diam{f_l(J)}. \]
Therefore,
\[ \sup_{m\geq k}\sum_{\substack{J \in \mathscr{I}_{m}\\ J\subset I}}\diam{f_{m}(J)} \leq \mathcal{M}_1(k,I) = \sup_{T} \sum_{(l,J)\in \partial T}\diam{f_l(J)} \leq  \liminf_{m\to \infty}\sum_{\substack{J \in \mathscr{I}_{m} \\ J \subset I}} \diam{f_{m}(J)}. \]
This shows that \[ \mathcal{M}_1(k,I) = \lim_{m\to\infty} \sum_{\substack{J \in \mathscr{I}_m \\ J\subset I}} \diam{f_m(J)}.\]

By (P4), the maps $f_m|\bigcup\mathscr{E}_m$ are 2-to-1. In \S\ref{sec:curve}, we did not examine overlaps of images of bridge intervals. By modifying the algorithm, the overlap of images of bridge intervals can be made 2-to-1 (see the proof of Proposition \ref{prop:2-to-1Holder}). Nevertheless, suppose that we know the overlaps of images of bridge intervals is at most $n$-to-1 for some $n\geq 2$. Then \[ \mathcal{M}_1(k,I) = \lim_{m\to \infty}\sum_{\substack{J \in \mathscr{I}_{m} \\ J \subset I}} \diam{f_{m}(J)} \leq \frac{1}{n}\liminf_{m\to\infty} \Haus^1(f_m(I)). \qedhere \]
\end{proof}

While Lemma \ref{lem:massequiv} does not hold when $s>1$, we always have the following comparison between the $s$-mass and the Hausdorff measure $\Haus^s$ of the closure of the points in $\bigcup_{k=0}^\infty V_k$.

\begin{lem}\label{lem:mass-measure}
For all $s\geq 1$, $\mathcal{H}^s\left(\overline{\bigcup_{k=0}^\infty V_k}\right)\lesssim_{s,C^*,\xi_2} \mathcal{M}_s([0,1]).$
\end{lem}

\begin{proof}Fix $\d>0$ and choose $m\in\N$ sufficiently large such that $2C^*\xi_2^{m+1}r_0/(1-\xi_2)\leq \delta$. By (V0), (V2), and (V4), the collection
$\{ B(v,C^*\rho_{m+1}r_0/(1-\xi_2)) : v\in V_m\}$ is a cover of $\overline{\bigcup_{k=0}^{\infty} V_k}$ with elements of diameter at most $2C^*\rho_{m+1}r_0/(1-\xi_2)\leq 2C^*\xi_2^{m+1}r_0\leq \delta$. Let $T$ be the maximal finite tree over $[0,1]$ of depth $m$, i.e.~ $T=\bigcup_{k=0}^m\mathscr{I}_k$. Then
\begin{equation*}\begin{split} \mathcal{H}^s_\delta \left(\overline{\textstyle\bigcup_{k=0}^{\infty} V_k}\right) \leq \sum_{v\in V_m} \left(\frac{2C^*\rho_{m+1}r_0}{1-\xi_2}\right)^s &\leq \left(\frac{2C^*\xi_2}{1-\xi_2}\right)^s\sum_{(m,I)\in\partial T} (\diam{f_m(I)})^s\\ &\lesssim_{s,C^*,\xi_2}\mathcal{M}_s([0,1]). \end{split}\end{equation*} Taking $\delta\downarrow 0$ completes the proof.
\end{proof}

We include Lemma \ref{lem:massequiv} and Lemma \ref{lem:mass-measure} for completeness. We will not use either lemma in any the estimates below.

\subsection{Terminal vertices and phantom mass}\label{sec:phantom}
Let $I \in \mathscr{N}_k$ be an interval such that $\a_{k,f_k(I)} < \a_0$. We classify $f_k(I)$ according to the arrangement of nearby points in $V_{k+1}$.
\begin{itemize}
\item If $I$ is as in \textsection\ref{sec:nonterm}, then $f_k(I)$ is called a \emph{non-terminal vertex in $V_k$}.
\item If $I$ is as in \textsection\ref{sec:1-term}, then $f_k(I)$ is called a \emph{1-sided terminal vertex in $V_k$}.
\item If $I$ is as in \textsection\ref{sec:2-term}, then $f_k(I)$ is called a \emph{2-sided terminal vertex in $V_k$}.
\end{itemize}

Motivated by \cite{Jones-TST} and \cite{BS3},  for each $k\geq 0$ we will define a set $\mathscr{P}_k\subset \{k\}\times\mathscr{N}_k$ and for each $(k,I)\in \mathscr{P}_k$ define a number $p_{k,I}>0$, which we call the \emph{phantom mass} at $(k,I)$. The phantom mass $p_{k,I}$ will let us pay for the length of edges between vertices in $V_{k+1}$ nearby $f_k(I)$ that do not lie between vertices in $V_k$ nearby $f_k(I)$ (i.e.~ the blue edges in Figure \ref{fig:flatpoints}). To start, define an auxiliary parameter $P$ depending only on $s$, $C^*$, and $\xi_2$ by requiring that $\left[P+2(1.1C^*)^s\right]\xi_2^s = P.$ That is, \begin{equation}\label{eq:Pdefinition} P= \frac{2(1.1C^*)^s}{1-\xi_2^s}.\end{equation}
For each $k\geq 0$, define
\[ \mathscr{P}_{k}= \{(k,I) :  \text{$I \in \mathscr{N}_k$, $\a_{k,f_k(I)} < \a_0$ and $f_{k}(I)$ is $1$- or $2$-sided terminal in $V_k$}\}.\] For each $k\geq 0$ and $(k,I)\in\mathscr{P}_k$, assign
\[ p_{k,I} :=
\begin{cases}
2P\rho_k^sr_0^s, &\text{ if }f_k(I)\text{ is 2-sided terminal}\\
P\rho_k^sr_0^s, &\text{ if }f_k(I)\text{ is 1-sided terminal}.
\end{cases}\]

\begin{lem}\label{lem:flatpointsum} Let $I\in\mathscr{N}_k$ be an interval such that $\a_{k,f_k(I)}<\a_0$. If $f_k(I)$ is 1-sided terminal, then $$\sum_{J\in \mathscr{I}_{k+1}(I)} (\diam f_{k+1}(J))^s < 2(1.1C^*)^s\rho_{k+1}^sr_0^s.$$ If $f_k(I)$ is 2-sided terminal, then $$\sum_{J\in \mathscr{I}_{k+1}(I)} (\diam f_{k+1}(J))^s < 4(1.1C^*)^s \rho_{k+1}^sr_0^s.$$
\end{lem}

\begin{proof} Suppose $v=f_k(I)$ is 1-sided terminal and let $\{v_1,\dots,v_n\}$ be an enumeration of the points in $V_{k+1}\cap B(v, C^*\rho_{k+1}r_0)$ starting from $v_1=v$ and moving consecutively towards the terminal direction. Then $$\sum_{J\in \mathscr{I}_{k+1}(I)} (\diam f_{k+1}(J))^s = 2 \sum_{i=1}^n |v_{i+1}-v_i|^s \leq 2(1+3\alpha_{k,v}^2)^s |v_1-v_n|^s < 2(1.1)^s (C^*\rho_{k+1} r_0)^s$$ by Lemma \ref{lem:var}, since $1+3\alpha_{k,v}^2 \leq 1+3(1/16)^2 < 1.1$. The case that $v$ is 2-sided terminal follows from a similar computation.
\end{proof}

\subsection{Special bridge intervals}
Given $k\geq 1$, we define
\[ \mathscr{B}_k^*:= \{(k,I) : \text{$I \in \mathscr{B}_k(J)$, $J\in \mathscr{E}_{k-1}$ is as in \textsection\ref{sec:flatedges}}\}.\]
Recall from \textsection\ref{sec:flatedges} that if $(k,I) \in \mathscr{B}_k^*$, then $14A^*\rho_k r_0\leq \diam{f_k(I)} <14 A^*\rho_{k-1} r_0$.

\begin{lem}\label{lem:bridgelength} If $I\in\mathscr{B}_k^*$, then $$\frac{1}{6}(\diam f_k(I))^s \geq P\rho_{k}^s r_0^s.$$\end{lem}

\begin{proof} Because $\diam f_k(I)\geq 14A^*\rho_kr_0$, it suffices to check that $(14A^*)^s \geq 6P$. Recalling the definition of $A^*$, we find that $$(14A^*)^s = \left(\frac{14C^*}{1-\xi_2}\right)^s \geq \frac{12(1.1C^*)^s}{1-\xi_2^s}=6P.$$ Here $1/(1-\xi_2)^s \geq 1/(1-\xi_2^s)$ because $0<\xi_2<1$ and $s\geq 1$.
\end{proof}

Let $T$ be a finite tree over $[0,1]$, let $m$ be the depth of $T$ and let $0\leq k \leq m$ be an integer.
Define
\[ \mathscr{B}_k^*(T) := \{(k,I) \in \mathscr{B}^*_k : \text{ there exists $(k,J)\in T\cap(\{k\}\times\mathscr{N}_k)$ such that $J\cap \overline{I}\neq \emptyset$} \}.\]
Although the sets $\mathscr{B}_k^*(T)$ are not necessarily subsets of $\partial T$, we show in the next lemma that each element in $\partial T$ generates at most two elements in $\bigcup_{k=1}^m\mathscr{B}_k^*(T)$.

\begin{lem}\label{lem:bridges*2}
Let $T$ be a finite tree over $[0,1]$ with depth $m\geq 1$ and let $k\leq m$.
\begin{enumerate}
\item For each $(k,I) \in \mathscr{B}_{k}^*(T)$, there exists a unique $(l,J) \in \partial T$ such that $I \subset J$. In fact, $J\in \mathscr{E}_l\cup\mathscr{B}_l$.
\item For each $(l,J) \in \partial T$ such that $J\in \mathscr{E}_l\cup\mathscr{B}_l$, there exist at most two distinct $(k,I) \in \bigcup_{i=0}^m\mathscr{B}_{i}^*(T)$ such that $I\subset J$.
\end{enumerate}
\end{lem}

\begin{proof}
The first claim of (1) follows immediately from Remark \ref{rem:chainpartition}. For the second claim, let $(k,I) \in \mathscr{B}_{k}^*(T)$. There are two cases.

\emph{Case 1.} If $(k,I) \in T$, then $\mathscr{I}_{n}(I)  = \{I\}$ and $I\in\mathscr{B}_n$ for all $n > k$, since $I$ is a bridge interval. Therefore, $(l,I) \in \partial T$ for some $l\leq m$ and $I\in\mathscr{B}_l$.

\emph{Case 2.} Suppose now that $(k,I) \not\in T$.
If $J$ was in $\mathscr{N}_l\cup\mathscr{F}_l$, then the closure of $I$ would be contained in the interior of $J$ and $I$ would intersect only intervals $I' \in \mathscr{I}_{l'}$ for which $(l',I') \not\in T$, which is a contradiction.

For (2), fix $(l,J) \in \partial T$ with $J\in \mathscr{E}_l\cup\mathscr{B}_l$ and assume that $\{(k_i,J_i):i=1,2,3\}$ are distinct elements in $\bigcup_{j=0}^m\mathscr{B}_{j}^*(T)$ such that $J_i \subset J$ for $i=1,2,3$. Without loss of generality, assume that $J_2$ is between $J_1$ and $J_3$, in the orientation of $J$. Then there exists $l' \geq l+1$ and $J'\subseteq J$ such that $J' \in\mathscr{N}_{l'}$ and $(l',J') \in T$. But then $(l,J) \not\in \partial T$ which is a contradiction.
\end{proof}

We now impose an additional restriction on $\alpha_0$.

\begin{lem}\label{lem:flatdiam} For all $C^*$, $\xi_1$, and $\xi_2$, there exists $\alpha_1\in(0,1/16]$ such that if $\alpha_0\leq \alpha_1$, then for all $k\geq 0$, $(v,v')\in\Flat(k)$, and $y,y'\in V_{k+1}(v,v')$, we have $|y-y'| \leq |v-v'|$. \end{lem}

\begin{proof} Enumerate $V_{k+1}(v,v')=\{v_1,\dots,v_n\}$ from left to right, so that $v_1=v$ and $v_n=v'$. Let $y=v_l$ and $y'=v_m$ for some $1\leq l<m\leq n$. If $y=v$ and $y'=v'$, the conclusion is trivial. Thus, let us suppose that there exists at least one point to the left of $y$ or the right of $y'$, say without loss of generality that $l\geq 2$. Let $x_i=\pi_{\ell_{k,v}}$ for all $1\leq i\leq n$. Then, arguing as in the proof of Lemma \ref{lem:var},
$$\frac{|v_1-v_2| + |v_l-v_m|}{1+3\alpha_0^2} \leq |x_1-x_2| + |x_l-x_m| \leq |x_1-x_n|\leq |v-v'|.$$ Because $V_{k+1}$ is $\rho_{k+1}$ separated, \begin{equation*}\begin{split}|v_l-v_m| \leq (1+3\alpha_0^2) |v-v'| - |v_1-v_2| &\leq (1+3\alpha_0^2)|v-v'| - \rho_{k+1}r_0\\ &=|v-v'|\left(1+3\alpha_0^2 -\frac{\rho_{k+1}r_0}{|v-v'|}\right).\end{split}\end{equation*} Now $|v-v'| \leq 14A^*\rho_k r_0 \leq 14A^*\xi_1^{-1}\rho_{k+1} r_0$. Hence $|v_l-v_m|\leq |v-v'|$ provided that $$1+3\alpha_0^2 - \frac{\xi_1}{14A^*} \leq 1.$$ Thus, we can take \begin{equation}\label{alphabound} \alpha_1 := \min\left\{\frac{1}{16}, \left( \frac{\xi_1}{42 A^*} \right)^{1/2}\right\}.\qedhere\end{equation}\end{proof}

Together Lemma \ref{lem:bridges*2} and Lemma \ref{lem:flatdiam} yield the following result.

\begin{cor}\label{cor:bridges*} Assume $\alpha_0\leq \alpha_1$.
If $T$ is a finite tree over $[0,1]$ of depth $m$, then,
\[ \sum_{k=1}^{m}\sum_{(k,J) \in \mathscr{B}_{k}^*(T)} (\diam{f_k(J)})^s \leq 2\sum_{(l,J) \in \partial T} (\diam{f_l(J)})^s.\]
\end{cor}

\subsection{An upper bound for the total mass}\label{sec:totalmass}
Here we prove the following proposition which gives an upper bound for the mass of $[0,1]$ in terms of the variation excess $\tau_{s}(k,v,v')$ of flat pairs $(v,v')\in \Flat(k)$ defined in \textsection\ref{sec:flat}.

\begin{prop}\label{prop:mass} Assume that $\alpha_0\leq \alpha_1$ (see Lemma \ref{lem:flatdiam}). For all $s\geq 1$,
\[ \mathcal{M}_s([0,1]) \lesssim_{s,C^*,\xi_2} r_0^s + \sum_{k=0}^{\infty} \sum_{(v,v')\in\Flat(k)} \tau_{s}(k,v,v') \rho_{k}^s r_0^s + \sum_{k=0}^{\infty} \sum_{\substack{v\in V_k \\ \a_{k,v}\geq \a_0}} \rho_k^s r_0^s.\]
\end{prop}

The proof of Proposition \ref{prop:mass} reduces to proving the following lemma (cf. \cite[(9.4)]{BS3}). Recall the definition of the parent tree from \textsection\ref{sec:trees}.

\begin{lem}\label{lem:estimates} Let $k\geq 1$ be an integer, let $T$ be a finite tree over $[0,1]$ of depth $k+1$ and let $p(T)$ be the parent tree. There exists a constant $C>0$ depending only on $s$, $C^*$, and $\xi_2$ such that
\begin{align}\label{eq:estimates}\tag{E}
\begin{split}
&\sum_{(l,I)\in \partial T}(\diam{f_l(I)})^s  + \sum_{(k+1,I)\in \mathscr{P}_{k+1} \cap\partial T}p_{k+1,I} \\
&\quad\leq\sum_{(l,I)\in \partial p(T)}(\diam{f_l(I)})^s + \sum_{(k,I)\in \mathscr{P}_{k}\cap \partial p(T)}p_{k,I} + \frac13\sum_{(k+1,I)\in \mathscr{B}_{k+1}^*(T)}(\diam{f_{k+1}(I)})^s\\
&\quad\ + C\sum_{\substack{v\in V_k \\ \a_{k,v}\geq \a_0}} \rho_{k}^sr_0^s+ C\sum_{\substack{w\in V_{k+1} \\ \a_{k+1,w}\geq \a_0}} \rho_{k+1}^sr_0^s + C\sum_{(v,v')\in\Flat(k)} \tau_{s}(k,v,v') \rho_k^s r_0^s.
 \end{split}
 \end{align}
\end{lem}

We prove Lemma \ref{lem:estimates} in \textsection\ref{sec:prooflemma}. Assuming that Lemma \ref{lem:estimates} holds, here is the proof of Proposition \ref{prop:mass}.

\begin{proof}[{Proof of Proposition \ref{prop:mass}}] Assume that $\alpha_1\leq \alpha_0$.
Let $T$ be a finite tree over $[0,1]$ of depth $m$. Set $T_m = T$ and for each $1\leq k\leq m$ (if any) set $T_{k-1} = p(T_{k})$. Note that $T_{0} = \{0\}\times\mathscr{I}_0$. By Lemma \ref{lem:estimates}, for all $1\leq k\leq m$ (if any),
\begin{align*}
&\sum_{(l,I) \in \partial T_{k}} (\diam{f_l(I)})^s + \sum_{(k,I)\in \mathscr{P}_{k}\cap\partial T_{k}}p_{k,I}  \\
&\quad\leq \sum_{(l,I) \in \partial T_{k-1}} (\diam{f_l(I)})^s + \sum_{(k-1,I)\in \mathscr{P}_{k-1}\cap\partial T_{k-1}}p_{k-1,I}
+ \frac13\sum_{(k,I)\in \mathscr{B}_{k}^*(T_{k})}(\diam{f_k(I)})^s\\
&\quad\ + C\sum_{\substack{v\in V_{k-1} \\ \a_{k-1,v}\geq \a_0}} \rho_{k-1}^s r_0^s+ C\sum_{\substack{w\in V_{k} \\ \a_{k,w}\geq \a_0}} \rho_{k}^s r_0^s + C\sum_{(v,v')\in\Flat(k-1)} \tau_{s}(k-1,v,v') \rho_{k-1}^s r_0^s.
\end{align*}
Iterating the latter inequality,
\begin{align*}
\sum_{(l,I) \in \partial T} (\diam{f_l(I)})^s &\leq \sum_{I \in \mathscr{E}_0\cup \mathscr{B}_0} (\diam{f_0(I)})^s + \sum_{I \in \mathscr{P}_0}p_{0,I}
+ \frac13\sum_{k=1}^{m}\sum_{I\in\mathscr{B}_k^*(T)} (\diam{f_k(I)})^s \\
&\ + 2C\sum_{k=0}^{m}\sum_{\substack{v\in V_k \\ \a_{k,v}\geq \a_0}} \rho_{k}^s r_0^s+ C\sum_{k=1}^{m}\sum_{(v,v')\in\Flat(k)} \tau_{s}(k,v,v') \rho_{k}^s r_0^s.
\end{align*}
Since $\alpha_0\leq \alpha_1$, we obtain $\frac13\sum_{k=1}^{m}\sum_{I\in\mathscr{B}_k^*(T)} (\diam{f_k(I)})^s \leq \frac23\sum_{(l,I) \in \partial T} (\diam{f_l(I)})^s$ by Corollary \ref{cor:bridges*}. This is the only place in the proof of the Proposition that we use the restriction $\alpha_0\leq \alpha_1$.
Therefore,
\begin{align*}
\sum_{(l,I) \in \partial T} (\diam{f_l(I)})^s &\leq 3\sum_{I \in \mathscr{E}_0\cup \mathscr{B}_0} (\diam{f_0(I)})^s + 3\sum_{I \in \mathscr{P}_0}p_{0,I}\\
 &\ + 6C\sum_{k=0}^{m}\sum_{\substack{v\in\ V_k\\ \a_{k,v}\geq \a_0}}\rho_{k}^s r_0^s + 3C\sum_{k=0}^{m-1}\sum_{(v,v')\in \Flat(k) }\tau_{s}(k,v,v') \rho_{k}^s r_0^s.
\end{align*}
There are now two alternatives. On one hand, suppose that $\alpha_{0,v}<\alpha_0$ for some $v\in V_0$. Then $V_{0}$ projects onto an $(1+3(1/16)^2)^{-1}r_0$ separated set in $\ell_{0,v}$ of diameter at most $2C^*r_0$  by Lemma \ref{lem:approx}. Hence $\card{V_0} \lesssim_{C^*} 1$ and $$\sum_{I \in \mathscr{E}_0\cup \mathscr{B}_0} (\diam{f_0(I)})^s + \sum_{I \in \mathscr{P}_0}p_{0,I} \lesssim_{s,C^*,\xi_2}  r_0^s.$$
On the other hand, suppose that $\alpha_{0,v}\geq \alpha_0$ for all $v\in V_0$. Then $$\sum_{I \in \mathscr{E}_0\cup \mathscr{B}_0} (\diam{f_0(I)})^s + \sum_{I \in \mathscr{P}_0}p_{0,I} \lesssim_{s,C^*,\xi_2} \sum_{\substack{v\in V_0\\ \alpha_{0,v}\geq \alpha_0}} r_0^s$$ In either case, we arrive at
\begin{align*}
\sum_{(l,I)\in\partial T} (\diam{f_l(I)})^s &\lesssim_{s,C^*,\xi_2} r_0^s+ \sum_{k=0}^{m-1}\sum_{(v,v')\in \Flat(k) }\tau_{s}(k,v,v') \rho_{k}^s r_0^s + \sum_{k=0}^{m}\sum_{\substack{v\in\ V_k\\ \a_{k,v}\geq \a_0}}\rho_{k}^s r_0^s \\
&\lesssim_{s,C^*,\xi_2} r_0^s+ \sum_{k=0}^{\infty}\sum_{(v,v')\in \Flat(k) }\tau_{s}(k,v,v') \rho_{k}^sr_0^s + \sum_{k=0}^{\infty}\sum_{\substack{v\in\ V_k\\ \a_{k,v}\geq \a_0}}\rho_{k}^s r_0^s.
\end{align*}
Since $T$ was an arbitrary tree over $[0,1]$, we obtain the desired bound on $\mathcal{M}_s([0,1])$.
\end{proof}

\subsection{Proof of Lemma \ref{lem:estimates}}\label{sec:prooflemma}
The proof is divided into five estimates (\ref{eq:est1}), (\ref{eq:est2}), (\ref{eq:est3}), (\ref{eq:est4}) and (\ref{eq:est5}), whose sum gives (\ref{eq:estimates}). Towards this end, we split the left hand side of (\ref{eq:estimates}) into smaller sums by making the following four decompositions.

Firstly, $\partial T$ can be partitioned as $\mathcal{E}_1\cup \mathcal{E}_2\cup \mathcal{E}_3\cup \mathcal{B}_1\cup \mathcal{B}_2\cup \mathcal{B}_3\cup\mathcal{B}_4 \cup \mathcal{F}\cup (\partial T \cap \partial p(T))$, where
\begin{align*}
\mathcal{E}_1 &= \{(k+1,I) \in \partial T : \text{$I \in \mathscr{E}_{k+1}(J)$ and  $J\in\mathscr{N}_k$  is as in \textsection\ref{sec:noflatpoints}}\},\\
\mathcal{E}_2 &= \{(k+1,I) \in \partial T : \text{$I \in \mathscr{E}_{k+1}(J)$ and  $J\in\mathscr{N}_k$  is as in \textsection\ref{sec:flatpoints}}\},\\
\mathcal{E}_3 &= \{(k+1,I) \in \partial T : \text{$I \in \mathscr{E}_{k+1}(J)$ and  $J\in\mathscr{E}_k$  is as in \textsection\ref{sec:flatedges}}\},\\
\mathcal{B}_1 &= \{(k+1,I) \in \partial T : \text{$I \in\mathscr{B}_{k+1}(J)$ and $J \in \mathscr{N}_k$ is as in \textsection\ref{sec:noflatpoints}}\},\\
\mathcal{B}_2 &= \{(k+1,I) \in \partial T : \text{$I \in \mathscr{B}_{k+1}(J)$ and  $J\in\mathscr{E}_k$  is as in \textsection\ref{sec:noflatedges}}\},\\
\mathcal{B}_3 &= \{(k+1,I) \in \partial T : \text{$I \in \mathscr{B}_{k+1}(J)$ and  $J\in\mathscr{E}_k$  is as in \textsection\ref{sec:flatedges}}\},\\
\mathcal{B}_4 &= \{(k+1,I) \in \partial T : \text{$I \in \mathscr{B}_{k+1}(J)$ and  $J\in\mathscr{B}_k$ is as in \textsection\ref{sec:bridges}}\},\\
\mathcal{F} &\subseteq \{k+1\}\times(\mathscr{N}_{k+1}\cup\mathscr{F}_{k+1}).
\end{align*}

Secondly, $\mathscr{P}_{k+1}\cap\partial T$ can be partitioned as $\mathcal{P}_1\cup \mathcal{P}_2\cup \mathcal{P}_3$, where
\begin{align*}
\mathcal{P}_1 &= \{(k+1,I) \in \partial T\cap\mathscr{P}_{k+1}: \text{$I \in \mathscr{N}_{k+1}(J)$ and $J\in \mathscr{N}_k$ is as in \textsection\ref{sec:noflatpoints}}\},   \\
\mathcal{P}_2 &= \{(k+1,I) \in \partial T\cap\mathscr{P}_{k+1}: \text{$I \in \mathscr{N}_{k+1}(J)$ and $J\in \mathscr{N}_k$ is as in \textsection\ref{sec:flatpoints}}\},\\
\mathcal{P}_3 &= \{(k+1,I) \in \partial T\cap\mathscr{P}_{k+1}: \text{$I \in \mathscr{N}_{k+1}(J)$ and $J\in \mathscr{E}_k$ is as in \textsection\ref{sec:flatedges}}\}.
\end{align*}

Thirdly, $ \partial p(T)$ can be partitioned as $\mathcal{E}_1'\cup \mathcal{E}_2' \cup \mathcal{F}' \cup (\partial T \cap \partial p(T))$, where
\begin{align*}
\mathcal{E}_1' &= \{(k,I) \in \partial p(T) \setminus \partial T : \text{$I \in \mathscr{E}_{k}$ is as in \textsection\ref{sec:noflatedges}}\},\\
\mathcal{E}_2' &= \{(k,I) \in \partial p(T) \setminus \partial T : \text{$I \in \mathscr{E}_{k}$ is as in \textsection\ref{sec:flatedges}}\},\\
\mathcal{B}_1' &= \{(k,I) \in \partial p(T) \setminus \partial T : \text{$I \in \mathscr{B}_{k}$ is as in \textsection\ref{sec:bridges}}\},\\
\mathcal{F}' &\subseteq \{k\}\times(\mathscr{N}_{k}\cup\mathscr{F}_{k}).
\end{align*}

Fourthly, set $\mathscr{B}^{\,**}_{k+1}(T)= \mathscr{B}_{k+1}^*(T)\sqcup\mathscr{B}_{k+1}^*(T)$. Then define collections $\mathcal{B}^*_1$ and $\mathcal{B}^*_2$ as follows. If $I\in \mathscr{N}_{k}$, $\alpha_{k,f_k(I)}<\a_0$, $\mathscr{I}_{k+1}(I)\subset\partial T$, and $f_k(I)$ is an endpoint of the image $f_{k+1}(J)$ of $(k+1,J)\in\mathscr{B}^*_{k+1}$, then we include a copy of $(k+1,J)$ from $\mathscr{B}^{\,**}_{k+1}(T)$ in $\mathcal{B}^*_1$. If $K\in\mathscr{N}_k$ and $f_k(K)$ is the endpoint of $f_{k+1}(I)$ for some $(k+1,I)\in\mathcal{B}_3$ that lies strictly between the endpoints of the image $f_k(J)$ of the associated edge interval $J\in\mathscr{E}_{k}$, then we include a copy of $(k+1,I)$ from $\mathscr{B}^{\,**}_{k+1}(T)$ in $\mathcal{B}^*_2$. Because each bridge has only two endpoints, we can choose the included copies so that $\mathcal{B}^*_1\cup\mathcal{B}^*_2\subset \mathscr{B}^{\,**}_{k+1}.$

Before proceeding to the estimates, we remark that
\[ \sum_{(k,I) \in \mathcal{F}'}(\diam{f_k(I)})^s = \sum_{(k+1,I) \in \mathcal{F}}(\diam{f_{k+1}(I)})^s = 0,\] because $f_k$ is constant on each interval in $\mathscr{N}_k\cup\mathscr{F}_k$ by (P2).

\medskip

\underline{Estimate 1.} Here we deal with phantom masses and new intervals coming from some $I\in \mathscr{N}_k$ whose image is not flat. In particular, we will show that there exists $C_1$ depending only on $s$, $C^*$, and $\xi_2$ such that
\begin{equation} \label{eq:est1}\tag{E1}\begin{split}
\sum_{(k+1,I) \in \mathcal{E}_1\cup \mathcal{B}_1}&(\diam{f_{k+1}(I)})^s + \sum_{(k+1,J) \in \mathcal{P}_1}p_{k+1,J}\\ & \leq C_1 \sum_{\substack{v\in V_k \\ \a_{k,v}\geq \a_0}} \rho_{k}^sr_0^s + C_1\sum_{\substack{w\in V_{k+1} \\ \a_{k+1,w}\geq \a_0}} \rho_{k+1}^sr_0^s.
\end{split}\end{equation}
Since
\[ \{I \in \mathscr{N}_k : \a_{k,f_k(I)}\geq \a_0 \text{ and }\mathscr{I}_{k+1}(I)\subset \partial T\} \subseteq \{I \in \mathscr{N}_k\cap \partial p(T) : \a_{k,f_k(I)}\geq \a_0\},\]
inequality (\ref{eq:est1}) follows from the inequality
\begin{equation}\label{eq:est1'}
\begin{split}
&\sum_{\substack{I \in \mathscr{N}_k \\  \{k+1\}\times\mathscr{I}_{k+1}(I) \subset \partial T \\ \a_{k,f_k(I)}\geq \a_0}} \left ( \sum_{J \in \mathscr{E}_{k+1}(I) \cup \mathscr{B}_{k+1}(I)}(\diam{f_{k+1}(I)})^s + \sum_{\substack{J \in \mathscr{N}_{k+1}(I) \\ (k+1,J) \in \mathscr{P}_{k+1}}} p_{k+1,J} \right )\\
&\qquad\leq C_1\sum_{\substack{I \in \mathscr{N}_k \\  \{k+1\}\times\mathscr{I}_{k+1}(I) \subset \partial T \\ \a_{k,f_k(I)}\geq \a_0}} \rho_{k}^sr_0^s
+ C_1\sum_{\substack{I \in \mathscr{N}_k \\  \{k+1\}\times\mathscr{I}_{k+1}(I) \subset \partial T \\ \a_{k,f_k(I)}\geq \a_0}} \sum_{\substack{J\in\mathscr{N}_{k+1}(I)\\ \a_{k+1,f_{k+1}(J)}\geq \a_0}} \rho_{k+1}^sr_0^s.
\end{split}
\end{equation}
To prove (\ref{eq:est1'}), fix any $I \in \mathscr{N}_k$ such that $\{k+1\}\times\mathscr{I}_{k+1}(I) \subset \partial T$ and $ \a_{k,f_k(I)}\geq \a_0$. Recall that $\mathcal{N}_{k+1}(I)$ is in one-to-one correspondence with $V_{k+1,I}$ defined in \S\ref{sec:noflatpoints}. There are now two possibilities.
On one hand, suppose that $\alpha_{k+1,w}<\alpha_0$ for some $w\in V_{k+1,I}$. Then $V_{k+1,I}$ projects onto an $(1+3(1/16)^2)^{-1}\rho_{k+1}r_0$ separated set in $\ell_{k+1,w}$ of diameter at most $2C^*\rho_{k+1}r_0$ by Lemma \ref{lem:approx}. Hence $\card{\mathscr{I}_{k+1}(I)} \lesssim_{C^*} 1$ and
\begin{equation*}\sum_{J \in \mathscr{E}_{k+1}(I) \cup \mathscr{B}_{k+1}(I)}(\diam{f_{k+1}(J)})^s + \sum_{\substack{J \in \mathscr{N}_{k+1}(I) \\ (k+1,J) \in \mathscr{P}_{k+1}}} p_{k+1,J} \lesssim_{s,C^*,\xi_2} \rho_{k+1}^sr_0^s\lesssim_{s,C^*,\xi_2}\rho_{k}^sr_0^s. \end{equation*}
On the other hand, suppose that $\alpha_{k+1,w}\geq \alpha_0$ for all $w\in V_{k+1,I}$. Then \begin{equation*} \sum_{J \in \mathscr{E}_{k+1}(I) \cup \mathscr{B}_{k+1}(I)}(\diam{f_{k+1}(J)})^s + \sum_{\substack{J \in \mathscr{N}_{k+1}(I) \\ (k+1,J) \in \mathscr{P}_{k+1}}} p_{k+1,J} \lesssim_{s,C^*,\xi_2} \sum_{J\in\mathscr{N}_{k+1}(I)} \rho_{k+1}^sr_0^s. \end{equation*}
Because the sets $V_{k+1,I}$ for different $I\in\mathscr{N}_k$ are pairwise disjoint (see \S\ref{sec:noflatpoints}), summing over all such $I$ yields (\ref{eq:est1'}).

\medskip
\underline{Estimate 2.} We estimate the new phantom masses and new intervals coming from some $I\in \mathscr{N}_{k}$ such that $\a_{k,f_k(I)} < \a_0$ and $\mathscr{I}_{k+1}(I) \subset \partial T$.  In particular, we show that
\begin{equation}\label{eq:est2}\tag{E2} \begin{split}
&\sum_{(k+1,I)\in\mathcal{E}_2}(\diam{f_{k+1}(I)})^s + \sum_{(k+1,I) \in \mathcal{P}_2}p_{k+1,I}\\ &\qquad\leq \sum_{(k,I)\in \mathscr{P}_{k}\cap \partial p(T)}p_{k,I} + \frac16\sum_{(k+1,I) \in \mathcal{B}^*_1} (\diam{f_{k+1}(I)})^s.\end{split}
\end{equation} This estimate is responsible for the choice of the constant $14A^*$ appearing in the definition of $\Flat(k)$, and thus, for the constant $30A^*$ appearing in the definition of $\alpha_{k,v}$. Inequality (\ref{eq:est2}) is equivalent to
\begin{equation}\label{eq:est2*}
\begin{split}
\sum_{\substack{I \in \mathscr{N}_k\\ \{k+1\}\times\mathscr{N}_{k+1}(I) \subset \partial T \\ \a_{k,f_k(I)} < \a_0 }} &\left ( \sum_{J \in \mathscr{E}_{k+1}(I)} (\diam{f_{k+1}(I)})^s + \sum_{\substack{J \in \mathscr{N}_{k+1}(I) \\ (k+1,J) \in \mathscr{P}_{k+1}}}p_{k+1,J} \right )\\
&\leq \sum_{\substack{I \in \mathscr{N}_k\\ \{k+1\}\times\mathscr{N}_{k+1}(I) \subset \partial T \\ \a_{k,f_k(I)} < \a_0 }}p_{k,I} + \frac16 \sum_{(k+1,I') \in \mathcal{B}^*_1}(\diam{f_{k+1}(I)})^s.
 \end{split}\end{equation}
To prove (\ref{eq:est2*}) fix $I \in \mathscr{N}_k$ such that $\{k+1\}\times\mathscr{N}_{k+1}(I) \subset \partial T$ and $\a_{k,f_k(I)} < \a_0$. There are nine cases (1a, 1b, 2a, 2b, 2c, 2d, 3a, 3b, 3c). We sincerely apologize to the reader.

For the first four cases (1a, 1b, 2a, 2b), we show that
\begin{equation}\label{eq:est2'}
\sum_{J \in \mathscr{E}_{k+1}(I)} (\diam{f_{k+1}(J)})^s + \sum_{\substack{J \in \mathscr{N}_{k+1}(I)\\ (k+1,J) \in \mathscr{P}_{k+1}}}p_{k+1,J} \leq  p_{k,I} .
\end{equation}

\emph{Case 1.} Suppose $f_k(I)$ is 2-sided terminal in $V_k$.

\emph{Case 1a.} If $\mathscr{N}_{k+1}(I)=\{I\}$, then $f_{k+1}(I)$ is 2-sided terminal in $V_{k+1}$, $\mathscr{E}_{k+1}(I)=\emptyset$ and the new phantom mass $p_{k+1,I}=2P \rho_{k+1}^sr_0^s$ is dominated by the old phantom mass $p_{k,I}=2P\rho_{k}^sr_0^s$. Hence \eqref{eq:est2'} holds.

\emph{Case 1b.} Assume that $\mathscr{N}_{k+1}(I)$ contains at least two elements (see Figure \ref{fig:flatpoints} above). In this case, at most two elements of $\mathcal{N}_{k+1}(I)$ map to 1-sided terminal vertices in $V_{k+1}$. By Lemma \ref{lem:flatpointsum},
\begin{align*} p_{k+1,J_1}+p_{k+1,J_2} + \sum_{J\in\mathscr{I}_{k+1}(I)} (\diam f_{k+1}(J))^s & \leq 2P\rho_{k+1}^sr_0^s + 4(1.1C^*)^s\rho_{k+1}^s r_0^s\\
&\leq 2P \rho_k^sr_0^s = p_{k,I},\end{align*} because $P$ was chosen to be sufficiently large such that $$[P+2(1.1C^*)^s]\xi_2^s \leq P.$$
Thus, \eqref{eq:est2'} holds in this case, as well. \smallskip

\smallskip

\emph{Case 2.} Suppose that $f_k(I)$ is 1-sided terminal in $V_k$. Let $(v_1,v')$ be the unique element in $\Flat(k)$ with $v_1=f_k(I)$ and let $v_2$ be the first vertex in $V_{k+1}$ between $v_1$ and $v'$ in the direction going from $v_1$ to $v'$. By property (P7), we can find an interval $L$ in $\mathscr{E}_k$ such that $f_k(L)=(v_1,v')$ and $I\cap \overline{L}\neq\emptyset$. Let $K$ be an interval in $\mathscr{I}_{k+1}(L)$ such that $f_k(K)=(v_1,v_2)$. There will be four cases, depending on whether $\mathscr{N}_{k+1}(I)$ contains one or more elements and whether $K$ belongs to $\mathscr{E}_{k+1}(L)$ or $\mathscr{B}_{k+1}(L)$.

\emph{Case 2a.} Suppose that $\mathscr{N}_{k+1}(I)=\{I\}$ and $K\in\mathscr{E}_{k+1}(L)$. Then $f_{k+1}(I)$ is 1-sided terminal in $V_{k+1}$, $\mathscr{E}_{k+1}(I)=\emptyset$, and the new phantom mass $p_{k+1,I}=P\rho_{k+1}^sr_0^s$ is dominated by the old phantom mass $p_{k,I}=P\rho_k^sr_0^s$. Hence \eqref{eq:est2'} holds.

\emph{Case 2b.} Suppose that $\mathscr{N}_{k+1}(I)$ contains at least two elements (see Figure \ref{fig:flatpoints} above) and $K\in\mathscr{E}_{k+1}(L)$. Then at most one element of $\mathcal{N}_{k+1}(K)$ maps to a 1-sided terminal vertex in $V_{k+1}$. By Lemma \ref{lem:flatpointsum}, \begin{align*}
p_{k+1,J_1} + \sum_{J\in\mathscr{I}_{k+1}(I)} (\diam f_{k+1}(J))^s & \leq P\rho_{k+1}^sr_0^s + 2(1.1C^*)^s \rho_{k+1}^s r_0^s\\
& \leq P \rho_k^s r_0^s = p_{k,I},\end{align*} because $P$ was chosen to be sufficiently large such that $$[P+2(1.1C^*)^s]\xi_2^s \leq P.$$ Thus, \eqref{eq:est2'} holds, once again.

\emph{Case 2c.} Suppose that $\mathscr{N}_{k+1}(I)=\{I\}$ and $K\in\mathscr{B}_{k+1}(L)$. Then $f_{k+1}(I)$ is 2-sided terminal in $V_{k+1}$ and $\mathscr{E}_{k+1}(I)=\emptyset$. The new phantom mass that must be paid for is $p_{k+1,I}=2P\rho_{k+1}^sr_0^s$. In this case, we pay for one half of $p_{k+1,I}$ with $p_{k,I}=P\rho_{k}^sr_0^s$ and use Lemma \ref{lem:bridgelength} to pay for the other half of $p_{k+1,I}$ with $\frac16(\diam f_{k+1}(K))^s$, where $K\in\mathscr{B}_{k+1}^*(T)$. That is, $$p_{k+1,I} \leq p_{k,I} + \frac{1}{6}(\diam f_{k+1}(K))^s.$$

\emph{Case 2d.} Suppose that $\mathscr{N}_{k+1}(I)$ contains at least two points and $K\in\mathscr{B}_{k+1}(L)$. Then $f_{k+1}(I)$ is 1-sided terminal in $V_{k+1}$, $\mathscr{E}_{k+1}(I)$ is nonempty, and up to one of the new vertices drawn could be 1-sided terminal in $V_{k+1}$, as well. In this case, $$p_{k+1,I}+\underbrace{p_{k+1,J_1} + \sum_{J\in \mathscr{E}_{k+1}(I)} (\diam f_{k+1}(J))^s} \leq \frac{1}{6}(\diam f_{k+1}(K))^s+ p_{k,I}.$$ by Lemma \ref{lem:bridgelength}, Lemma \ref{lem:flatpointsum}, and the choice of $P$.

\smallskip

\emph{Case 3.} Suppose $f_k(I)$ is not terminal in $V_k$. Then $\mathscr{E}_{k+1}(I)=\emptyset$. It remains to pay for $p_{k+1,I}$ as needed. Let $L_1$, $L_2$, $K_1$, and $K_2$ be defined by analogy with $L$ and $K$ from Case 2, but corresponding to the two distinct flat pairs $(f_k(I),v')$ and $(f_k(I),v'')$. There are three cases, depending on whether $K_1$ and $K_2$ both edge intervals, one of $K_1$ or $K_2$ is an edge interval and the other is a bridge interval, or both $K_1$ and $K_2$ are bridge intervals.

\emph{Case 3a.} Suppose that $K_1$ belongs to $\mathscr{E}_{k+1}(L_1)$ and $K_2$ belongs to $\mathscr{E}_{k+1}(L_2)$. Then $f_{k+1}(I)$ is non-terminal in $V_{k+1}$. Hence $(k+1,I) \not\in\mathscr{P}_{k+1}$ and both sides of (\ref{eq:est2'}) are zero. In other words, there is nothing to pay for in Case 3a.

\emph{Case 3b.} Suppose that one of $K_1$ or $K_2$ is an edge interval and the other is a bridge interval, say without loss of generality that $K_1\in\mathscr{E}_{k+1}(L_1)$ and $K_2\in\mathscr{B}_{k+1}(L_2)$. Then $f_{k+1}(I)$ is 1-sided terminal in $V_{k+1}$ and $p_{k+1,I}=P\rho_{k+1}^sr_0^s \leq  \frac{1}{6}(\diam f_{k+1}(K_2))^s$ by Lemma \ref{lem:bridgelength}.

\emph{Case 3c.} Suppose that $K_1$ belongs to $\mathscr{B}_{k+1}(L_1)$ and $K_2$ belongs to $\mathscr{B}_{k+1}(L_2)$. Then $f_{k+1}(I)$ is 2-sided terminal in $V_{k+1}$ and $$p_{k+1,I} = 2P\rho_{k+1}^sr_0^s \leq \  \frac{1}{6}(\diam f_{k+1}(K_1))^s + \frac{1}{6}(\diam f_{k+1}(K_2))^s.$$

Adding up the estimates in the nine cases, we obtain (\ref{eq:est2}).

\medskip

\underline{Estimate 3.} On one hand, $(k,I)\in \mathcal{B}_1'$ if and only if $(k+1,I)\in\mathcal{B}_4$ (see \S\ref{sec:bridges}). When $(k,I)\in \mathcal{B}_1'$, we have $\diam{f_k(I)} = \diam{f_{k+1}(I)}$.  Thus,
$$\sum_{(k+1,I)\in\mathcal{B}_4}(\diam{f_{k+1}(I)})^s = \sum_{(k,I) \in \mathcal{B}_1'}(\diam{f_{k}(I)})^s.$$
On the other hand, when both endpoints of the image of an edge interval are non-flat, the edge interval becomes a bridge interval and pays for itself (see \S\ref{sec:noflatedges}):
$$\sum_{(k+1,I)\in\mathcal{B}_2}(\diam{f_{k+1}(I)})^s =  \sum_{(k,I) \in \mathcal{E}_1'}(\diam{f_{k}(I)})^s.$$ All together, we have \begin{equation}\label{eq:est3}\tag{E3}
\sum_{(k+1,I)\in\mathcal{B}_2\cup\mathcal{B}_4}(\diam{f_{k+1}(I)})^s =  \sum_{(k,I) \in \mathcal{E}_1'\cup\mathcal{B}_1'}(\diam{f_{k}(I)})^s.
\end{equation}

\medskip

\underline{Estimate 4.} Next, we control the new phantom masses at endpoints of images $f_{k+1}(J)$ of bridge intervals $J\in\mathscr{B}_{k+1}(I)$ coming from some edge interval $I\in\mathscr{E}_k$ as in \S\ref{sec:flatedges} such that the endpoint lies between the endpoints of $f_k(I)$. Specifically, we show that
\begin{align}\label{eq:est4}\tag{E4}
\sum_{(k+1,I)\in \mathcal{P}_3}p_{k+1,I} \leq \frac16\sum_{(k+1,I)\in\mathcal{B}^*_2}(\diam{f_{k+1}(I)})^s.
\end{align}
Inequality (\ref{eq:est4}) is equivalent to
\begin{align}\label{eq:est4'}
\sum_{(k,I)\in\mathcal{E}_2'}\sum_{\substack{J\in \mathscr{N}_{k+1}(I)\\ (k+1,J) \in \mathscr{P}_{k+1}}}p_{k+1,J} \leq\frac16\sum_{(k+1,I)\in\mathcal{B}^*_2}(\diam{f_{k+1}(I)})^s.
\end{align}
To prove (\ref{eq:est4'}), fix $(k,I)\in\mathcal{E}_2'$ and $J\in \mathscr{N}_{k+1}(I)$ be such that $(k+1,J)\in\mathscr{P}_{k+1}$. Then $f_{k+1}(J)$ lies strictly between the endpoints of $f_k(I)$ and $f_{k+1}(J)$ is either 1- or 2-sided terminal in $V_{k+1}$. On one hand, if $f_{k+1}(J)$ is 1-sided terminal, then there exists precisely one element $(k+1,K)\in\mathcal{B}_2^*$ such that $f_{k+1}(J)$ is an endpoint of $f_{k+1}(K)$. In this case, $$p_{k+1,J} =P\rho_{k+1}^sr_0^s\leq \frac{1}{6}(\diam f_{k+1}(K))^s$$ by Lemma \ref{lem:bridgelength}. On the other hand, if $f_{k+1}(J)$ is 2-sided terminal, then there exist two elements $(k+1,K_1)$ and $(k+1,K_2)$ in $\mathcal{B}_2^*$ such that $f_{k+1}(J)$ is the common endpoint of $f_{k+1}(K_1)$ and $f_{k+1}(K_2)$. In this case, $$p_{k+1,J} = 2P\rho_{k+1}^sr_0^s \leq \frac{1}{6}(\diam f_{k+1}(K_1))^s+\frac{1}{6}(\diam f_{k+1}(K_2))^s$$ by Lemma \ref{lem:bridgelength}.

\begin{rem} In Estimates 2 and Estimate 4, each endpoint of the image $f_{k+1}(I)$ of $(k+1,I)\in\mathcal{B}_1^*\cup \mathcal{B}_2^*$ is used once and each $f_{k+1}(I)$ has only two endpoints. Hence $$\frac16\sum_{(k+1,I)\in\mathcal{B}^*_1\cup \mathcal{B}^*_2}(\diam{f_{k+1}(I)})^s\leq \frac{1}{6}\sum_{(k+1,I)\in\mathscr{B}^{\,**}_{k+1}(T)}=\frac{1}{3}\sum_{(k+1,I)\in\mathscr{B}^{*}_{k+1}(T)}.$$
\end{rem}

\medskip

\underline{Estimate 5.} In this final estimate, we deal with new intervals in $\partial T$ coming from an edge interval in $\partial p(T)$ which has an endpoint with flat image. We will show that
\begin{align}\label{eq:est5}\tag{E5}
\sum_{(k+1,I)\in\mathcal{E}_3\cup\mathcal{B}_3} (\diam{f_{k+1}(I)})^s \leq \sum_{(k,I) \in \mathcal{E}_2'}(\diam{f_{k+1}(I)})^s + (14A^*)^s\sum_{(v,v')\in\Flat(k)} \tau_{s}(k,v,v') \rho_k^sr_0^s.
\end{align}
For each $I \in \mathcal{E}_2'$, pick an endpoint $x_I$ of $I$ such that $\a_{k,f_k(x_I)} < \a_0$ and let $y_I$ be the other endpoint of $I$. Estimate (E5) follows immediately from
\begin{equation}\label{eq:est5'}
\begin{split}
&\sum_{(k,I) \in \mathcal{E}_2' } \sum_{J\in \mathscr{E}_{k+1}(I)\cup\mathscr{B}_{k+1}(I)}(\diam{f_{k+1}(J)})^s \\
&\quad\leq \sum_{(k,I) \in \mathcal{E}_2'}\left((\diam f_{k}(I)\right)^s + (14A^*)^s\tau_{s}(k,f_k(x_I),f_k(y_I)))\rho_k^sr_0^s.
\end{split}
\end{equation}
To show (\ref{eq:est5'}), fix $(k,I)\in \mathcal{E}_2'$ and enumerate $\mathscr{E}_{k+1}(I)\cup\mathscr{B}_{k+1}(I) = \{I_1,\dots, I_{n}\}$. Let $v=f_k(x_I)$ and let $v' = f_{k}(y_I)$. By definition of $\tau_s(k,v,v')$ and (P3), we have \begin{equation*}\begin{split}\sum_{i=1}^{n-1}(\diam f_{k+1}(I_i))^s &\leq |v-v'|^s + \tau_s(k,v,v')|v-v'|^s \\ &\leq (\diam f_k(I))^s + (14A^*)^2\tau_s(k,v,v')\rho_k^sr_0^s.\end{split}\end{equation*} Summing over all pairs $(k,I)\in\mathcal{E}_2'$, we obtain \eqref{eq:est5'}.

\medskip

Adding (\ref{eq:est1}), (\ref{eq:est2}), (\ref{eq:est3}), (\ref{eq:est4}), and (\ref{eq:est5}), we arrive at (\ref{eq:estimates}). This completes the proof of Lemma \ref{lem:estimates}.

\section{H\"older parametrization}\label{sec:Holder}

In \S\ref{sec:proof1}, we prove the following theorem, which is the paper's main result. Afterwards, in \S\ref{sec:cortothm}, we derive several corollaries, including Theorem \ref{thm:main2}. In \S\ref{sec:modifiedTST}, we state and prove a refinement of Theorem \ref{thm:main} that gives an essentially 2-to-1 curve. In \S\ref{sec:Carleson}, we show that replacing \eqref{eq:main-condition} with a Carleson type condition produces an upper Ahlfors regular curve.

Given parameters $C^*$, $\xi_1$, and $\xi_2$, let $\alpha_1$ be defined by \eqref{alphabound}. That is, \begin{equation*}\alpha_1 = \min\left\{\frac{1}{16}, \left( \frac{\xi_1(1-\xi_2)}{42 C^*} \right)^{1/2}\right\}.\qedhere\end{equation*}

\begin{thm}[H\"older Traveling Salesman with Nets] \label{thm:main}
Assume that $X=l^2(\R)$ or $X=\R^N$ for some $N\geq 2$. Let $s\geq 1$, let $\mathscr{V}=(V_k,\rho_k)_{k\geq 0}$ be a sequence of finite sets $V_k$ in $X$ and numbers $\rho_k>0$ that satisfy properties (V0)--(V5) defined in \textsection\ref{sec:flat}. If $\alpha_0\in(0,\alpha_1]$ and
\begin{equation} \label{eq:main-condition} S^s_{\mathscr{V}} := \sum_{k=0}^{\infty}\sum_{(v,v')\in \Flat(k)}\tau_{s}(k,v,v') \rho_k^s+ \sum_{k=0}^{\infty}\sum_{\substack{v\in V_k\\ \a_{k,v} \geq \a_0}} \rho_k^s < \infty,\end{equation}
then there exists a  $(1/s)$-H\"older map $f:[0,1] \to X$ such that $f([0,1])\supset \bigcup_{k\geq 0}V_k$ and the H\"older constant of $f$ satisfies $H \lesssim_{s,C^*,\xi_1,\xi_2} r_0 (1+  S^s_{\mathscr{V}})$.
\end{thm}

\subsection{Proof of Theorem \ref{thm:main} }\label{sec:proof1}

In this subsection, $w$ will always denote a finite word in the alphabet $\N=\{1,2,\dots\}$, including the empty word $\emptyset$. We denote the length of a word $w$ by $|w|$.

The conclusion holds trivially if $\bigcup_{k\geq 0}V_k$ is a singleton. Thus, in addition to $S^s_{\mathscr{V}}<\infty$, we may assume that $\bigcup_{k\geq 0}V_k$ contains at least two points. Because $\alpha_0\leq \alpha_1$, Proposition \ref{prop:mass} gives \begin{equation}\label{e:S-mass-bound} 0<\mathcal{M}_s([0,1]) \lesssim_{s,C^*,\xi_2} r_0^s(1+ S^s_{\mathscr{V}})<\infty.\end{equation}

To proceed, we start by renaming the intervals in $\{[0,1]\}\cup\mathscr{I}$. Denote $\D_{\emptyset} = [0,1]$ and write $\mathscr{I}_0= \{\D_1,\dots, \D_{n_{\emptyset}}\}$, enumerated according to the orientation of $[0,1]$. Inductively, suppose that for some word $w$ with $|w|=k$, we have defined $\D_{w} \in \mathscr{I}_{k}$. Suppose also that
\[ \mathscr{I}_{k+1}(\D_w) = \{J_1,\dots,J_{n_w}\},\]
enumerated according to the orientation of $[0,1]$. Then for each $i\in\{1,\dots,n_w\}$, denote $\D_{wi} = J_i$. Denote by $\mathcal{W}$ the set of all finite words with letters from $\N$ for which an interval $\D_{w}$ has been defined.

Next, we use the masses of intervals defined in \S\ref{sec:mass} to modify the length of intervals $\D_w$.
Define $\D'_{\emptyset} = [0,1]$. Let $\{\D'_{1},\dots,\D'_{n_{\emptyset}}\}$ be a partition of $\D'_{\emptyset}$, enumerated according to the orientation of $[0,1]$, satisfying
\begin{enumerate}
\item $\D_i'$ is open (resp.~ closed) if and only if $\D_i$ is open (resp.~ closed), and
\item $\diam{\D'_{i}} = \mathcal{M}_s(0,\D_{i})/\mathcal{M}_s([0,1])$.
\end{enumerate}
These intervals exist, because $\mathcal{M}_s([0,1]) = \sum_{i=1}^{n_{\emptyset}}\mathcal{M}_s(0,\D_i)$.
Inductively, suppose that an interval $\D'_w\subset [0,1]$ has been defined for some $w\in\mathcal{W}$ such that
\[ \diam{\D'_w} \geq \frac{\mathcal{M}_s(|w|-1,\D_w)}{\mathcal{M}_s([0,1])}\]
and $\D_{w}'$ is open (resp.~ closed) if and only if $\D_{w}$ is open (resp.~ closed). Let $\{\D'_{w1},\dots,\D'_{wn_w}\}$ be a partition of $\D'_w$, enumerated according to the orientation of $[0,1]$, satisfying
\begin{enumerate}
\item $\D_{wi}'$ is open (resp.~ closed) if and only if $\D_{wi}$ is open (resp.~ closed), and
\item $\diam{\D'_{wi}} \geq \mathcal{M}_s(|w|,\D_{wi})/\mathcal{M}_s([0,1])$.
\end{enumerate}
This partition exists by Lemma \ref{lem:sum}(3).

Define the family $$\mathscr{E}_k' = \{\D'_w : \D_w \in \mathscr{E}_k\}$$ and similarly define the families $\mathscr{B}_k'$, $\mathscr{N}_k'$, $\mathscr{F}_k'$, and $\mathscr{I}_k'$.
For each $k\geq 0$, define a continuous map $F_k:[0,1] \to X$ by
\[ F_k|\D_{w}' = (f_k|\D_w)\circ \phi_w \qquad\text{for all }w\in \W,\]
where $\phi_w$ is the unique increasing affine homeomorphism mapping $\D_w'$ onto $\D_w$ when $\D_w'$ is nondegenerate and $\phi_w$ maps to any point in $\D_w$ when $\D_w'$ is a singleton. (The latter possibility occurs only when $\Delta_w'$ belongs to $\mathscr{F}'_k$ or $\mathscr{N}'_k$.)

We now prove two auxiliary results for the sequence $(F_k)_{k\geq 0}$.

\begin{lem}\label{lem:Hold-conv1}
For all $k\geq 0$ and $x\in[0,1]$, we have $|F_{k+1}(x)-F_k(x)| \leq 30A^*\xi_2r_0\rho_{k}.$
\end{lem}

\begin{proof}
Fix $x\in [0,1]$ and $w\in \W$ such that $x\in\D_w$ and $|w|=k$. Let also $i\in\{1,\dots,n_{w}\}$ be such that $x\in\D_{wi}'$.

If $\D_w'\in \mathscr{B}_k'$, then $\D_w' = \D_{wi}'$, $\D_w = \D_{wi}$, and $F_k|\D_w' = F_{k+1}|\D_{wi}'$. We conclude that $|F_{k+1}(x) - F_k(x)| = 0$.

If $\D_w'\in \mathscr{F}_k'$, then $\D_w' = \D_{wi}' = \{x\}$, and $F_k(x) = F_{k+1}(x)$. Hence $|F_{k+1}(x) - F_k(x)| = 0$.

If $\D'_w\in \mathscr{E}_k'$, then $|F_{k+1}(x) - F_k(x)|\leq 2\diam{F_k(\D_w')} = 2\diam{f_k(\D_w)} \leq 28A^* \rho_{k+1}r_0$.

If $\D'_w\in \mathscr{N}_k'$, then $|F_{k+1}(x) - F_k(x)|\leq 2\diam{f_{k}(\D_w)} \leq \diam{\tilde{V}_{k+1,I}} \leq 30A^* \rho_{k+1}r_0$, where $\tilde V_{k+1,I}$ is a set defined in \S\ref{sec:noflatpoints}.
\end{proof}

\begin{lem}\label{lem:Hold-conv2}
For all $k\geq 0$ and $x,y \in[0,1]$,
\[|F_{k}(x) - F_k(y)| \leq \mathcal{M}_s([0,1]) r_0^{1-s}\rho_k^{1-s}|x-y|
\lesssim_{s,C^*,\xi_2} r_0 (1+  S^s_{\mathscr{V}}) \rho_k^{1-s} |x-y|.\]
\end{lem}

\begin{proof}
Fix $k\geq 0$  and $x,y \in [0,1]$. Without loss of generality, we may assume that $x<y$. We consider three cases. In the first two cases, the points $x$ and $y$ belong to  the same interval $\D_w'$, $|w|=k$, while in the third case they belong to different intervals.

\emph{Case 1.} If $x,y \in \D_{w}' \in \mathscr{N}_k'\cup\mathscr{F}_k'$, then $|F_k(x)-F_k(y)| = 0|x-y|$, because the map $F_k|\D_w'$ is constant.

\emph{Case 2.} Suppose that $x,y \in \D_{w}' \in \mathscr{E}_{k}'\cup\mathscr{B}_k'$. Since $F_k|\D_w'$ is affine,
\begin{align*}
|F_k(x) - F_k(y)| &\leq \mathcal{M}_s([0,1])\frac{\diam{f_k(\D_w)}}{\mathcal{M}_s(|w|-1,\D_w)} |x-y|\\
&\leq \mathcal{M}_s([0,1])\diam{f_k(\D_w)}^{1-s}|x-y|\\
&\leq \mathcal{M}_s([0,1])r_0^{1-s}\rho_k^{1-s}|x-y|,
\end{align*} by (V3) and the assumption $s\geq 1$. Thus, by \eqref{e:S-mass-bound}, $$|F_k(x)-F_k(y)| \lesssim_{s,C^*,\xi_2} r_0(1+S_{\mathscr{V}}^s)\rho_k^{1-s}|x-y|.$$

\emph{Case 3.} Suppose that $x\in \D_{w}'$ and $y\in \D_{u}'$ for some $\D_w',\D_u' \in \mathscr{I}_k'$ with $\D_w' \cap \D_u' = \emptyset$. By the preceding cases and the Fundamental Theorem of Calculus,
\[ |F_k(x)-F_k(y)| \leq \int_x^y |\nabla F_k(t)| \, dt \leq \mathcal{M}_s([0,1])r_0^{1-s}\rho_k^{1-s}|x-y| \lesssim_{s,C^*,\xi_2} r_0(1+S_{\mathscr{V}}^s)\rho_{k}^{1-s}|x-y|.\qedhere\]
\end{proof}

We are now ready to prove Theorem \ref{thm:main}.

\begin{proof}[{Proof of Theorem \ref{thm:main}}]
Define $F:[0,1]  \to X$ pointwise by
\[ F(x) := F_0(x) + \sum_{k=0}^{\infty} (F_{k+1}(x) - F_k(x)).\]
By Lemma \ref{lem:Hold-conv1}, $F$ is well defined and continuous in all $[0,1]$. By Lemma \ref{lem:Hold-conv1}, Lemma \ref{lem:Hold-conv2}, and Lemma \ref{l:LipHold} from the appendix, $F$ is $(1/s)$-H\"older continuous with H\"older constant
\[ H \leq \frac{1}{\xi_1}\left(\mathcal{M}_s([0,1])r_0^{1-s} + 60A^*r_0 \frac{\xi_2}{1-\xi_2}\right)\lesssim_{s,C^*,\xi_1,\xi_2} r_0 (1+ S^s_{\mathscr{V}}).\]
Finally, for any integer $k\geq 0$ and any integer $m\geq k$, we have $V_k \subset F_m([0,1])$. Therefore, $V_k \subset F([0,1])$ for all integers $k\geq 0$. \end{proof}

\subsection{Corollaries to Theorem \ref{thm:main} and Proof of Theorem \ref{thm:main2}}\label{sec:cortothm}

\begin{cor}[tube approximation] \label{cor:nets1}
For all $s>1$, $C^*\geq 1$, and $0<\xi_1<\xi_2<1$, there exists $\a^*>0$ with the following property. Assume that $X=l^2(\R)$ or $X=\R^N$ for some $N\geq 2$. Let  $\mathscr{V} = (V_k,\rho_k)_{k\geq 0}$ be a sequence of finite sets in $X$ and numbers $\rho_k>0$ satisfying properties (V0)--(V5) of \textsection\ref{sec:flat} with constants $C^*$, $\xi_1$, and $\xi_2$. If
\[S^{s,+}_{\mathscr{V}} := \sum_{k=0}^{\infty}\sum_{\substack{v\in V_k\\ \a_{k,v} \geq \a^*}} \rho_k^s < \infty,\]
then there exists a  $(1/s)$-H\"older map $f:[0,1] \to X$ such that $\bigcup_{k\geq 0}V_k \subset f([0,1])$ and the H\"older constant of $f$ satisfies $H \lesssim_{s,C^*,\xi_1,\xi_2} r_0 (1+ S^{s,+}_{\mathscr{V}})$.
\end{cor}

\begin{proof}
By Lemma \ref{lem:flat2}, there exists $\epsilon_{s,C^*,\xi_1,\xi_2}\in(0,1/16]$ such that if $k\geq 0$, $v\in V_k$, and $\a_{k,v} \leq \epsilon_{s,C^*,\xi,\xi_2}$, then $\tau_{s}(k,v,v') = 0$ for all $(v,v')\in\Flat(k)$. Set $\alpha^*=\min\{\epsilon_{s,C^*,\xi_1,\xi_2},\alpha_1\}$ (a careful inspection shows $\epsilon_{s,C^*,\xi_1,\xi_2}$ is strictly smaller than $\alpha_1$). Thus, with $\alpha_0=\alpha^*$,
\[ \sum_{k=0}^{\infty}\sum_{(v,v')\in \Flat(k)}\tau_{s}(k,v,v') \rho_k^s + \sum_{k=0}^{\infty}\sum_{\substack{v\in V_k\\ \a_{k,v} \geq \a^*}} \rho_k^s = \sum_{k=0}^{\infty}\sum_{\substack{v\in V_k\\ \a_{k,v} \geq \a^*}} \rho_k^s< \infty.\] The conclusion follows immediately by Theorem \ref{thm:main}.
\end{proof}

\begin{cor}\label{cor:nets2} Alternatively, if
\[S^{s,p}_{\mathscr{V}} := \sum_{k=0}^{\infty}\sum_{v\in V_k} \a_{k,v}^p\, \rho_k^s < \infty\quad\text{for some }p>0,\]
then there exists a  $(1/s)$-H\"older map $f:[0,1] \to X$ such that $\bigcup_{k\geq 0}V_k \subset f([0,1])$ and the H\"older constant of $f$ satisfies $H \lesssim_{s,C^*,\xi_1,\xi_2} r_0 (1+ (\alpha^*)^{-p}S^{s,p}_{\mathscr{V}})$.
\end{cor}

\begin{proof} Inspecting the definitions of the two sums, $S^{s,+}_{\mathscr{V}} \leq (\alpha^*)^{-p} S^{s,p}_{\mathscr{V}}$.\end{proof}

We now turn to the proof of Theorem \ref{thm:main2}.

\begin{rem}\label{rem:alpha-beta}
For all $x\in\R^N$ and $r>0$, a minimal dyadic cube $Q$ in $\R^N$ such that $x\in Q$ and $3Q$ contains $B(x,r)$ satisfies $\diam 3Q \leq Cr$ for some $C=C(N)>0$.
\end{rem}

\begin{proof}[{Proof of Theorem \ref{thm:main2}}] Let $N\geq 2$ and $s>1$ be given. Fix $\beta_0>0$ to be specified below. Assume that $E\subset\R^N$ is a bounded set such that $$S_E^{s,+}=\sum_{\stackrel{Q\in\Delta(\R^N)}{\beta_E(3Q)\geq \beta_0}} (\diam Q)^s<\infty.$$
Pick any $x_0 \in E$ and set $r_0=\diam E$. Define $V_0 = \{x_0\}$. Assume that $V_k$ has been defined for some $k$. Choose a maximal $2^{-(k+1)}$-separated set in $E$ such that $V_{k+1}\supset V_k$. Then the sequence $\mathscr{V}=(V_k,2^{-k})_{k\geq 0}$ satisfies conditions (V0)--(V4) in \textsection\ref{sec:flat} with $C^*=2$ and $\xi_1=\xi_2 = 1/2$. Note that $$A^*=\frac{C^*}{1-\xi_2}=4,\quad 30A^*=120.$$ For all $k\geq 0$ and $v\in V_k$, let $Q_{k,v}$ be a minimal dyadic cube such that $v\in Q_{k,v}$ and $3Q_{k,v}$ contains $B(v,120\cdot 2^{-k}r_0)$ and choose $\ell_{k,v}$ be a line such that
\[ \sup_{x\in E \cap 3Q} \dist(x, \ell_{k,v})=\beta_E(3Q) \diam{3Q}.\] Then, by Remark \ref{rem:alpha-beta}, there exists $C=C(N)>0$ such that $$\alpha_{k,v}:=\frac{1}{2^{-{k+1}}r_0}\sup_{x\in V_{k+1}\cap B(v,120\cdot 2^{-k}r_0)} \dist (x,\ell_{k,v}) \leq  \frac{\diam 3Q}{2^{-{k+1}}r_0} \beta_E(3Q) \leq 240C \beta_E(3Q).$$ We now specify that $\beta_0 = \a^*/240C$, where $\a^*$ is the constant from Corollary \ref{cor:nets1} and $C$ the constant from Remark \ref{rem:alpha-beta}. Because each dyadic cube $Q$ in $\R^N$ is associated to some $(k,v)$ at most $C(n)$ times, it follows that
\[S^{s,+}_{\mathscr{V}} \lesssim_N r_0^{-s}S^{s,+}_E < \infty.\]
Therefore, by Corollary \ref{cor:nets1}, there exists a $(1/s)$-H\"older m fap $f:[0,1] \to \R^N$ such that $\bigcup_{k\geq 0}V_k \subset f([0,1])$ and the H\"older constant of $f$ satisfies
\[H \lesssim_{N,s} r_0 (1+ S^{s,+}_{\mathscr{V}})  \lesssim_{N,s} \diam{E} + \frac{S^{s,+}_E}{(\diam{E})^{s-1}}.\]
Because $\bigcup_{k\geq 0}V_k$ is dense in $E$, the curve $f([0,1])$ also contains the set $E$.
\end{proof}

\subsection{A refinement of Theorem \ref{thm:main}}\label{sec:modifiedTST}
The parameterization in Theorem \ref{thm:main} can be made in such a way so that the sequence of maps $F_k$ obtained are \emph{essentially 2-to-1} in the sense of the following proposition.

\begin{prop}\label{prop:2-to-1Holder}
Let $\mathscr{V}=(V_k,\rho_k)_{k\geq 0}$ and $\alpha_0$ satisfy the hypothesis of Theorem \ref{thm:main} and let $x_0\in V_0$. There exists a sequence of piecewise linear maps $F_k:[0,1]\rightarrow X$ with the following properties.
\begin{enumerate}
\item For all $k\geq 0$, $F_k(0)=x_0=F_k(1)$.
\item For all $k\geq 0$, there exists $G_k\subset[0,1]$ such that $F_k|G_k$ is 2-1 and $F_k([0,1]\setminus G_k)$ is a finite set.
\item For all $k\geq 0$, $F_k([0,1])\supset V_k$; for all $x\in V_{k+1}$, $\dist(x,F_k([0,1])) \leq C^*\rho_{k+1}r_0$.
\item For all $k\geq 0$, $\|F_k - F_{k+1}\|_{\infty} \lesssim_{C^*,\xi_2} \rho_{k+1}r_0$.
\item For all $k\geq 0$, the map $F_k$ is Lipschitz with $\Lip(F_k) \lesssim_{s,C^*,\xi_1,\xi_2} r_0(1+S^s_{\mathscr{V}})\rho_k^{1-s}$.
\end{enumerate} The maps $F_k$ converge uniformly to a $(1/s)$-H\"older map $F:[0,1]\rightarrow X$ whose image contains $\bigcup_{k\geq 0}V_k$, the parameterization $F$ starts and ends at $x_0$ in the sense of (1), and the H\"older constant of $F$ satisfies $H\lesssim_{s,C^*,\xi_1,\xi_2} r_0(1+S^s_{\mathscr{V}})$.
\end{prop}

\begin{proof}
Following the algorithm of \textsection\ref{sec:curve}, we construct for each $k\geq 0$, four families $\mathscr{E}_k$, $\mathscr{B}_k$, $\mathscr{F}_k$, $\mathscr{N}_k$ of intervals in $[0,1]$ and a continuous piecewise linear map $f_k:[0,1] \to X$ that satisfy (P1)--(P7). In Step 0, we may assume that $f_0(0) = f_0(1) = x_0$. Thus, $f_k(0) = f_k(1) = x_0$ for all $k\geq 0$. Moreover, for all $x\in V_{k+1}$,
\[ \dist(x,F_k([0,1])) \leq \dist(x,V_k) < C^*r_0\rho_{k+1}\] by (V4). From the construction, $\|f_k-f_{k+1}\|_{\infty} \lesssim_{A^*}\rho_{k+1}r_0$. Hence $\|f_k-f_{k+1}\|_{\infty} \lesssim_{C^*,\xi_2}\rho_{k+1}r_0$. We have shown that the maps $f_k$ satisfy properties (1), (3) and (4).

As for property (2), we already know from (P4) that $f_k|\bigcup \mathscr{E}_k$ is 2-to-1. We proceed to modify $f_k$ on each $I\in\mathscr{B}_k$. From the algorithm in \textsection\ref{sec:curve}, recall that for each $I\in \mathscr{B}_k$, there exists unique $I' \in \mathscr{B}_k$, $I'\neq I$ such that $f_k(I) = f_k(I')$. Enumerate
\[\mathscr{B}_k = \{I_1, I_1',\dots,I_l, I_l'\},\]
where $f_k(I_j) = f_k(I_j')$. Starting with $I_1 = (a_1,b_1)$, define $\tilde{f}_k|I_1$ so that
\begin{enumerate}
\item[(a)] $\tilde{f}_k| I_1$ is piecewise linear and continuous, $\tilde{f}_k(a_1) = f_{k}(a_1)$ and  $\tilde{f}_k(b_1) = f_{k}(b_1)$;
\item[(b)] $\mathcal{H}^1(\tilde{f}_k(I_1)) \leq |f_k(a_1) - f_k(b_1)| + \rho_{k}r_0$;
\item[(c)] $\tilde{f}_k(I_1) \cap \bigcup_{I \in \mathscr{E}_k} f_k(I)$ is a finite set.
\end{enumerate}
Let $\psi_1: I_1' \to I_1$ be the unique orientation-reversing linear map from $I_1'$ onto $I_1$. Then define $\tilde{f}_k|I_1' = (\tilde{f}_k|I_1)\circ \psi_1$. For induction, assume that we have defined $\tilde{f}_k$ on $$I_1,I_1',\dots, I_{r-1},I_{r-1}'.$$ Define $\tilde{f}_k|I_r$ as with $I_1$, only this time we require that the set
\[\tilde{f}_k(I_r) \cap\left (\bigcup_{i=1}^{r-1}\tilde{f}_k(I_i)\cup \bigcup_{i=1}^{r-1}\tilde{f}_k(I_i')\cup \bigcup_{I \in \mathscr{E}_k} f_k(I) \right)\]
be finite. Let $\psi_r: I_r' \to I_r$ be the unique orientation-reversing linear map from $I_r'$ onto $I_r$ and define $\tilde{f}_k|I_r' = (\tilde{f}_k|I_r)\circ \psi_r$. Extending $\tilde{f}_k|I = f_k|I$ for all $I\in\mathscr{E}_k\cup\mathscr{N}_k\cup\mathscr{F}_k$, we obtain a sequence $\tilde{f}_k$ of maps that satisfy properties (1)--(4).

The rest of the proof is similar to that of Theorem \ref{thm:main} and we only sketch the steps. Define the $\mathcal{M}_s$ for each $I \in \mathscr{I}_k$ and define the collections of intervals $\{\D_w\}$ and $\{\D_w'\}$. For each $w$, let $\phi_w: \D_w' \to \D_w$ be the unique affine homeomorphism from $\D_w'$ onto $\D_w$ and let $F_k|\D_w' = (\tilde{f}_k|\D_w)\circ\phi_w$. Although the maps $f_k$ are different from $\tilde{f}_k$, we have by (b) that $\diam{f_k(I)} \simeq_{\xi_2}\diam{\tilde{f}_k(I)}$ for all $k\geq 0$ and all $I\in\mathscr{I}_k$. Thus, Lemma \ref{lem:Hold-conv1} and Lemma \ref{lem:Hold-conv2} still hold with constants depending at most on $s$, $C^*$, $\xi_1$ and $\xi_2$. Therefore, the maps $F_k$ satisfy properties (1)--(5).
\end{proof}

\subsection{A Carleson condition for an upper Ahlfors $s$-regular curve}\label{sec:Carleson}

Replacing \eqref{eq:main-condition} in the main theorem with a Carleson-type condition ensures that the H\"older curve is upper Ahlfors regular. This answers a question posed to us by T. Orponen.

\begin{thm} \label{cor:carleson}
Assume that $X=l^2(\R)$ or $X=\R^N$ for some $N\geq 2$. Let $s\geq 1$, let $\mathscr{V}=((V_k,\rho_k))_{k\geq 0}$ be a sequence of finite sets $V_k$ in $X$ and numbers $\rho_k>0$ that satisfy properties (V0)--(V4) defined in \textsection\ref{sec:flat}. Let $\Lambda \geq C^*$ and $\Lambda^*:=\Lambda/(1-\xi_2)$. Suppose for all $k\geq 0$ and $v\in V_{k+1}$, we are given a line $\ell_{k,v}$ and $\alpha_{k,v}\geq 0$ such that \begin{equation} \tag{$\widetilde{\mathrm{V5}}$} \sup_{x\in V_{k+1}\cap B(v,30\Lambda^*\rho_k r_0)} \dist(x,\ell) \leq \alpha_{k,v} \rho_{k+1}.\end{equation} If $\Lambda\gg_{\xi_1,\xi_2} C^*$, $\alpha_0\in(0,\alpha_1]$, and there exists $M<\infty$ such that for all $j\geq 0$ and $w\in V_{j}$,
\begin{equation}\label{eq:carleson} S^s_{\mathscr{V}}(j,w) := \sum_{k=j}^{\infty}\sum_{\substack{(v,v')\in \Flat(k) \\ v,v'\in B(w,\Lambda \rho_{j}r_0)}} \tau_{s}(k,v,v') \rho_k^s+ \sum_{k=j}^{\infty}\sum_{\substack{v\in V_k\\ \a_{k,v} \geq \a_0 \\  v\in B(w,\Lambda \rho_{j}r_0)}} \rho_k^s \leq M \rho_{j}^s,\end{equation}
then there exists a  $(1/s)$-H\"older map $f:[0,1] \to X$ such that $f([0,1])\supset \bigcup_{k\geq 0}V_k$ and the curve $f([0,1])$ is upper Ahlfors $s$-regular with constant depending on at most $s$, $C^*$, $\xi_1$, $\xi_2$, and $M$.
\end{thm}

\begin{proof} By (V0) and (V4), $\excess\left(\textstyle\bigcup_{k=1}^\infty V_k, V_0\right) \leq \rho_1r_0+\rho_2r_0+\dots = \frac{r_0}{1-\xi_2}\leq  A^*r_0,$ where $$\excess(A,B)=\sup_{x\in A}\inf_{y\in B} |x-y|$$ whenever $A$ and $B$ are nonempty sets in $X$. Hence $\bigcup_{k=0}^\infty V_k \subset B(x_0,2A^*r_0)$ by (V1). Thus, there exists a $(1/s)$-H\"older map $f:[0,1]\rightarrow X$ such that $\Gamma:=f([0,1])\supset \bigcup_{k=0}^\infty V_k$ and the H\"older constant of $f$ satisfies $H_f \lesssim_{s,C^*,\xi_1,\xi_2} r_0(1+M)$ by Theorem \ref{thm:main}, since $S^s_\mathscr{V}=S^s_{\mathscr{V}}(0,w)\leq M$. In particular, $$\Haus^s(\Gamma) \leq H_f^s\, \Haus^1([0,1]) \lesssim_{s,C^*,\xi_1,\xi_2,M} r_0^s.$$
Let $x\in \Gamma$ and let $0<r\leq \diam \Gamma$. Because $\bigcup_{k=0}^\infty V_k \subset B(x_0,2A^*r_0)$, we have $\diam \Gamma \leq 4A^*r_0$. If $r\geq r_0$, then $$\Haus^s(\Gamma\cap B(x,r))\leq \Haus^s(\Gamma)\lesssim_{s,C^*,\xi_1,\xi_2,M} r_0^s \lesssim_{s,C^*,\xi_1,\xi_2,M} r^s.$$ Otherwise, $0<r<r_0$, say $\rho_{j+1}r_0\leq r < \rho_j r_0$ for some integer $j\geq 0$.

Choose $w\in V_j$ such that $|w-x|= \dist(x, V_j)$. By Lemma \ref{lem:Hold-conv1}, (V0), and (V4), $$\dist(x,f_j([0,1])) \leq \frac{30\xi_2}{1-\xi_2}A^*\rho_jr_0.$$ Because $\alpha_0\leq \alpha_1$, the longest line segment drawn between vertices in $V_j$ has length at most $14 A^* \rho_{k-1}r_0$. By (V0), it follows that $$\excess(f_j([0,1]),V_j) \leq \frac{7}{\xi_1} A^*\rho_jr_0.$$ Thus, $\dist(x,V_j)\lesssim_{\xi_1,\xi_2}A^* \rho_jr_0 \ll_{\xi_1,\xi_2} \Lambda\rho_jr_0.$  For all $k\geq 0$, define $$\widetilde{\rho_k} := \frac{\rho_{j+k}}{\rho_j}.$$ Define $\widetilde{V_0} := V_{j} \cap B(w,\Lambda\rho_jr_0)$. Then, for each $k\geq 1$, recursively define $\widetilde{V_{k}}$ to be set of all $x\in V_{j+k} \cap B(w,\Lambda\rho_jr_0)$ such that $\dist(x,\widetilde{V}_{k-1})\leq \Lambda \rho_{j+k}r_0$. Then $\widetilde{\mathscr{V}}=(\widetilde{V_k},\widetilde{\rho_k})_{k\geq 0}$ satisfy (V0)--(V4) with respect to $\widetilde x_0=w$ and $\widetilde r_0=\rho_jr_0$, $\widetilde{C^*}=\Lambda$, $\widetilde{\xi_1}=\xi_1$, and $\widetilde{\xi_2}=\xi_2$. For all $k\geq 0$ and $v\in \widetilde{V}_k$, assign $\widetilde{\ell}_{k,v}:=\ell_{j+k,v}\quad\text{and}\quad \widetilde{\alpha}_{k,v}=\alpha_{j+k,v}.$ Then $\widetilde{\mathscr{V}}$ satisfies (V5) with respect to $\widetilde{\ell}_{k,v}$ and $\widetilde{\alpha}_{k,v}$ by ($\widetilde{\mathrm{V5}}$). Moreover, by \eqref{eq:carleson}, $$S_{\widetilde{\mathscr{V}}}^s=\sum_{k=0}^\infty \sum_{(v,v')\in\widetilde{\Flat(k)}} \widetilde{\tau}(k,v,v') \widetilde{\rho_k}^s+ \sum_{k=0}^\infty \sum_{\substack{v\in \widetilde{V_k} \\ \widetilde{\alpha_{k,v}}\geq \alpha_0}} \widetilde{\rho_k}^s = \frac{S_{\mathscr{V}}(j,w)}{\rho_j^s}\leq M.$$ Thus, by Theorem \ref{thm:main}, there is a $(1/s)$-H\"older map $g$ with H\"older constant $H_g \lesssim_{s,\Lambda,\xi_1,\xi_2} \widetilde{r_0}(1+M)$ such that $g([0,1])$ contains $\bigcup_{k\geq 0} \widetilde V_{k}.$ Because the algorithm in \S\ref{sec:curve} works locally in the image, $\dist(x,V_j)\lesssim_{\xi_1,\xi_2}A^* \rho_jr_0 \ll_{\xi_1,\xi_2} \Lambda\rho_jr_0$, and $r<\rho_jr_0$, we can guarantee that $g([0,1])$ contains $f([0,1])\cap B(x,r)$ provided that $\Lambda$ is sufficiently large. Therefore, $$\Haus^s(\Gamma\cap B(x,r)) \leq H_g^s\, \Haus^1([0,1]) \lesssim_{s,\Lambda,\xi_1,\xi_2,} \widetilde{r_0}(1+M) \lesssim_{s,C^*,\xi_1,\xi_2,M} (\rho_jr_0)^s \lesssim_{s,C^*,\xi_1,\xi_2,M} r^s, $$ where the final inequality holds because $\rho_{j+1}r_0\leq r$.\end{proof}

\section{Lipschitz parameterization}\label{sec:Lip}
Using the method above, we obtain the following refinement of the sufficient half of the Analyst's TST in Hilbert space, which is originally due to Jones \cite{Jones-TST} in the Euclidean case and due to Schul \cite{Schul-Hilbert} in the infinite-dimensional case.

\begin{prop}[Sufficient half of the Analyst's Traveling Salesman with Nets]\label{prop:2-to-1Lip} Assume that $X=l^2(\R)$ or $X=\R^N$ for some $N\geq 2$. Let $\mathscr{V}=(V_k,\rho_k)_{k\geq 0}$ be a sequence of finite sets $V_k$ in $X$ and numbers $\rho_k>0$ that satisfy properties (V0)--(V5) defined in \textsection\ref{sec:flat}. If
\begin{equation} S_{\mathscr{V}} := \sum_{k=0}^\infty \sum_{v\in V_k} \alpha_{k,v}^2 \rho_k < \infty,\end{equation}
then for every $x_0\in V_0$, we can find a sequence of piecewise linear maps $F_k:[0,1]\rightarrow X$ with the following properties.
\begin{enumerate}
\item For all $k\geq 0$, $F_k(0)=x_0=F_k(1)$.
\item For all $k\geq 0$, there exists $G_k\subset[0,1]$ such that $F_k|G_k$ is 2-1 and $F_k([0,1]\setminus G_k)$ is a finite set.
\item For all $k\geq 0$, $F_k([0,1])\supset V_k$; for all $x\in V_{k+1}$, $\dist(x,F_k([0,1])) \leq C^*\rho_{k+1}r_0$.
\item For all $k\geq 0$, $\|F_k - F_{k+1}\|_{\infty} \lesssim_{C^*,\xi_2} \rho_{k+1}r_0$.
\item For all $k\geq 0$, the map $F_k$ is Lipschitz with $\Lip(F_k) \lesssim_{C^*,\xi_1,\xi_2} r_0(1+ S_{\mathscr{V}})$.
\end{enumerate} The maps $F_k$ converge uniformly to a Lipschitz map $F:[0,1]\rightarrow X$ whose image contains $\bigcup_{k\geq 0}V_k$, the parameterization $F$ starts and ends at $x_0$ in the sense of (1), and the Lipschitz constant of $F$ satisfies $L\lesssim_{C^*,\xi_1,\xi_2} r_0(1+S_{\mathscr{V}})$.\end{prop}

\begin{proof} Let $\alpha_0=\alpha_1$ (see \eqref{alphabound}), which depends only on $C^*$, $\xi_1$, and $\xi_2$. If $(v,v')\in\Flat(k)$, then $\tau_1(k,v,v') \leq 3\alpha_{k,v}^2$ by Lemma \ref{lem:monotone}. Thus, by definition of $S_{\mathscr{V}}^1$ (see Theorem \ref{thm:main}), \begin{align*} S_{\mathscr{V}}^{1} &= \sum_{k=0}^{\infty}\sum_{(v,v')\in \Flat(k)}\tau_{1}(k,v,v') \rho_k+ \sum_{k=0}^{\infty}\sum_{\substack{v\in V_k\\ \a_{k,v} \geq \a_0}} \rho_k\\ &\leq 6\sum_{k=0}^{\infty}\sum_{\substack{v\in V_k\\ \a_{k,v} < \a_0}}\alpha_{k,v}^2 \rho_k+ \frac{1}{\alpha_0^2}\sum_{k=0}^{\infty}\sum_{\substack{v\in V_k\\ \a_{k,v} \geq \a_0}} \alpha_{k,v}^2\rho_k \leq \frac{1}{\alpha_0^2} S_{\mathscr{V}}.\end{align*} The conclusion now follows from Proposition \ref{prop:2-to-1Holder}.
\end{proof}

\part{Applications and Further Results}

In \S\ref{sec:gmt}, we give an application of the H\"older Traveling Salesman theorem to the geometry of measures. In particular, we obtain sufficient conditions for a pointwise doubling measure in $\R^N$ to be carried by $(1/s)$-H\"older curves, $s>1$. This extends the work \cite{BS1,BS3} by the first author and Schul, which characterizes 1-rectifiable Radon measures in $\R^N$ in terms of geometric square functions. In \S8, we use the method of Part I to obtain a Wa\.{z}ewski type theorem for flat continua, which we described above in \S\ref{sec:wazewski}. Finally, in \S\ref{sec:examples}, we present examples of H\"older curves and of sets that are not contained in any H\"older curve to highlight the rich geometry of sets in $\R^N$ and illustrate the strengths and limitations of our principal results.

\section{Fractional rectifiability of measures}\label{sec:gmt}

One goal of geometric measure theory is to understand the structure of a measure in $\R^N$ through its interaction with families of lower dimensional sets. For an extended introduction, see the survey \cite{Badger-survey} by the first author. In this section, we use the H\"older Traveling Salesman theorem to establish criteria for \emph{fractional rectifiability} of pointwise doubling measures in terms of $L^p$ Jones beta numbers. In particular, we extend part of the recent work of the first author and Schul \cite{BS3} on measures carried by rectifiable curves to measures carried by H\"older curves (see Theorem \ref{thm:Jsp}). The study of fractional (that is, non-integer dimensional) rectifiability of measures was first proposed by Mart\'in and Mattila \cite{MM1993,MM2000} and examined further by the first and third author \cite{BV}.

\subsection{Generalized rectifiability}
Let $\mathcal{A}$ be a nonempty family of Borel sets in $\R^N$ and let $\mu$ be a Borel measure on $\R^N$.
We say that $\mu$ is \emph{carried by} $\mathcal{A}$ if there exists a sequence $(A_i)_{i=1}^\infty$ of sets in $\mathcal{A}$ such that $\mu(\R^N\setminus\bigcup_i A_i)=0$. At the other extreme, we say that $\mu$ is \emph{singular to} $\mathcal{A}$ if $\mu(A)=0$ for all $A\in\mathcal{A}$. If $\mu$ is $\sigma$-finite, then $\mu$ can be uniquely written as the sum of a Borel measure $\mu_\mathcal{A}$ carried by $\mathcal{A}$ and a Borel measure $\mu^\perp_\mathcal{A}$ singular to $\mathcal{A}$ (e.g.~see the appendix of \cite{BV}). These definitions encode several commonly used notions of rectifiability of measures (see \cite{Badger-survey}).

Let $1\leq m\leq N-1$. Let $\mathcal{A}$ denote the family of Lipschitz images of $[0,1]^m$ in $\R^N$. We say that a Borel measure $\mu$ is \emph{$m$-rectifiable} if $\mu$ is carried by $\mathcal{A}$; we say that $\mu$ is \emph{purely $m$-unrectifiable} if $\mu$ is singular to $\mathcal{A}$. A Borel set $E\subset\R^n$ with $0<\Haus^m(E)<\infty$ is called $m$-rectifiable or purely $m$-unrectifiable if $\Haus^m\res E$, the $m$-dimensional Hausdorff measure restricted to $E$, has that property. The classes of 1-rectifiable sets and purely 1-unrectifiable sets are also called \emph{Besicovitch regular} sets and \emph{Besicovitch irregular} sets, respectively, in reference to the pioneering investigations by Besicovitch \cite{Bes28,Bes38} into the geometry of 1-sets in the plane.

\begin{ex} Let $\Gamma_1,\Gamma_2,\dots$ be a sequence of rectifiable curves in $\R^N$ and let $a_1,a_2,\dots$ be a sequence of positive weights. Then the measure $\mu=\sum_i a_i\Haus^1\res \Gamma_i$ is 1-rectifiable. Note that if the closure of $\bigcup_i \Gamma_i$ is $\R^N$ and the weights are chosen so that $\sum_{i} a_i \Haus^1(\Gamma_i)=1$, we get a 1-rectifiable Borel probability measure $\mu$ whose support is $\R^N$. \end{ex}

\begin{ex} Let $C\subset\R$ be  the middle halves Cantor set (formed by replacing $[0,1]$ with $[0,\frac14]\cup[\frac{3}{4},1]$ at iterating). Then $E=C\times C\subset\R^2$ is a Cantor set of Hausdorff dimension one, $0<\Haus^1(E)<\infty$, $E$ is Ahlfors 1-regular in the sense that $$\Haus^1(E\cap B(x,r))\simeq r\quad\text{for all $x\in E$ and $0<r\leq \diam E$},$$ and $E$ is Besicovitch irregular (e.g.~see \cite{MM1993}). In particular, the set $E$ is compact and measure-theoretically one-dimensional, but $E$ is not contained in any rectifiable curve.\end{ex}

\subsection{$L^p$ Jones beta numbers and rectifiability}
Let $\mu$ be a Radon measure on $\R^N$, that is, a locally finite Borel regular measure, let $1\leq m\leq N-1$, let $p>0$, let $x\in\R^N$, let $r>0$, and let $L$ be an $m$-dimensional affine subspace of $\R^N$. We define
\begin{equation}\label{e:beta-p-line} \beta^{(m)}_p(\mu,x,r,L) := \left ( \int_{B(x,r)} \left (\frac{\dist(z,L)}{r}\right )^p \frac{d\mu(z)}{\mu(B(x,r))} \right )^{1/p},\end{equation} \begin{equation}\label{e:beta-p}
 \beta^{(m)}_p(\mu,x,r) := \inf_{L}\beta^{(m)}_p(\mu,x,r,L),\end{equation}
where the infimum is taken over all $m$-planes $L$ in $\R^N$. The quantity $\beta^{(m)}_p(\mu,x,r)$ is called the \emph{$m$-dimensional $L^p$ Jones beta number of $\mu$ in $B(x,r)$}. The $L^p$ Jones beta numbers were introduced by David and Semmes \cite{DS91,DS93} to study quantitative rectifiability of Ahlfors regular sets and boundedness of singular integral operators. The normalization of the measure in \eqref{e:beta-p-line} that we have chosen (i.e.~ dividing by $\mu(B(x,r))$) ensures that $\beta^{(m)}_p(\mu,x,r)\in[0,1]$ and $\beta^{(m)}_p$ is invariant under dilations $T_\lambda(z)=\lambda z$ in the sense that \begin{equation}\beta^{(m)}_p(\mu,x,r)=\beta^{(m)}_p(T_{\lambda}[\mu],\lambda x,\lambda r),\quad T_\lambda[\mu](E)=\mu(\lambda^{-1}E)\end{equation} for all $\mu$, $x\in\R^N$, $r>0$, and $\lambda>0$. By monotonicity of the integral, \begin{equation}\label{beta-double} s\mu(B(y,s))^{1/p}\beta_p^{(m)}(\mu,y,s)\leq r\mu(B(x,r))^{1/p}\beta_p^{(m)}(\mu,x,r)\quad\text{when }B(y,s)\subset B(x,r).\end{equation}

In a pair of papers, Tolsa \cite{Tolsa-n} and Azzam and Tolsa \cite{AT15} characterize $m$-rectifiable Radon measures $\mu$ on $\R^N$ with $\mu\ll\Haus^m$ in terms of $L^2$ Jones beta numbers. The restriction $\mu\ll\Haus^m$ is equivalent to the upper density bound $\limsup_{r\downarrow 0} r^{-m}\mu(B(x,r))<\infty$ $\mu$-a.e. (e.g.~see \cite[Chapter 6]{Mattila}) and implies that the Hausdorff dimension of the measure is at least $m$ (see \cite{measure-dimension}). The proof that \eqref{e:AS-Jones} implies the measure $\mu$ is $m$-rectifiable uses an intricate stopping time argument in conjunction with David and Toro's \emph{Reifenberg algorithm for sets with holes} \cite{DT} to construct bi-Lipschitz images of $\R^m$ inside $\R^N$ that carry $\mu$. For related developments, see \cite{ENV,Ghinassi}.

\begin{thm}[see {\cite{Tolsa-n}, \cite{AT15}}] \label{t:AT} Let $\mu$ be a Radon measure on $\R^N$. Assume that \begin{equation}\label{e:ubound} 0<\limsup_{r\downarrow 0} \frac{\mu(B(x,r))}{r^m}<\infty\quad\text{for $\mu$-a.e. $x\in\R^N$}.\end{equation} Then $\mu$ is $m$-rectifiable if and only if \begin{equation}\label{e:AS-Jones} \int_0^1 \beta^{(m)}_2(\mu,x,r)^2\,\frac{\mu(B(x,r))}{r^m}\,\frac{dr}{r}<\infty\quad\text{for $\mu$-a.e. $x\in\R^N$.}\end{equation}\end{thm}

In \cite{BS3}, the first author and Schul characterize $1$-rectifiable Radon measure $\mu$ on $\R^N$ in terms of $L^p$ Jones beta numbers and the lower density $\liminf_{r\downarrow 0} r^{-1}\mu(B(x,r))$. In contrast with Theorem \ref{t:AT}, the main theorem in \cite{BS3} does not require an \emph{a priori} relationship between the null sets of $\mu$ and $\Haus^1$, nor a bound on the Hausdorff dimension of $\mu$. To lighten the notation, we present Badger and Schul's theorem for pointwise doubling measures and refer the reader to \cite[Theorem A]{BS3} for the full result. Although the classes of measures satisfying \eqref{e:ubound} and \eqref{e:point-double} have no direct relationship with each other, \emph{a posteriori} an $m$-rectifiable measure satisfying \eqref{e:ubound} also satisfies \eqref{e:point-double}. The proof that \eqref{e:Jones-BS} implies the measure $\mu$ is 1-rectifiable uses a technical extension of the sufficient half of the Analyst's Traveling Salesman theorem. See \cite[Proposition 3.6]{BS3}.

\begin{thm}[see {\cite[Theorem E]{BS3}}] \label{t:BS} Let $\mu$ be a Radon measure on $\R^n$ and let $p\geq 1$. Assume that $\mu$ is pointwise doubling in the sense that \begin{equation}\label{e:point-double} \limsup_{r\downarrow 0} \frac{\mu(B(x,2r))}{\mu(B(x,r))} <\infty\quad\text{for $\mu$-a.e. $x\in\R^N$.}\end{equation} Then $\mu$ is 1-rectifiable if and only if \begin{equation}\label{e:Jones-BS}\int_0^1 \beta^{(1)}_p(\mu,x,r)^2\,\frac{r}{\mu(B(x,r))}\,\frac{dr}{r}<\infty\quad\text{for $\mu$-a.e. $x\in\R^N$.}\end{equation}
\end{thm}

\subsection{Sufficient conditions for fractional rectifiability}

The following theorem is an application of the H\"older Traveling Salesman theorem and generalizes the ``sufficient half'' of Theorem \ref{t:BS} (also see \cite[Theorem A]{BV}). The exponents $p$ and $q$ in the H\"older case ($s>1$) are less restrictive than in the Lipschitz case ($s=1$).

\begin{thm} \label{thm:Jsp} Let $\mu$ be a Radon measure on $\R^N$, let $s>1$, and let $p,q>0$. Then
\[
\mu\res\left\{x\in\R^n: \int_0^1 \beta^{(1)}_p(\mu,x,r)^q\frac{r^s}{\mu(B(x,r))}\,\frac{dr}{r}<\infty\text{ and } \limsup_{r\downarrow 0}\frac{\mu(B(x,2r))}{\mu(B(x,r))}<\infty \right\}\]
is carried by $(1/s)$-H\"older curves.
\end{thm}

At the core of the proof of Theorem \ref{thm:Jsp} is the following lemma.

\begin{lem}\label{lem:Jsp}
Let $\mu$ be a Radon measure in $\R^N$, and let $s>1$ and $p,q>0$ be fixed. Given $x_0\in\R^N$ and parameters $M>0$, $\theta>0$, and $P>0$, let $A$ denote the set of points $x\in B(x_0,1/2)$ such that
\begin{equation}\label{eq:Jsp1} \int_0^1 \beta^{(1)}_p(\mu,x,r)^q\frac{r^s}{\mu(B(x,r))}\,\frac{dr}{r} \leq M,\end{equation} \begin{equation}\label{double-3} \mu(B(x,2r)) \leq P\mu(B(x,r))\text{ for all }r\in (0,1],\end{equation} and let $A'$ denote the set of points in $A$ such that
\begin{equation}\label{eq:Jsp3}\mu(A\cap B(x,r)) \geq \theta\mu(B(x,r))\text{ for all }r\in (0,1].\end{equation}
Then $A'$ is contained in a $(1/s)$-H\"older curve $\Gamma=f([0,1])$ with H\"older constant depending on at most $N$, $s$, $p$, $q$, $M$, $P$, $\theta$, and $\mu(A)$.
\end{lem}

\begin{proof}
Let $\{A_k'\}_{k\geq 0}$ be a nested sequence of $2^{-k}$-nets in $A'$, so that the sets $V_k\equiv A_k'$ and scales $\rho_k=2^{-k}$ satisfy conditions (V0)--(V4) of \S\ref{sec:prelim} with parameters $r_0=1$, $C^*=2$, $\xi_1=\xi_2=1/2$. Note that $$A^*=\frac{C^*}{1-\xi_2}=4\quad\text{and}\quad 30A^*=120.$$ By \eqref{eq:Jsp1},
\begin{equation}\begin{split}\label{eq:meas3}
M \mu(A) &\geq \int_A\int_0^1 \beta^{(1)}_p(\mu,x,r)^q \frac{r^s}{\mu(B(x,r))} \frac{dr}{r} d\mu(x) \\ &= \sum_{k=9}^\infty \int_{2^{-(k+1)}}^{2^{-k}} (512r)^s \int_A \frac{\beta^{(1)}_p(\mu,x,512r)^q} {\mu(B(x,512r))}d\mu(x)\frac{dr}{r}
\end{split}\end{equation} where in the second line we used the change of variables $r\mapsto 512r$ (note $512=2^9$) and Tonelli's theorem. Now, the open balls $\{B(y,2^{-(k+1)}):y\in A_k'\}$ are pairwise disjoint, because the points in $A'_k$ are separated by distance at least $2^{-k}$. Thus, \begin{equation}\label{eq:combo1} M \mu(A) \geq \sum_{k=9}^{\infty}\int_{2^{-(k+1)}}^{2^{-k}}r^s\sum_{y\in A_k'}\underbrace{\int_{A\cap B(y,2^{-(k+1)})} \frac{\beta^{(1)}_p(\mu,x,512r)^q }{\mu(B(x,512r))} d\mu(x)}_{I(k,y,r)} \frac{dr}{r}.\end{equation}

Next, we  bound $I(k,y,r)$ from below. Fix $k\geq 9$, $y\in A_k'$, and $r\in[2^{-(k+1)},2^{-k}]$. Suppose that $x\in A\cap B(y,2^{-(k+1)})$. Then \begin{equation} \mu(B(x,512r)) \leq P^2 \mu(B(x,128 r))\leq P^2 \mu(B(y,129r))\leq P^2 \mu(B(y,255\cdot 2^{-k}))\end{equation} by \eqref{double-3}. Since $B(y,255\cdot 2^{-k})\subset B(x,256\cdot 2^{-k})\subset B(x,512r)$, it follows that \begin{equation}\begin{split} \beta^{(1)}_p(y,255\cdot 2^{-k}) &\leq \left(\frac{512r}{255\cdot 2^{-k}}\right)\left(\frac{\mu(B(x,512r))}{\mu(B(y,255\cdot 2^{-k}))}\right)^{1/p} \beta_p^{(1)}(\mu,x,512r)\\ &\leq 3P^{2/p}\beta_p^{(1)}(\mu,x,512r) \end{split}\end{equation} by \eqref{beta-double}. Hence \begin{equation}I(k,y,r) \geq 3^{-q}P^{-2-2q/p}\frac{\beta_p^{(1)}(\mu,y,255\cdot 2^{-k})^q}{\mu(B(y,255\cdot 2^{-k}))}\int_{A\cap B(y,2^{-(k+1)})} d\mu(x). \end{equation} Invoking doubling again, $\mu(B(y,255 \cdot 2^{-k}))\leq P^{9} \mu(B(y,2^{-(k+1)}))$. Thus, by \eqref{eq:Jsp3}, \begin{equation}\frac{1}{\mu(B(y,255\cdot 2^{-k}))} \int_{A\cap B(y,2^{-(k+1)})} d\mu(x) \geq P^{-9}\frac{\mu(A\cap B(y,2^{-(k+1)}))}{\mu(B(y,2^{-(k+1)}))}\geq P^{-9}\theta.\end{equation} Therefore, \begin{equation} \label{eq:combo2} I(k,y,r) \geq 3^{-q}P^{-11-2q/p}\theta \, \beta_p^{(1)}(\mu,y,255\cdot 2^{-k})^q\end{equation}

Combining \eqref{eq:combo1} and \eqref{eq:combo2}, we obtain \begin{equation} 3^q P^{11+2q/p}\theta^{-1}M\mu(A) \geq \sum_{k=9}^\infty \left(\int_{2^{-(k+1)}}^{2^{-k}} r^s\frac{dr}{r}\right) \sum_{y\in A_k'} \beta_p^{(1)}(\mu,y,255\cdot 2^{-k})^q.\end{equation} In particular, we conclude that \begin{equation}\sum_{k=9}^\infty  \sum_{y\in A_k'} \beta_p^{(1)}(\mu,y,255\cdot 2^{-k})^q\, 2^{-ks} \leq \frac{s}{1-2^{-s}}3^q P^{11+2q/p}\theta^{-1}M\mu(A)<\infty \end{equation} For each $k\geq 9$ and $y\in A'_k$, let $\ell_{k,v}$ be any line such that \begin{equation}\beta_p^{(1)}(\mu,y,255\cdot 2^{-k},\ell_{k,v})\leq 2\beta_p^{(1)}(\mu,y,255\cdot 2^{-k}).\end{equation} We will now bound the distance of points in $A'_{k+1}\cap B(y,120\cdot 2^{-k})$ to $\ell_{k,v}$. Fix any point $z\in A'_{k+1}\cap B(y,120\cdot 2^{-k})$ and let $t2^{-{k+1}}=\dist(z,\ell_{k,v})$. Then \begin{equation}\begin{split} \beta_p(\mu,y,255\cdot 2^{-k},\ell_{k,v})^q &\geq\left(\frac{\frac{1}{2}t2^{-(k+1)}}{255\cdot 2^{-k}}\right)^q\left(\frac{\mu(B(z,\frac{1}{2}t2^{-(k+1)}))}{\mu(B(y,255\cdot 2^{-k}))}\right)^{q/p} \\ &\geq \left(\frac{t}{1020}\right)^q P^{-(q/p)\log_2(1920/t)} \geq \left(\frac{t}{1920}\right)^{q+(q/p)\log_2(P)},
\end{split}  \end{equation} where in the second line we used doubling of $\mu$ at $z$. It follows that \begin{equation}\alpha_{k,v}:= 2^{k+1}\sup_{z\in A'_{k+1}\cap B(y,120\cdot 2^{-k})}
\dist (z,\ell_{k,v}) \leq C(p,q,P) \beta_p(\mu,y,255\cdot 2^{-k},\ell_{k,v})^\eta,\end{equation} where $\eta[q+(q/p)\log_2(P)] = q$. Therefore, all together, \begin{equation} \sum_{k=9}^\infty \sum_{y\in A'_k} \alpha_{k,v}^{q+(q/p)\log_2(P)} 2^{-ks} \leq C(s,p,q,M,P,\theta,\mu(A))<\infty.\end{equation} Finally, by Corollary \ref{cor:nets2}, the set $A'$ is contained in the Hausdorff limit of $A_k'$ and this is contained in a $(1/s)$-H\"older curve $\Gamma=f([0,1])$ with H\"older constant depending on at most $s$, $p$, $q$, $M$, $P$, $\theta$, and  $\mu(A)$. \end{proof}

Theorem \ref{thm:Jsp} follows from countably many applications of Lemma \ref{lem:Jsp} and a standard density theorem for Radon measures in $\R^N$. See the proof of \cite[Theorem 6.7]{BV}, where a similar argument is employed. We leave the details to the reader.

\section{H\"older parameterization of flat continua}\label{sec:flatcontinua}

The goal of this section is to prove Proposition \ref{thm:flat}, which we now restate.

\begin{prop}\label{thm:flat-statedagain}
There exists a constant $\beta_1 \in (0,1)$ such that if $s>1$ and $E\subset\R^N$ is compact, connected, $\Haus^s(E)<\infty$, $E$ is lower Ahlfors $s$-regular with constant $c$, and
\begin{equation}\label{e:mv-statedagain}
 \beta_{E}\left(\overline{B}(x,r)\right) \leq \beta_1\quad \text{for all }x\in E\text{ and }0<r\leq \diam E, \end{equation}
then $E=f([0,1])$ for some injective $(1/s)$-H\"older continuous map $f:[0,1]\to\R^N$ with H\"older constant $H \lesssim_{s} c^{-1}\mathcal{H}^s(E)(\diam E)^{1-s}$.
\end{prop}

The proposition is trivial if $E$ is a singleton (in which case \eqref{e:mv-statedagain} is vacuous). Thus, we may assume that $E\subset\R^N$ is a \emph{continuum}; that is, $E$ is compact, connected, and contains at least two points. Furthermore, as the hypothesis and the conclusion are scale-invariant, we may assume without loss of generality that $\diam E=1$. To complete the proof of the proposition, we mimic the proof of Theorem \ref{thm:main}, but with a few modifications. In \S\ref{sec:flatcontinua_algo}, we perform a simplified version of the algorithm in \textsection\ref{sec:curve}. Then, in \textsection\ref{sec:massflat}, we establish an upper bound on $\mathcal{M}_s([0,1])$ in terms of $c^{-1}\Haus^s(E)$, which fills the role that Proposition \ref{prop:mass} played for Theorem \ref{thm:main}. Equipped with this mass bound, the proof of Proposition \ref{thm:flat-statedagain} essentially follows by repeating the proof of Theorem \ref{thm:main} \emph{mutatis mutandis}.

At the core of the proof of Proposition \ref{thm:flat-statedagain} is the following property, which is satisfied by continua that are sufficiently flat at all locations and scales. We defer a proof of Lemma \ref{lem:flatcont} to \S\ref{sec:prop-continua}. An adventurous reader may wish to supply their own proof.

\begin{lem}\label{lem:flatcont}
Suppose that $E\subset\R^N$ is a continuum satisfying
\[ \beta_E\left(\overline{B}(x,r)\right) \leq 2^{-11} \qquad\text{for all $x\in E$ and $0<r\leq \diam E$}.\]
Then for all distinct $x,y \in E$ and for all $z\in [x,y]$, there exists $z' \in E$ such that $\pi_{\ell_{x,y}}(z') = z$ and $|z-z'|\leq 2^{-4}|x-y|$, where $\ell_{x,y}$ denotes the line containing $x$ and $y$.
\end{lem}

\subsection{Traveling Salesman algorithm for flat continua}\label{sec:flatcontinua_algo}
Fix a constant $\beta_1>0$ (small) to be specified later. Let $E\subset\R^N$ be a continuum satisfying the hypothesis of Proposition \ref{thm:flat-statedagain}. Without loss of generality, we assume that $\diam E=1$. Pick $x_0$ and $y_0$ such that $|x_0-y_0|=1$, set $r_0= 1$, $C^*=2$, and $\xi_1=\xi_2=1/2$. Then $$A^*=\frac{C^*}{1-\xi_2}=4,\quad 14A^*=56,\quad 30A^*=120.$$ Furthermore, the scales $\rho_k = 2^{-k}$ satisfy (V0).

Set $V_0=\{x_0,y_0\}$, $\alpha_{0,x_0}=4\beta_E\left(\overline{B}(x_0,1)\right)$ and $\alpha_{0,y_0}=4\beta_E\left(\overline{B}(y_0,1)\right)$, and let $\ell_{0,x_0}$ and $\ell_{0,y_0}$ be best fitting lines corresponding to $\beta_E$ on the closed balls $\overline{B}(x_0,1)$ and $\overline{B}(y_0,1)$, respectively.

Suppose that for some $k\geq 0$ and for all $0\leq j\leq k$, we have defined sets $V_j$, numbers $\alpha_{j,v}\geq 0$ for all $v\in V_j$, and lines $\ell_{j,v}$ for all $v\in V_j$ satisfying (V5). Choose $V_{k+1}$ to be any maximal $2^{-(k+1)}$-separated subset of $E$ such that $V_{k+1}\supset V_k$. For each $v\in V_{k+1}$, set $$\alpha_{k+1,v}:=\frac{2r_{k+1}}{2^{-(k+2)}}\beta_E\left(\overline{B}(v,r_{k+1})\right)\leq 480\beta_E\left(\overline{B}(v,r_{k+1})\right),$$ where $r_{k+1}:=\min\{120\cdot 2^{-(k+1)},1\}$, and let $\ell_{k+1,v}$ be a best fitting line corresponding to $\beta_E\left(\overline{B}(v,r_k)\right)$. The reader may check that the sequence of sets $(V_k)_{k=0}^\infty$ satisfy properties (V1)--(V5) in \ref{sec:flat}.

We now specify $\alpha_0=512\beta_1\leq 1/16$. This ensures that $\alpha_{k,v}<\alpha_0$ for all $k\geq 0$ and $v\in V_k$. Moreover, $\beta_1$ is sufficiently small that we may invoke Lemma \ref{lem:flatcont} for $E$.
With $\alpha_0$ fixed, carry out a modified version of the algorithm in \S\ref{sec:curve}, in which (P4), (P6), and (P7) are replaced by: \begin{enumerate}
\item[(P4')] $f_k|\bigcup{\mathscr{E}_k}$ is one-to-one.
\item[(P6')] For each $I\in\mathscr{N}_k\cup\mathscr{F}_k$, the image $f_k(I)\in V_k$. For every $v\in V_k$, there exists a unique interval $I\in\mathscr{N}_k\cup\mathscr{F}_k$ such that $v=f_k(I)$.
\item[(P7')] If $\mathscr{I}_k=\mathscr{E}_k\cup\mathscr{B}_k\cup\mathscr{N}_k\cup\mathscr{F}_k$ is enumerated according to the natural order in $[0,1]$, say $\mathscr{I}_k=\{I_1,\dots,I_{2l+1}\}$, then the intervals alternate between elements of $\mathscr{N}_k\cup \mathscr{F}_k$ and $\mathscr{E}_k$. (Thus, the family $\mathscr{B}_k=\emptyset$.) Moreover, $\card \mathscr{N}_k=2$ and $I_1,I_{2l+1}\in\mathscr{N}_k$. The vertices $f_k(I_1)$ and $f_{k}(I_{2l+1})$ are 1-sided terminal in $V_{k}$, while each other vertex $f_{k}(I_{2j+1})$ is non-terminal in $V_{k}$ for all $1\leq j\leq l-1$.
\end{enumerate}
We now sketch some steps in the modified algorithm.

\subsubsection{Step 0} Partition $[0,1]=[0,1/3]\cup (1/3,2/3) \cup [2/3,1]$ and assign  $$\mathscr{E}_0=\{(1/3,2/3)\},\quad \mathscr{B}_0=\emptyset,\quad \mathscr{N}_0=\{[0,1/3],  [2/3,1]\},\quad  \mathscr{F}_0=\emptyset.$$ Also set $f_0([0,1/3])=x_0$ and $f_0([2/3,1])=y_0$, and define $f_0|(1/3,2/3)$ to be the restriction of the affine map which interpolates between $x_0$ and $y_0$. Verifying properties (P1), (P2), (P3), (P4'), (P5), (P6'), and (P7') is straightforward. We omit the details.

\subsubsection{Induction Step} Suppose that $\mathscr{E}_k$, $\mathscr{B}_k$, $\mathscr{N}_k$, $\mathscr{F}_k$, and $f_k$ have been defined and satisfy properties (P1), (P2), (P3), (P4'), (P5), (P6'), and (P7').

By (P7') and the induction assumption, the procedure in \S\ref{sec:bridges} is moot, because there are no bridge intervals.

Follow the procedure in \S\ref{sec:flatedges} for $I\in\mathscr{E}_k$ as written, except assign all closed subintervals $\mathscr{N}_{k+1}(I)\cup\mathscr{F}_{k+1}(I)$ generated in $I$ to $\mathscr{F}_{k+1}(I)$ instead of $\mathscr{N}_{k+1}(I)$. Also set $\mathscr{N}_{k+1}(I)=\emptyset$. Below, we check that $\mathscr{B}_{k+1}(I)=\emptyset$; see Lemma \ref{lem:TSPflat}.

Because $\alpha_{k,v}<\alpha_0$ for all $v\in V_k$, the procedure in \S\ref{sec:noflatedges} is moot.

Follow the procedure in \S\ref{sec:fixed} as written.

Replace the procedure in \S\ref{sec:flatpoints} as follows. \begin{quotation}
By property (P7'), only the intervals in $\mathscr{I}_k$ containing $0$ and $1$ belong to $\mathscr{N}_k$ and $f_k$ maps each of them onto a 1-sided terminal vertex in $V_k$.
Let $I$ be the interval in $\mathscr{N}_k$ containing $0$ and let $v=f_k(I)$. Let $(v,v')\in\Flat(k)$ be the unique flat pair with first coordinate $v$. Choose an orientation for $\ell_{k,v}$ so that $[v,v']$ lies on the right side of $v$. Enumerate the points in $V_{k+1}\cap B(f_k(I), C^*\rho_{k+1}r_0)$ on the left side of $v$ (including $v$) as $v_{l}, v_{l-1},\dots, v_1=v$, starting at the leftmost vertex and working right. The construction splits into three cases.

\emph{Case 1a.} If $l=1$, then no new points appeared to the left of $v$ and we set $\mathscr{N}_{k+1}(I)=I$ and $f_{k+1}|I=f_k|I$. Set $\mathscr{E}_{k+1}(I)=\mathscr{B}_{k+1}(I)=\mathscr{F}_{k+1}(I)=\emptyset$.

\emph{Case 1b.} If $l=2$, then one new point appeared to the left of $v$, the new point is 1-sided terminal in $V_{k+1}$ and $v$ is non-terminal in $V_{k+1}$. Subdivide $I=[0,a]=[0,\frac13a]\cup(\frac13a,\frac23a)\cup[\frac23a,a]$, set $\mathscr{N}_{k+1}(I)=[0,\frac13a]$, $\mathscr{E}_{k+1}(I)=(\frac13a,\frac23a)$, $\mathscr{F}_{k+1}(I)=[\frac23a,a]$, and $\mathscr{B}_{k+1}=\emptyset$. Define $f_{k+1}|I$ by assigning $f_{k+1}([0,\frac13a])=v_2$, $f_{k+1}|(\frac13a,\frac23a)$ to be the restriction of the affine map that interpolates between $v_2$ and $v_1$, and $f_{k+1}([\frac23a,a])=v_1$.

\emph{Case 1c.} If $l\geq 3$, then subdivide $I$ into $2l-1$ intervals, alternating between closed and open intervals. Assign the first interval to $\mathscr{N}_{k+1}(I)$ and the subsequent closed intervals to $\mathscr{F}_{k+1}(I)$. Assign the open intervals in $\mathscr{E}_{k+1}(I)$ and set $\mathscr{B}_{k+1}(I)=\emptyset$. The map $f_{k+1}|I$ is the piecewise linear map starting at $v_l$, connecting $v_l$ to $v_{l-1}$, \dots, connecting $v_2$ to $v_1$, which is constant on the intervals in $\mathscr{N}_{k+1}(I)\cup\mathscr{F}_{k+1}(I)$ and constant speed on the intervals in $\mathscr{E}_{k+1}(I)$.

Carry out a similar construction for the interval $J$ in $\mathscr{N}_k$ containing 1, but modified so that $\mathscr{N}_{k+1}(J)$ contains only one interval and that interval contains $1$.
\end{quotation}

Because $\alpha_{k,v}<\alpha_0$ for all $v\in V_k$, the procedure in \S\ref{sec:noflatpoints} is not used.

\begin{lem}\label{lem:TSPflat} For all $k\geq 0$, $\mathscr{B}_k=\emptyset$. Moreover, for all $I\in\mathscr{E}_k$, $\diam f_k(I) < 3 \cdot 2^{-k}$.
\end{lem}

\begin{proof} We need to check that, for flat continua, the procedure in \S\ref{sec:flatedges} does not generate bridge intervals. Given $I\in\mathscr{E}_k$, let $v$ and $v'$ denote the endpoints of $f_k(I)$. Choose an orientation of $\ell_{k,v}$ so that $v'$ lies to the right of $v$. Enumerate $V_{k+1}(v,v')=\{v_1,\dots,v_n\}$, where $v_1=v$, $v_n=v'$, and $v_{i+1}$ is the first point to the right of $v_i$ for all $1\leq i\leq n-1$. Suppose for contradiction that $\mathscr{B}_{k+1}(I)\neq\emptyset$. Then  $$|v_{j+1}-v_j|\geq 14 A^* \rho_{k+1}= 56\cdot 2^{-(k+1)}\quad \text{for some $1\leq j\leq n-1$.}$$ Let $x=(v_j+v_{j+1})/2$ denote the midpoint between $v_j$ and $v_{j+1}$. By Lemma \ref{lem:flatcont}, there exists $y\in E$ such that $|y-x|\leq (1/16)|v_{j+1}-v_j|$. Thus, $$\dist(y, V_{k+1}) \geq \dist(x,V_{k+1}) - |y-x| \geq \frac{7}{16}|v_{j}-v_{j-1}| > 24 \cdot 2^{-(k+1)}.$$ This contradicts our assumption that $V_{k+1}$ is a maximal $2^{-(k+1)}$-separated set for $E$. Therefore, $\mathscr{B}_{k+1}(I)=\emptyset$ for all $I\in\mathscr{E}_k$. The only other instances in the algorithm where bridge intervals could be generated are in the procedures in \S\S \ref{sec:noflatedges} and \ref{sec:noflatpoints}. However, since $\alpha_{k,v}<\alpha_0$ for all $k\geq 0$ and $v\in V_k$, these procedures were never used.

Similarly, suppose to get a contradiction that there exists $I\in \mathscr{E}_k$ such that $\diam f_k(I) \geq 3 \cdot 2^{-k}$. Let $v$ and $v'$ denote the endpoints of $f_k(I)$, and let $x=(v+v')/2$ denote their midpoint. By Lemma \ref{lem:flatcont}, there exists $y\in E$ such that $|y-x|<(1/16)|v-v'|$. Then $$\dist(y,V_k) \geq \dist(x,V_k) - |x-y| \geq \frac{7}{16}|v-v'| \geq \frac{21}{16} 2^{-k}.$$ This contradicts our assumption that $V_{k+1}$ is a maximal $2^{-(k+1)}$-separated set for $E$.
\end{proof}

Verifying properties (P1), (P2), (P3), (P4'), (P5), (P6'), and (P7') for $\mathscr{E}_{k+1}$, $\mathscr{B}_{k+1}$, $\mathscr{N}_{k+1}$, $\mathscr{F}_{k+1}$, and $f_{k+1}$ is again routine. We leave the details to the reader.

\subsection{Mass estimate}\label{sec:massflat} Let $\mathcal{M}_s([0,1])$ be defined as in \S\ref{sec:mass}.

\begin{lem}\label{lem:flat3}
$\mathcal{M}_s([0,1]) \leq 48^s c^{-1}\mathcal{H}^s(E)$.\end{lem}

\begin{proof} Fix a finite tree $T$ over $[0,1]$ of depth $m$ (see \S\ref{sec:mass}) and suppose that
\[ \partial T = \{(k_1',J_1),(k_1,I_1),\dots,(k_l',J_l),(k_l,I_l),(k_{l+1}',J_{l+1})\},\]
enumerated according to the orientation of $[0,1]$ so that
\[ \{I_1,\dots,I_l\} \subseteq \bigcup_{k\geq 0}\mathscr{E}_k \qquad\text{and}\qquad\{J_1,\dots,J_{l+1}\} \subseteq \bigcup_{k\geq 0}(\mathscr{N}_k\cup\mathscr{F}_k).\]
The first interval $J_1\in\mathscr{N}_{k_1'}$ and the last interval $J_{l+1}\in\mathscr{N}_{k_{l+1}'}$, since they contain $0$ and $1$, respectively. The remaining intervals $J_i\in\mathscr{F}_{k'_i}$, because they do not contain $0$ or $1$.

For each $1\leq i\leq l$, let $x_i$ denote the midpoint of $f_{k_i}(I_i)$.

\begin{claim}\label{lem:edgesfar} One one hand, the set $E\cap B(x_i,(3/16)2^{-k_i})$ is nonempty for all $1\leq i\leq l$. On the other hand, the family $\left\{E\cap B(x_i, (1/4) 2^{-k_i}): 1\leq i\leq l\right\}$ is pairwise disjoint.\end{claim}

\begin{proof} Given $1\leq i\leq l$, let $v_i$ and $v_{i}'$ denote the endpoints of $f_{k_i}(I_i)$. By Lemma \ref{lem:TSPflat}, $$|v_i-v_i'| < 3\cdot 2^{-k_i}.$$ Hence there exists $z_i\in E\cap B(x_i, (1/16)|v_i-v_i'|) \subset E \cap B(x_i, (3/16)\cdot 2^{-k_i})$ by Lemma \ref{lem:flatcont}.

Suppose in order to reach a contradiction that there exists $$z\in E\cap B(x_i, (1/4) 2^{-k_i})\cap B(x_j, (1/4) 2^{-k_j})$$ for some $i\neq j$ with $0\leq k_i \leq k_j\leq m$.

\emph{Case 1}. Suppose that $k_j \geq k_i + 3$. Then $$|v_j-x_i| \leq |v_j-x_j|+|x_j-z|+|z-x_i| < \frac32\cdot 2^{-k_j} + \frac{1}{4} 2^{-k_j} + \frac14 \cdot 2^{-k_i}< \frac12 \cdot 2^{-k_i}$$ and similarly for $v_j'$.  Because $\{v_j,v_j'\}\subset B(x_i,\frac{1}{2} 2^{-k_i})$, we conclude that $v_j$ and $v_j'$ lie between $v_i$ and $v_i'$ with respect to the linear ordering of $V_i\cap B(v_i,120\cdot 2^{-k_i})$. It follows that $I_j$ is contained in $I_i$, but $I_j\neq I_i$. This contradicts the assertion that $I_i\in\partial T$.

\emph{Case 2.} Suppose that $k_j \leq k_i + 2$. Then $$|v_i-v_j| \leq |v_i-x_i|+|x_i-z|+|z-x_j|+|x_j-v_j|< \frac32 2^{-k_i} + \frac14 \cdot 2^{-k_i} + \frac14 \cdot 2^{-k_j}
+ \frac32 2^{-k_j}<8\cdot 2^{-k_j},$$ $$|v_i'-v_j| \leq |v_i'-v_i|+|v_i-v_j| < 3\cdot  2^{-k_i} + 8\cdot 2^{-k_j} \leq   20 \cdot 2^{-k_j}.$$ In particular, $v_i$, $v_i'$, $v_j$, $v_j'$ belong to $V_j\cap B(v_j,120\cdot 2^{-k_j})$, which is linearly ordered by Lemma \ref{lem:var}, where $v_j$ and $v_{j+1}$ are consecutive points. Assume that $v_i$ and $v_i'$ both lie on the left or the right side of $[v_j,v_j']$, say without loss of generality that the appear from left to right as $v_j,v_j',v_i,v_i'$. Then $$|x_i-v_j'| \geq \frac{1}{2} |v_i'-v_i| \geq \frac{1}{2}\cdot 2^{-k_i} >\frac14\cdot 2^{-k_i}.$$ It follows that $B(v_j,\frac14\cdot{2^{-k_j}}) \cap B(v_i,\frac14\cdot2^{-k_i})$ is empty, which contradicts our assumption. Thus, one of $v_i$ or $v_i'$ lies on the left side of $[v_j,v_j']$ and the other lies on the right side. Then $I_j\subset I_i$. If $k_j \geq k_i+1$, then we reach the same contradiction as in Case 1. If $k_j=k_i$, then it follows that $I_j=I_i$, which contradicts our assumption that $i\neq j$.
\end{proof}

We now continue with the proof of Lemma \ref{lem:flat3}. By Claim \ref{lem:edgesfar}, we can find balls $B(z_i, (1/16) 2^{-k_i})$ centered in $E$ for all $1\leq i\leq l$, which are pairwise disjoint. Moreover, because $E$ is lower Ahlfors regular, we have $c[(1/16) 2^{-k_i}]^s \leq \Haus^s(E\cap B(z_i, (1/16) 2^{-k_i}))$ for all $1\leq i\leq l$. Therefore, by Lemma \ref{lem:TSPflat} and additivity of measures on disjoint sets, \begin{equation*}
 \sum_{(k,I)\in\partial T} (\diam{f_k(I)})^s = \sum_{i=1}^l (\diam {f_{k_i}(I_i)})^s \leq \sum_{i=1}^l (3\cdot 2^{-k_i})^s \leq 48^s c^{-1}\Haus^s(E).\end{equation*} Because $T$ was an arbitrary finite tree over $[0,1]$, we obtain the corresponding inequality for the total mass $\mathcal{M}_s([0,1])$.
\end{proof}

\begin{cor}\label{cor:massest}
If $k\geq 0$, $I\in \mathscr{I}_k$, and $a$ is an endpoint of $I$, then
\[ \mathcal{M}_s(k,I) \leq 48^s c^{-1}\,\mathcal{H}^s(E\cap B(f_k(a),6\cdot 2^{-k})).\]
\end{cor}

\begin{proof} Let $T$ be a finite tree over $(k,I)$. By Lemma \ref{lem:TSPflat}, the image of any edge interval in $\partial T$ is contained in a ball centered at $f_k(a)$ of radius at most $$3\cdot 2^{-k}+3\cdot 2^{-(k+1)}+\dots=6\cdot 2^{-k}.$$ The conclusion follows by repeating the proof of Lemma \ref{lem:flat3}. \end{proof}

\subsection{Proof of Proposition {\ref{thm:flat-statedagain}}}

With Lemma \ref{lem:flat3} in hand, follow the proof of Theorem \ref{thm:main} in \S\ref{sec:proof1}, \emph{mutatis mutandis}. Construct families of intervals $\mathscr{E}_k'$, $\mathscr{N}_k'$, and $\mathscr{F}_k'$ as in \S\ref{sec:proof1}. We are free to specify the following additional constraints.
\begin{itemize}
\item If $I \in \mathscr{E}_k$, say $I \in \mathscr{E}_k(I_0)$ for some $I_0 \in \mathscr{E}_{k-1}\cup\mathscr{N}_{k-1}$, then the corresponding interval $I' \in \mathscr{E}_k'$ satisfies
\begin{equation}\label{eq:bound}
\diam{I'} = \frac{\mathcal{M}_s(k,I)}{\mathcal{M}_s([0,1])} + \frac{1}{\card(\mathscr{E}_{k}(I_0))}\left ( \diam{I_0} - \sum_{J \in \mathscr{I}_{k}(I_0)}\frac{\mathcal{M}_s(k,J)}{\mathcal{M}_s([0,1])} \right ).
\end{equation}
\item If $I \in \mathscr{F}_k$, then the corresponding interval $I' \in \mathscr{F}_k'$ satisfies $\diam{I'} = 0$. That is, $I'$ is a singleton.
\item If $I \in \mathscr{N}_k$, then the corresponding interval $I' \in \mathscr{N}_k'$ satisfies $\diam{I'} = \dfrac{\mathcal{M}_s(k,I)}{\mathcal{M}_s([0,1])}$.
\end{itemize}

\begin{lem}\label{lem:diamtozero}
For any $\e>0$, there exists $k_0\geq 0$ such that $\diam{I} \leq \e$ for all $k\geq k_0$ and $I\in \mathscr{I}_k'$.
\end{lem}

\begin{proof}
Fix $\e>0$. Note that if $J \in \mathscr{E}_l$, then $\card(\mathscr{E}_{l+1}(I)) \geq 2$ by Lemma \ref{lem:flatcont}, because $\diam f_l(J) \geq 2^{-l}$ and $V_{l+1}$ is a maximal $2^{-(l+1)}$-separated set in $E$.

Suppose that $I \in \mathscr{I}_k$. Set $J_k = I$. Inductively, given $J_l \in \mathscr{I}_l$, let $J_{l-1}$ denote the unique interval in $\mathscr{I}_{l-1}$ with $J_l \subset J_{l-1}$. That is,
\[ I = J_k \subset J_{k-1} \subset \cdots \subset J_1 \subset J_0 = [0,1].\]
For each $i \in \{0,\dots,k\}$, let $J_i' \in \mathscr{I}_i'$ be the interval associated to $J_i$. We claim that for each $i \in \{0,\dots,k\}$
\begin{equation}\label{eq:bound2}
\diam{J_i'} \leq \sum_{l=0}^{i}  2^{l-i}\frac{\mathcal{M}_s(l,J_l)}{\mathcal{M}_s([0,1])}
.\end{equation}
We prove (\ref{eq:bound2}) by induction. For $i=0$ the claim is clear. Assume that (\ref{eq:bound2}) is true for $0\leq i<k$. If $J_{i+i} \in \mathscr{F}_{i+1}(J_i)$, then $\diam{J_{i+1}'}=0$ and (\ref{eq:bound2}) is clear. If $J_{i+i} \in \mathscr{F}_{i+1}(J_i)$, then $\diam{J_{i+1}'} = \mathcal{M}_s(i+1,J_{i+1})/\mathcal{M}_s([0,1])$ and (\ref{eq:bound2}) is again clear. If $J_{i+i} \in \mathscr{E}_{i+1}(J_i)$, then by (\ref{eq:bound}) and the induction hypothesis,
\[ \diam{J_{i+1}'} \leq \frac{\mathcal{M}_s(i+1,J_{i+1})}{\mathcal{M}_s([0,1])} + \frac{1}{2}\diam{J_i'} \leq  \sum_{l=0}^{i+1}  2^{l-(i+1)}\frac{\mathcal{M}_s(l,J_l)}{\mathcal{M}_s([0,1])}.\] This establishes \eqref{eq:bound2}.

Choose an integer $l_0$ sufficiently large so that $2^{-l_0}l_0 \leq \e/2$. If $k\geq l_0$, then by (\ref{eq:bound2}) and the fact that $\mathcal{M}_s([0,1])\geq (\diam E)^s=1$,
\[ \diam{I'} \leq \sum_{l=0}^{l_0-1}  2^{l-k}\frac{\mathcal{M}_s(l,J_l)}{\mathcal{M}_s([0,1])} +\sum_{l=l_0}^k  2^{l-k}\frac{\mathcal{M}_s(l,J_l)}{\mathcal{M}_s([0,1])}
\leq \frac{\e}{2}+ 2\,\mathcal{M}_s(l_0,J_{l_0}).\] Thus, by Corollary \ref{cor:massest}, $$\diam {I'} \leq \frac{\epsilon}{2} + 2\,\frac{48^s}{c}\,\sup_{x\in E} \Haus^s(E\cap B(x, 6\cdot 2^{-l_0}))$$ whenever $k\geq l_0$.
Now, because $E$ is compact, $\mathcal{H}^s(E)<\infty$, and $\mathcal{H}^s$ has no atoms, $$\lim_{n\rightarrow\infty}\sup_{x\in E} \mathcal{H}^s(E \cap B(x, 6\cdot 2^{-n}))=0.$$ Hence, by choosing $l_0$ even larger if necessary, we can ensure that $$\sup_{x\in E}  2\,\frac{48^s}{c} \mathcal{H}^s(E\cap B(x, 6\cdot 2^{-l_0})) < \frac{\e}{2}.$$ Therefore, $\diam {I'} < \epsilon$ for all $k\geq l_0$ provided that $l_0$ is sufficiently large depending on $\epsilon$ and $E$.
\end{proof}

Following the proof of Theorem \ref{thm:main} in \S\ref{sec:proof1}, we obtain a sequence of maps $F_k:[0,1]\rightarrow \R^N$ and a $(1/s)$-H\"older continuous map $F:[0,1]\rightarrow\R^N$ satisfying the following properties.
\begin{enumerate}
\item For each $k\geq 0$, there exist (possibly degenerate) closed intervals $[0,a_k]$ and $[b_k,1]$ such that $F_k|[0,a_k]$ and $F|[0,b_k]$ are constant maps and $F(0),F(1)\in V_k$.
\item For each $k\geq 0$, the map $F_k|(a_k,b_k)$ is an injective piecewise linear map connecting $F_k(0)$ to $F_k(1)$ along line segments between points in $V_k$ of length at most $3\cdot 2^{-k}$.
\item If $x\in V_k\setminus\{F_k(0),F_k(1)\}$ for some $k\geq 0$, then $F_k^{-1}(x)=F_j^{-1}(x)$ for all $k\geq j$ (because intervals in $\mathscr{F}_j$ are frozen).
\item The maps $F_k$ converge uniformly to $F$ and $F([0,1])$ contains $\bigcup_{k=0}^\infty V_k$.
\item The H\"older constant of $F$ satisfies $$H \leq \frac{1}{\xi_1}\left(\mathcal{M}_s([0,1])r_0^{1-s} + 60A^* r_0 \frac{\xi_2}{1-\xi_2}\right)
\lesssim_s c^{-1}\mathcal{H}^s(E)(\diam E)^{1-s}.$$\end{enumerate}
It remains to show that $F([0,1])=E$ and $F$ is injective.

On one hand, since each $V_k$ is a maximal $2^{-k}$-separated set in $E$,
$$F([0,1]) \supset \overline{\textstyle\bigcup_{k=0}^\infty V_k} = E$$
by (3). On the other hand, if $x\in F([0,1])$, say $x=F(t)$, then
$$\dist(x,E) \leq \liminf_{k\rightarrow\infty} \dist(F_k(t),E) =0$$
by (2). Thus, $F([0,1])=E$.

To check injectivity, we first establish a lemma. We say that an interval $I$ \emph{separates} two numbers $x < y$ if $x<z<y$ for all $z\in I$.

\begin{lem}\label{lem:inj}
Let $0 \leq x<y\leq 1$. If $k_0$ is the least integer $k\geq 0$ such that there exists an interval in $\mathscr{E}_k'$ separating $x$ and $y$, then $|F_k(x)-F_k(y)| \gtrsim 2^{-k_0}$ for all $k\geq k_0$.
\end{lem}

\begin{proof}Fix $0 \leq x<y\leq 1$, let $k_0$ be as in the statement of the lemma and let $k\geq k_0$. Let $I_0 \in \mathscr{E}_{k_0}$ be such that $I_0$ separates $x$ and $y$. The proof is divided into two cases.

\emph{Case 1.} Assume that $k\leq k_0 + 3$. By Remark \ref{rem:lleq5} and minimality of $k_0$, there exist at most $13$ intervals $I\in \mathscr{E}_{k_0}'$ separating $x$ from $y$. Therefore, there exist consecutive intervals $J_1,\dots,J_l \in \mathscr{I}_k$ such that $x, y \in \bigcup_{i=1}^l J_i$, $I_0 \subset \bigcup_{i=1}^l J_i$ and $l\leq 15$. Let $a$ be an endpoint of $I_0$. Since $\diam{F_k(J_i)} \leq 3\cdot 2^{-k_0}$ for all $i\in\{1,\dots,l\}$, the points $F_k(x)$ and $F_k(y)$ are in $B := B(F_k(a), 45\cdot 2^{-k_0})$. Let $\ell$ be a best fitting line for $B$ and let $\pi$ be the orthogonal projection on $\ell$.  Since $\beta_E(B) \leq 2^{-11}$, the points of $F_k(\bigcup_{i=1}^l J_i)\cap B$ are linearly ordered according to their projection on $\ell$. In particular, $|z-w| \leq 2|\pi(z)-\pi(w)|$ for all $z,w \in F_k(\bigcup_{i=1}^l J_i)\cap B$. Thus,
\[ |F_k(x) - F_k(y)| \geq |\pi_{\ell}(F_k(x)) - \pi_{\ell}(F_k(y))| \geq \diam{\pi_{\ell}(F_k(I_0))} \geq \frac12\diam{F_k(I_0)} \geq \frac12 2^{-k_0}.\]

\emph{Case 2.} Assume that $k\geq k_0+4$. For each integer $i\geq 0$, let $P_i(x,y)$ be the endpoints of intervals in $\mathscr{E}_{k_0+4+i}' $ lying between $x$ and $y$. Let $x_i$ (resp.~ $y_i$) be the leftmost (resp.~ rightmost) element of $P_i(x,y)$. For all $i\geq 0$,
\[ x \leq x_{i+1} \leq x_i < y_i \leq y_{i+1} \leq y, \]
and $I_0$ separates $x_0$ from $y_0$. As in Case 1, if $\ell$ is a best fitting line for $F_{k_0}(a)$, then
\[ |F_{k_0+3}(x_0) - F_{k_0+3}(y_0)| \geq |\pi_{\ell}(F_{k_0+3}(x_0)) - \pi_{\ell}(F_{k_0+3}(y_0))| \geq \diam{\pi_{\ell}(F_{k_0}(I_0))} \geq \frac9{10}2^{-k_0}.\]
Now, each $x_{i+1}$ (resp.~ $y_{i+1}$) is contained in the closure of an interval in $\mathscr{E}_{k+4+i}'$ which has $x_i$ (resp.~ $y_i$) as an endpoint. This fact along with Lemma \ref{lem:TSPflat} yields
\begin{align*}
|F_{k_0+4+i}(x_i)-F_{k_0+5+i}(x_{i+1})| &\leq 3 \cdot 2^{-k_0-4-i},\\ |F_{k_0+4+i}(y_i)-F_{k_0+5+i}(y_{i+1})| &\leq 3 \cdot 2^{-k_0-4-i}.
\end{align*}
Therefore, by the triangle inequality,
\begin{align*}
|F_k(x) - F_k(y)| &\geq |F_{k_0+4}(x_0)-F_{k_0+4}(y_0)| - \sum_{i=0}^{\infty}|F_{k_0+4+i}(x_i)-F_{k_0+5+i}(x_{i+1})|\\
&\hspace{1in}- \sum_{i=0}^{\infty}|F_{k_0+4+i}(y_i)-F_{k_0+5+i}(y_{i+1})|\\
&\geq \frac9{10}2^{-k_0} - 6\cdot 2^{-k_0-4} - 6\cdot 2^{-k_0-4}.
\end{align*}
Hence $|F_k(x) - F_k(y)| \geq (3/20)2^{-k_0}$ and the proof is complete.
\end{proof}

Suppose that $x,y \in [0,1]$ with $x<y$. By Lemma \ref{lem:diamtozero}, there exist intervals $I\in \mathscr{E}'_k$ that separate $x$ and $y$ provided that $k$ is sufficiently large. If $k_0$ is the least such integer, then $|F(x) - F(y)| \gtrsim 2^{-k_0}>0$ by Lemma \ref{lem:inj}. This shows that $F$ is injective and completes the proof of Proposition \ref{thm:flat-statedagain}.

\subsection{Proof of Lemma {\ref{lem:flatcont}}} \label{sec:prop-continua}

We first give an auxiliary estimate.

\begin{lem}\label{lem:flat-twolines}
Let $E\subset\R^N$, $x\in E$, and $r>0$. If $y\in E\cap \overline{B}(x,r)$, $|y-x|\geq 64\beta_E(\overline{B}(x,r))r$, and $\ell_{x,y}$ is the line passing through $x$ and $y$, then $$\dist(z,\ell_{x,y}) \leq 4\beta_E(\overline{B}(x,r))\left(1+\frac{1.1r}{|y-x|}\right)r\quad\text{for all }z\in E\cap B(x,r).$$
\end{lem}

\begin{proof} Let $z\in E\cap \overline{B}(x,r)$, $z\neq x$. Let $\ell$ be a best fitting line for $E$ in $\overline{B}(x,r)$. Then $\dist(x,\ell)$, $\dist(y,\ell)$, and $\dist(z,\ell)$ are bounded above by $\beta_E(\overline{B}(x,r)) 2r$. Let $\ell_x=\ell-x$. Then $x\in\ell_x$ and $\dist(y,\ell_x)$ and $\dist(z,\ell_x)$ are bounded above by $2\beta_E(\overline{B}(x,r))2r$. If $y\in\ell_{x}$, then we have $\dist(z,\ell_{x,y}) = \dist(z,\ell_{x}) \leq 2\beta_E(\overline{B}(x,r))2r$ and we are done.

To continue, suppose that $y\not\in\ell_x$ and let $y'=\pi_{\ell_x}(y)$ and let $z'=\pi_{\ell_x}(z)$. Define $$w=x+\frac{|z'-x|}{|y'-x|}(y-x)\in\ell_{x,y}.$$
Since $z' \in \ell_x$ between $x$ and $y'$, we have that $z' = x+|z'-x||y'-x|^{-1}(y'-x)$. Therefore, $$\dist(z',\ell_{x,y}) \leq |z'-w| = |y'-y|\frac{|z'-x|}{|y'-x|} \leq |y'-y|\frac{r}{|y'-x|}.$$ Thus, by the triangle inequality, $$\dist(z,\ell_{x,y}) \leq |z'-z| + |y'-y|\frac{r}{|y'-x|} \leq 2\beta_E(\overline{B}(x,r))\left(1+ \frac{r}{|y'-x|}\right)2r.$$ Since $\dist(y,\ell_x) \leq 2\beta_E(\overline{B}(x,r))2r \leq (1/16)|x-y|,$ we have $$1.1|y'-x| \geq (1+3(1/16)^2)|y'-x|\geq |y-x|$$ by Lemma \ref{lem:approx}, applied with $V=\{x,y\}$. Therefore, \begin{equation*}\dist(z,\ell_{x,y}) \leq 2\beta_E(\overline{B}(x,r))\left(1+\frac{1.1r}{|y-x|}\right)2r. \qedhere\end{equation*}\end{proof}

We now give a proof of the key lemma.

\begin{proof}[Proof of Lemma {\ref{lem:flatcont}}] Without loss of generality, we may assume that $\diam{E}=1$. Fix $x,y\in E$ and let $n\geq 0$ be the unique integer such that
\[ 2^{-(n+1)} < |x-y| \leq 2^{-n}.\]

\emph{Case 1.} Suppose that $n\in\{0,1\}$. Let $\ell$ be the line containing $x$ and $y$. Since $\beta_E(\overline{B}(x,1)) \leq 2^{-11}$ and since $|x-y|>2^{-2}$, by Lemma \ref{lem:flat-twolines} we have that
\begin{equation*}\begin{split} \sup_{w\in E}\dist(w,\ell) = \sup_{w\in E\cap \overline{B}(x,1)}\dist(w,\ell)  &\leq \left( 1+ \frac{1.1}{|x-y|}\right ) 4\beta_E(\overline{B}(x,1)) \\ &\leq \frac{21.6}{2^{11}} \leq \frac{1}{2^{6}} < \frac{|x-y|}{2^4}.\end{split}\end{equation*}
Therefore, $E$ is contained inside the tube $T := \overline{B}(x,1) \cap B(\ell,2^{-4}|x-y|)$. Let $D$ be a closed $(N-1)$-ball centered at $z$, perpendicular to $\ell$ and of radius $2^{-4}|x-y|$. In other words, $D$ is the set of all points in $\overline{B}(z,2^{-6})$ whose projection on $\ell$ is $z$. Then $D$ cuts $ T$ into two pieces, one containing $x$ and another containing $y$. By connectedness of $E$, we must have $D\cap E \neq \emptyset$.

\emph{Case 2.} Suppose that $n\geq 2$. The procedure here is roughly the same as that in Case 1, with the difference that the tube $T$ is replaced by a more complicated set. By connectedness of $E$, for each $k\in \{1,\dots, n-1\}$, there exists a point $y_k \in B(x,2^{-k})\cap E$. For each $k\in \{1,\dots,n-1\}$ let $\ell_k$ be the line containing $x$ and $y_k$. Let also $\ell_n$ be the line containing $x$ and $y$.

Working as in Case 1, we can show that for each $k\in\{1,\dots,n\}$,
\[ E\cap \overline{B}(x,2^{-(k-1)}) \subset T_k := \overline{B}(x,2^{-(k-1)})\cap B(\ell_k,2^{-5} 2^{-(k-1)}).\]
Since $\diam{E} =1$, we also have $E \subset T_1$. For each $k\in \{1,\dots,n-1\}$, let $T_{k,1}$, $T_{k,2}$ be the two components of $T_k \setminus \overline{B}(x, 2^{-k})$. Set
\[ T = T_{1,1}\cup\cdots\cup T_{n-1,1} \cup T_n \cup T_{n-1,2}\cup\cdots\cup T_{1,2}.\]
The sets $T_{1,1}, \dots,  T_{n-1,1},  T_n ,  T_{n-1,2},\dots, T_{1,2}$ intersect at most in pairs. In particular,
\begin{enumerate}
\item if $i\in\{1,2\}$, then $T_{1,i}\cap T_{m,j} = \emptyset$ unless $m\in\{1,2\}$ and $j=i$;
\item if $k\in\{2,\dots,n-2\}$ (if any) and $i\in\{1,2\}$, then $T_{k,i} \cap T_{m,j} = \emptyset$ unless $m\in\{k-1,k,k+1\}$ and $j=i$;
\item if $i\in\{1,2\}$, then $T_{n-1,i}\cap T_{m,j} = \emptyset$ unless $m\in\{n-2,n-1\}$ and $j=i$;
\item $T_n \cap T_{m,j} = \emptyset$ unless $m=n-1$.
\end{enumerate}
As with Case 1, if $D$ is an $(N-1)$-ball centered at $z$, perpendicular to $\ell_n$ and of radius $2^{-4}2^{-n}$, then  $D$ cuts $T_{n}$ into two pieces, one containing $x$ and another containing $y$. Consequently, $D$ cuts $T$ into two pieces, one containing $x$ and another containing $y$. By connectedness of $E$ and the fact that $E\subset T$, we must have $D\cap E \neq \emptyset$. \end{proof}

\section{Examples}\label{sec:examples}

In this section, we give examples of H\"older curves and of sets that are not contained in H\"older curves to illuminate Theorem \ref{thm:main2}, Proposition \ref{thm:flat}, and Theorem \ref{thm:main}.

\subsection{H\"older curves that are non-flat in all scales}\label{sec:notnec}  First up, we show that condition (\ref{eq:thm2}) in Theorem \ref{thm:main2} is not necessary for a bounded set to be contained in a $(1/s)$-H\"older curve when $s>1$. In contrast, when $s=1$, condition \eqref{eq:betasum} in the Analyst's Traveling Salesman theorem is necessary and sufficient for a bounded set to be contained in a rectifiable curve.

Let $N\geq 2$ and $1\leq m \leq N-1$ be integers. Given a nonempty set $E \subset \R^N$ and an $N$-cube $Q\subset\R^N$ with $E\cap Q \neq \emptyset$, define the $m$-dimensional beta number
\[ \beta_{E}^{(m)}(Q) := \inf_{P} \sup_{x\in E\cap Q}\frac{\dist(x,P)}{\diam{Q}}\]
where the infimum is taken over all $m$-planes $P$ in $\R^N$. If $E\cap Q = \emptyset$, set $\beta_{E}^{(m)}(Q) = 0$. Note that $\beta_{E}^{(1)}(Q) =\beta_E(Q)$ as defined in \textsection\ref{sec:intro} and that $\beta_{E}^{(m)}(Q) \leq \beta_{E}^{(n)}(Q)$ whenever $m\geq n$.

\begin{prop}\label{ex:notnec}
For any $N\geq 2$ and any $s\in (1,N]$, there exists a $(1/s)$-H\"older curve $E\subset \R^N$ such that
\begin{equation}\label{eq:notnec}
\sum_{\substack{Q\in\D(\R^N)\\ \beta_{E}^{(N-1)}(3Q)\geq (6\sqrt{N})^{-1}}} (\diam{Q})^s = \infty.
\end{equation}
\end{prop}

The construction splits into three cases. Before proceeding, we introduce some notation. Given a cube $Q\subset \R^N$, denote by $\D(Q)$ the set of dyadic cubes in $\D(\R^N)$ that are contained in $Q$. Moreover, given positive integers $m \leq N$, there exists a polynomial $P_{N,m}$ of degree $m$ with the following property: If $n\in\N$ and $\{Q_1,\dots,Q_{N^n}\}$ is a partition of $[0,1]^N$ into $N$-cubes of side-length $1/n$, then
\[ \card\{Q_i : \text{$Q_i$ intersects the $m$-skeleton of $\partial [0,1]^N$}\} = P_{N,m}(n).\]
Recall that if $I_1,\dots,I_N$ are nondegenerate compact intervals, and $Q=I_1\times\cdots\times I_N$ is an $N$-cube, then the $m$-skeleton of $Q$ is the union of sets $I_1'\times\cdots\times I_N'$ where $I_j' = I_j$ for $m$ indices $j$ and $I_j' = \partial I_j$ for the remaining $N-m$ indices $j$. Finally, we note that if $K$ is the set of vertices of a cube $Q$ in $\R^N$ and $P$ is an $(N-1)$-plane, then
\begin{equation}\label{eq:sierp4}
\dist(x,P) \geq (2\sqrt{N})^{-1}\diam{K} \qquad\text{for all $x\in K$}.
\end{equation}

\emph{Case 1: $s=N$.} We simply take $E=[0,1]^N$. It is well known that there exists a $(1/N)$-H\"older parametrization $f:[0,1]\to E$. On the other hand, by (\ref{eq:sierp4}),
\[\sum_{\substack{Q\in\D(\R^N)\\ \beta_{E}^{(N-1)}(3Q)\geq (2\sqrt{N})^{-1}}} (\diam{Q})^N \geq \sum_{k=0}^{\infty}\sum_{\substack{Q\in\D([0,1]^N)\\ \diam{Q} = \sqrt{N}2^{-k}}}(\diam{Q})^N = \sum_{k=0}^{\infty} 2^{Nk}(\sqrt{N}2^{-k})^N = \infty.\]

\emph{Case 2: $s\in(1,N)\setminus \N$.}  Let $m$ be the integer part of $s$. Since the degree of $P_{N,m}$ is strictly less than $s$ and strictly lager than $s-1$, we can fix $n\in\N$ such that
\begin{equation}\label{eq:sierp1}
n^s - (n-2)^s < P_{N,m}(n) < n^s - 1.
\end{equation}
By (\ref{eq:sierp1}) and the Intermediate Value Theorem, there exists $\lambda \in (\frac1{n},1-\frac2{n})$ such that
\begin{equation}\label{eq:sierp2}
P_{N,m}(n) n^{-s} + \lambda^s = 1.
\end{equation}
Partition $[0,1]^N$ into $N$-cubes of side-lengths $1/n$ and let $\{Q_i\}_{i=1}^{l}$ ($l=P_{N,m}(n)$) be those cubes that intersect the $m$-skeleton of $[0,1]^N$. Let also $Q_0 = [1/n, 1/n+\lambda]^N$. For each $i=0,\dots,l$, let $\phi_i$ be a similarity of $\R^N$ such that $\phi_i([0,1]^N) = Q_i$. Finally, define $E\subset \R^N$,
\[ E := \bigcap_{k=1}^{\infty} \bigcup_{i_1\cdots i_k \in \{0,\dots,l\}^k}\phi_{i_1}\circ\cdots\circ\phi_{i_k}([0,1]^N). \]
Since the maps $\{\phi_0,\dots,\phi_l\}$ satisfy the open set condition, $E$ is Ahlfors regular \cite{hutchinson}. By (\ref{eq:sierp2}) the Hausdorff dimension of $E$ is equal to $s$, so $E$ is $s$-regular.

\begin{lem}\label{lem:Eisconn}
The set $E$ is connected.
\end{lem}

\begin{proof}
Set $\mathcal{W} = \bigcup_{k\geq 0}\{0,\dots,l\}^k$ with the convention that $\{0,\dots,l\}^0$ is the empty word $\emptyset$ and $\phi_{\emptyset}$ is the identity map of $\R^N$ For each $w\in\W$, let $\mathcal{K}_w$ denote the $1$-skeleton of $\phi_w([0,1]^N)$. The proof is based now on two observations. First, by the choice of cubes $Q_1,\dots,Q_l$, it follows that $\mathcal{K}_w \subset E$ for all $w\in\W$. Second, $\mathcal{K}_w \cap \mathcal{K}_{wi} \neq \emptyset$ for all $w\in\W$ and $i\in\{0,\dots,l\}$.

Now fix $x\in E$. There exists a sequence of words $(w_n)_{n\geq 0}$ in $\W$ such that $w_0$ is the empty word, $w_{n+1} = w_ni_n$ with $i_n \in \{0,\dots,l\}$, and $x \in \bigcap_{n\geq 0}\phi_{w_n}(x)$. The set $\bigcup_{n\geq 0}\mathcal{K}_{w_n}$ is a path that joins $x$ with the origin. Hence $E$ is connected.
\end{proof}

By Lemma \ref{lem:Eisconn}, the fact that the Hausdorff dimension of $E$ is $s$, and Theorem 4.12 in \cite{Remes}, there exists a $(1/s)$-H\"older map $f:[0,1] \to \R^N$ such that $f([0,1]) = E$. It remains to show (\ref{eq:notnec}). We first prove a lemma.

\begin{lem}\label{lem:notnec}
If $Q \in \D([0,1]^N)$ is a dyadic cube that intersects $E$, then there exists a dyadic cube $Q'\subset 3Q$ such that $\diam{Q'}\geq (3n)^{-1}\diam{Q}$ and $\beta_{E}^{(N-1)}(3Q') \geq (6\sqrt{N})^{-1}$.
\end{lem}

\begin{proof}
Fix $x\in Q\cap E$ and let $i_1,i_2,\dots$ be a sequence of numbers in $\{0,\dots,l\}$ such that
\[ x\in \bigcap_{k=1}^{\infty}\phi_{i_1}\circ\cdots\circ\phi_{i_k}([0,1]^N).\]
Let $k_0$ be the smallest positive integer such that $\phi_{i_1}\circ\cdots\circ\phi_{i_{k_0}}([0,1]^N) \subset 3Q$ and define $K$ to be the set of vertices of $\phi_{i_1}\circ\cdots\circ\phi_{i_{k_0}}([0,1]^N)$. Since each $\phi_i$ has a scaling factor at least $1/n$, by minimality of $k_0$ we have that $\diam{K} \geq (1/n)\diam{Q}$.

Let $Q'$ be a dyadic cube in $\D(3Q)$ (possibly $Q'=Q$) of minimal diameter such that $K\subset 3Q'$. We claim that
\begin{equation}\label{eq:sierp3}
\frac{1}3\diam{K} \leq \diam{Q'} \leq \diam{K}.
\end{equation}
The lower inequality is clear. If $\diam{K} < \diam{Q}$, then, since $K$ has edges parallel to the axes, $K$ is contained in $3Q_0$ for some dyadic cube $Q_0\subset 3Q$ with $\diam{Q_0} = \frac12 \diam{Q'}$, which is a contradiction. That establishes the upper inequality of (\ref{eq:sierp3}).

By (\ref{eq:sierp4}), (\ref{eq:sierp3}), and the fact that $K\subset E$,
\[ \beta_{E}^{(N-1)}(3Q') \geq \beta_{K}^{(N-1)}(3Q') \geq \frac{(2\sqrt{N})^{-1}\diam{K}}{\diam{3Q}} = \frac{(6\sqrt{N})^{-1}\diam{K}}{\diam{Q}} \geq (6\sqrt{N})^{-1}. \] This proves the lemma.
\end{proof}

By Ahlfors $s$-regularity of $E$, there exists a constant $C>1$ such that
\[ \card\{Q\in \D([0,1]^N) : \diam{Q} = \sqrt{N}2^{-k} \text{ and } Q \cap E \neq \emptyset\} \geq C^{-1} 2^{sk}.\]
Fix a positive integer $k_0$ such that $2^{k_0} >3n$. For $k\in\N$, set
\[ \mathcal{Q}_k = \{Q\in \D([0,1]^N) : \diam{Q} \in [\sqrt{N}2^{-k}, \sqrt{N}2^{-k-k_0}] \text{ and } \beta_{E}^{(N-1)}(3Q) \geq (6\sqrt{N})^{-1}\} .\]
By Lemma \ref{lem:notnec},
\begin{align*}
\card{\mathcal{Q}_k} \geq 3^{-N}\card\{Q\in \D([0,1]^N) : \diam{Q} = \sqrt{N}2^{-k} \text{ and } Q \cap E \neq \emptyset\} \geq C^{-1}3^{-N} 2^{sk}.
\end{align*}
Therefore,
\begin{align*}
\sum_{\substack{Q\in\D(\R^N)\\ \beta_{E}^{(N-1)}(3Q) \geq (6\sqrt{N})^{-1}}} (\diam{Q})^N &\geq \sum_{k=0}^{\infty}\sum_{Q \in \mathcal{Q}_{kk_0}}(\diam{Q})^N \\
 &\geq \sum_{k=0}^{\infty} C^{-1}3^{-N} 2^{skk_0} (\sqrt{N})^{s}2^{-s(k+1)k_0} = \infty.
 \end{align*}

\emph{Case 3: $s\in\{2,\dots,N-1\}$.} Fix $n\in\N$ large enough so that $P_{N,s-1}(n) < n^s$. Partition $[0,1]^N$ into $N$-cubes with disjoint interiors and side-lengths $1/n$ and let $\{Q_1,\dots,Q_{l}\}$ ($l=n^s$) be a collection of such cubes so that the set $\bigcup_{k=1}^l Q_i$ is connected and contains the $(s-1)$-skeleton of $[0,1]^N$. The rest of the construction is similar to Case 2 and is left to the reader.

\subsection{Ahlfors regular curves without H\"older parametrizations}\label{sec:regularcurve} Next, for all $s>1$, we construct Ahlfors $s$-regular curves that are not contained in any $(1/s)$-H\"older curve. The basic strategy is take a disconnected set, which is not contained in a H\"older curve, and then extend the set to transform it into an $s$-regular curve. We call the curves that we construct ``Cantor ladders".

\begin{prop}
Let $N\in\N$ with $N\geq 2$, let $s\in (1,N)$, and let $m\in\N$ with $m\leq s$. There exists an Ahlfors $s$-regular curve $E\subset \R^N$, which is not contained in a $(m/s)$-H\"older image of $[0,1]^m$.
\end{prop}

We treat the cases $s\in \N$ and $s\not\in\N$ separately. Given $m\in\N$, let $\mathcal{W}_m$ be the set of finite words formed by the letters $\{1,\dots,m\}$ including the empty word $\emptyset$. We denote by $|w|$ the number of letters a word has with the convention $|\emptyset| = 0$.

\emph{Case 1.} Suppose that $s\in\{2,3,\dots,N-1\}$. Let $D_{\emptyset} = [0,1]^2$. Given a square $D_w \subset \R^2$ for some $w\in\mathcal{W}_4$, let $D_{w1}$, $D_{w2}$, $D_{w3}$, $D_{w4}$ be the four corner squares in $D_w$ with  $\diam{D_{wi}} = (1/4)\diam{D_w}$. Let $\mathcal{C}_1$ be the Cantor set in $\R^2$ defined by
\[ \mathcal{C}_1 = \bigcap_{k=0}^{\infty}\bigcup_{\substack{w\in\mathcal{W}_4 \\ |w|=k}} D_w.\]
For each $i = 1,\dots, 2^{|w|}$, define $D_{w,i} = D_w\times\{(2i-1)2^{-|w|-1}\}$,
\[ K_2 =   (\mathcal{C}_1\times[0,1]) \cup \bigcup_{w\in\mathcal{W}}\bigcup_{i=0}^{2^{|w|}-1} D_{w,i}\qquad\text{and}\qquad E =  K_2\times [0,1]^{s-2} \times \{0\}^{N-s-1}.\]
Here and for the rest of \textsection\ref{sec:regularcurve}, we use the convention $A\times\{0\}^0 = A$.

\emph{Case 2.} Suppose that $s\in(1,N)\setminus \N$. Let $p = s-\lfloor s \rfloor$ be the fractional part of $s$. Let $I_{\emptyset} = [0,1]$. Given an interval $I_w = [a_w,b_w]$ for some $w\in\mathcal{W}_2$, let
\[ I_{w1} = [a_w, a_w+2^{-p}(b_w-a_w)] \qquad\text{and}\qquad I_{w2} = [b_w-2^{-p}(b_w-a_w),b_w].\]
Let $\mathcal{C}_p$ denote the Cantor set in $\R$ defined by
\[ \mathcal{C}_p = \bigcap_{k=0}^{\infty}\bigcup_{\substack{w\in\mathcal{W}_2 \\ |w|=k}} I_w.\]
Let $S$ be the bi-Lipschitz embedded image of $([0,1],|\cdot|^{\frac1{p+1}})$ into $\R^2$. For each $w\in\mathcal{W}_2$, let $S_w$ be a rescaled copy of $S$ whose endpoints are the right endpoint of $I_{w1}$ and the left endpoint of $I_{w2}$. For each $w\in\W_2$ and $i = 1,\dots, 2^{|w|}-1$, define
\[ S_{w,i} = S_w + (0,(2i-1)2^{-|w|-1})\]
and define
\[ K_{p+1} = (\mathcal{C}_{p}\times [0,1])\cup\bigcup_{w\in\mathcal{W}}\bigcup_{i=0}^{2^{|w|}} S_{w,i}\qquad\text{and}\qquad E =   K_{p+1}\times [0,1]^{s-p-1} \times \{0\}^{N+p-s-1}.\]

Verification of the desired properties of $E$ is the same for the two cases, so we only treat Case 1. By Theorem 2.1 in \cite{MM2000}, there exists no $(m/s)$-H\"older map $f:[0,1]^m \to \R^N$ whose image contains $\mathcal{C}_1 \times [0,1]^{s-1}\times\{0\}^{N-s-1}$. We show that $E$ is a curve in \textsection\ref{sec:Eisacurve} and we prove $s$-regularity of $E$ in \textsection\ref{sec:Eisregular}.

\subsubsection{$E$ is a curve}\label{sec:Eisacurve}
By the Hahn-Mazurkiewicz theorem \cite[Theorem 3.30]{hocking}, to show that $E$ is a curve it is enough to show that $E$ is compact, connected, and locally connected.

For compactness, it is easy to see that $K_2 \subset [0,1]^3$, hence $E\subset [0,1]^N$. Moreover, as $|w|\to\infty$, the squares $D_{w,i}$ accumulate on $\mathcal{C}\times [0,1]$. Therefore, $K_2$ is closed. Consequently, $E$ is compact.

To settle both connectedness and local connectedness, we prove that there exists $C>1$ such that for all pairs of points $x,y \in E$ there exists a path joining $x$ with $y$ of diameter at most $C|x-y|^{1/2}$. Clearly, it suffices to show the claim for $K_2$ instead of $E$. Fix $x,y\in K_2$ and let $w_0$ be the word in $\W_4$ of maximum word-length such that the projections of $x$ and $y$ on $\R^2\times\{0\}$ are contained in $D_{w_0}$. This means that $|x-y| \geq \frac12 4^{-|w_0|}$. Choose $i_0 \in \N$ such that $\dist(x,D_{w_0,i_0})  \leq 2\cdot 2^{-|w_0|}.$
If $x_0$ and $y_0$ are the projections of $x$ and $y$ onto $D_{w_0,i_0}$, respectively, then
\[ \max\{|x_0-x|, |y_0-y|\} \lesssim 2^{-|w_0|} + 4^{-|w_0|} + |x-y| \simeq 2^{-|w_0|}.\]
There exist sequences $(w_n)_{n\in\N}, (u_n)_{n\in\N}$ of words in $\mathcal{W}_4$ and sequences $(i_n)_{n\in\N}, (j_n)_{n\in\N}$ of positive integers such that
\begin{enumerate}
\item $|w_n| = |u_n| = |w_0|+n$;
\item the orthogonal projection of $x$ (resp.~ $y$) on $\R^2$ is contained in $D_{w_n}$ (resp.~ $D_{u_n}$);
\item there exists $x_n \in D_{w_n,i_n}$ such that
\[ \max\{ |x-x_n|, |y-y_n|\} \leq 4^{-|w_0|-n}\sqrt{2} + 2^{-|w_0|-n}.\]
\end{enumerate}
Properties (1) and (2) imply that $D_{w_0} \supsetneq D_{w_1} \supsetneq D_{w_2} \supsetneq \cdots$ and $D_{w_0} \supsetneq D_{u_1} \supsetneq D_{u_2} \supsetneq \cdots$, while property (3) implies that the Hausdorff distances
\[ \dist_H(D_{w_n,i_n}, D_{w_{n+1},i_{n+1}})\lesssim 2^{-|w_0|-n} \quad\text{and}\quad  \dist_H(D_{u_n,j_n}, D_{u_{n+1},j_{n+1}}) \lesssim 2^{-|w_0|-n}.\]
Let $\g_0\subset K_2$ be the line segment joining $x_0$ with $y_0$. For each $n\geq 0$, let $z_n \in D_{w_n,i_n}$ be a corner point and let $z_n'$ be its projection on $D_{w_{n+1},i_{n+1}}$. Also, let  $p_n \in D_{u_n,j_n}$ be a corner point and let $p_n'$ be its projection on $D_{u_{n+1},j_{n+1}}$. Consider the curve
\[ \g = \g_0 \cup \bigcup_{n\in\N}([x_n,z_n] \cup [z_n,z_n'] \cup [z_n',x_{n+1}])\cup \bigcup_{n\in\N}([y_n,p_n] \cup [p_n,p_n'] \cup [p_n',y_{n+1}]),\]
which is a subset of $K_2$ and joins $x$ with $y$. Then
\begin{align*}
\diam{\g} &\lesssim \diam{\g_0} + \sum_{n\geq 0}\diam{\g_n} + \sum_{n\geq 0}\diam{\sigma_n}\\
&\leq |x_0-y_0| + \sum_{n\geq 0}(|x_n-z_n| + |z_n-z_n'| + |z_n'-x_{n+1}|) \\
&\hspace{.2in}+ \sum_{n\geq 0}(|y_n-p_n| + |p_n-p_n'| + |p_n'-y_{n+1}|)\\
&\lesssim 4^{-|w_0|} + \sum_{n\geq 0}(2^{-|w_0|-n}+4^{-|w_0|-n} + 2^{-|w_0|-n}) \lesssim 2^{-|w_0|} \simeq |x-y|^{1/2}.
\end{align*}

\subsubsection{$E$ is $s$-regular.}\label{sec:Eisregular} We show $s$-regularity for $E$. Because the product of regular compact spaces of dimension $s_1$ and $s_2$ is $(s_1+s_2)$-regular, to show that $E$ is $s$-regular,
it suffices to show that $K_{2}$ is $2$-regular.  Fix $x\in K_2$ and $r\in (0,\diam{K_2})$.

We first show that
\begin{equation}\label{eq:lowreg}
\mathcal{H}^2(B(x,r)\cap K_2) \gtrsim r^2.
\end{equation}
If $x\in \mathcal{C}_1\times[0,1]$, then (\ref{eq:lowreg}) follows from the $2$-regularity of $\mathcal{C}_1\times[-1,1]$. If $x \in D_{w,i}$ and $r\leq 10 \diam{D_w}$, then (\ref{eq:lowreg}) follows from the $2$-regularity of $D_{w,i}$. If $x \in D_{w,i}$ and $r\geq 10 \diam{D_w}$, then there exists $z\in (\mathcal{C}_1\times[0,1])\cap B(x,r)$ such that $B(z,r/2) \subset B(x,r)$ and (\ref{eq:lowreg}) follows from the $2$-regularity of $B(z,r/2)\cap K_2$.

For the upper regularity of $K_2$, instead of working with balls $B(x,r)$, it is more convenient to use cubes
\[ Q(x,r) = x + [-r/2,r/2]^3 \qquad x\in K_2, \quad r>0.\]
Without loss of generality, we may assume that $r=4^{-k_0}$ for some $k_0\in\N$. For each $k\geq 0$, let
\[ \mathcal{D}_k(x,r)= \{D_{w,i} : Q(x,r)\cap D_{w,i} \neq \emptyset \text{ and }|w|=k\}.\]
Then by the $2$-regularity of $\mathcal{C}_1\times[-1,1]$, it suffices to show that
\begin{align*}
 \sum_{k\geq 0}\sum_{D_{w,i}\in\mathcal{D}_k(x,r)} \mathcal{H}^2(Q(x,r)\cap D_{w,i}) \lesssim r^2.
\end{align*}
The following lemma will let us estimate the above sum. In the sequel, we denote by $m_0\geq 0$ the smallest integer for which there exists $D_{w,i}\in\mathcal{D}(x,r)$ with $|w|=m_0$.

\begin{lem}\label{lem:squareint}
Let $m_0\geq 0$ be the smallest integer for which $\mathcal{D}_{m_0}(x,r) \neq \emptyset$.
\begin{enumerate}
\item If $k>m_0$ and $\mathcal{D}_{k}(x,r) \neq \emptyset$, then $k\geq 2k_0$.
\item If $Q'$ is the projection of $Q(x,r)$ on $\R^2\times\{0\}$, then for all $k\geq 0$,
\[ \card{\{D_w : D_w\cap Q'\neq \emptyset \text{ and }|w|=k\}} \leq 1+4^{k}4^{-k_0}.\]
\item For each $w\in \W_4$,
\[ \card{\{ i : D_{w,i}\cap Q(x,r) \neq \emptyset\}} \leq 1+2^{|w|+1}4^{-k_0}.\]
\item For each $k\geq 0$, $\card{\mathcal{D}_k(x,r)} \leq (1+4^{k}4^{-k_0})(1+2^{k+1}4^{-k_0})$.
\item We have
\[ \sum_{D_{w,i}\in\mathcal{D}_{m_0}(x,r)}\mathcal{H}^2(D_{w,i}\cap Q(x,r)) \lesssim r^2.\]
\end{enumerate}
\end{lem}

\begin{proof}
For (1), recall that if $|w|>m_0$, then the vertical distance between $D_{w,i}$ and $D_{w_0,i_0}$ is at least $2^{-|w|}$. Since $r=4^{-k_0}$, the cube $Q(x,r)$ can not intersect any $D_{w,i}$, unless $4^{-k_0} \geq 2^{-|w|}$. Thus, $|w|\geq 2k_0$.

For (2), we first note that if $k\leq k_0$, then $Q'$ can intersect at most one square $D_w$ with $|w|=k$. We now use induction to show that for all $k\geq k_0$,
\[ \card{\{D_w : D_w\cap Q'\neq \emptyset \text{ and }|w|=k\}} \leq 4^{k}4^{-k_0}.\]
For $k=k_0$, it is true. Suppose that the claim is also true for some $k\geq k_0$. Then $Q'$ intersects $D_{w}$ with $|w|=k+1$ if and only if there exists $w'$ with $|w'|=k$ such that $Q\cap D_{w'} \neq \emptyset$ and $D_{w}\subset D_{w'}$. Since each square of generation $k$ contains $4$ squares of generation $k+1$,
\begin{align*}
\card{\{D_w : D_w\cap Q'\neq \emptyset \text{ and }|w|=k+1\}} &\leq 4\card{\{D_w : D_w\cap Q'\neq \emptyset \text{ and }|w|=k\}}\\
&\leq 4^{k+1}4^{-k_0}.
\end{align*}

For (3), fix $w\in\W_4$. Recall that the vertical height of $Q(x,r)$ is $2r=2\cdot4^{-k_0}$ and that the vertical distance between $D_{w,i}$ and $D_{w,j}$ with $i\neq j$ is at least $2^{-|w|}$. Therefore,
\begin{align*}
\card{\{ i : D_{w,i}\cap Q(x,r) \neq \emptyset\}} \geq 1 + (2r)/2^{-|w|} = 1+2^{|w|+1}4^{-k_0}.
\end{align*}

Claim (4) is immediate from (2) and (3).

It remains to show (5). On one hand, if $m_0> k_0$, then by (4), $\card(\mathcal{D}_{m_0}(x,r))=1$. Hence  (5) follows from the $2$-regularity of squares $D_{w,i}$. On the other hand, if $m_0\leq k_0$, then by (4),
\[ \sum_{D_{w,i}\in\mathcal{D}_{m_0}(x,r)}\mathcal{H}^2(D_{w,i}\cap Q(x,r)) \leq \card(\mathcal{D}_{m_0}(x,r)) (4^{-m_0})^2 \lesssim 2^{-m_0} 4^{-2k_0} \leq r^2. \qedhere\]
\end{proof}

By Lemma \ref{lem:squareint}, we have
\begin{align*}
\sum_{D_{w,i}\in\mathcal{D}(x,r)} \mathcal{H}^2(B(x,r)\cap D_{w,i}) &\leq \sum_{D_{w,i}\in\mathcal{D}_{m_0}(x,r)} \mathcal{H}^2(D_{w,i}) + \sum_{k=2k_0}^{\infty}\sum_{D_{w,i}\in\mathcal{D}_k(x,r)} \mathcal{H}^2(D_{w,i})\\
&\lesssim r^2 + \sum_{k=2k_0}^{\infty} 2^{k}4^{-k_0}4^k4^{-k_0} 4^{-2k}. 
\end{align*}
Finally,
\[ \sum_{k\geq 2k_0}^{\infty} 2^{k}4^{-k_0}4^k4^{-k_0} 4^{-2k} = 4^{-2k_0}\sum_{k\geq k_0} 2^{-k} \lesssim 4^{-3k_0} \lesssim r^2.\] Therefore, $K_2$ is $2$-regular.

\subsection{A compact countable set that is not contained in any H\"older cube}\label{sec:null}

\begin{prop}
For each $N\in\N$, $N\geq 2$, there exists a compact and countable set $E\subset\R^N$ with one accumulation point such that for any $m\in\{1,\dots,N-1\}$ and any $s\in[1,N/m)$, the set $E$ is not contained in a $(1/s)$-H\"older image of $[0,1]^m$.
\end{prop}

\begin{cor}
For each $N\in\N$, $N\geq 2$, there exists a compact and countable set $E\subset\R^N$ with one accumulation point such that $E$ is not contained in a rectifiable curve.
\end{cor}

For each integer $k\geq 0$, define $\mathcal{G}_{k}^0$ to be the union of all vertices of all dyadic cubes in $\R^N$ that are contained in $[0,1]^N$ and have side length $2^{-k}$. By a simple combinatorial argument, $\card(\mathcal{G}_{0}^k) = (2^k+1)^N$ for all $k\geq 0$.

Let $\phi_0$ be the identity map, and for each $k\geq 1$, define a map $\phi_k : \R^N \to \R^N$ by
\[ \phi_k(x) =(k+1)^{-2}x + \left (0,\dots,0,2\sum_{i=1}^{k} i^{-2} \right ).\]
Set $A := \sum_{i=1}^{\infty} i^{-2} =\pi^2/6$, and define the set
\[ E := \{(0,\dots,0,2A)\} \cup \bigcup_{k=0}^{\infty} \phi_k(\mathcal{G}_{k}^0).\]

The set $E$ is clearly countable. If $(x_1,\dots,x_N) \in E$, then $|x_i| \leq 1$ for all $i=1,\dots,N-1$ while $|x_N| \leq 2A$. Therefore, $E$ is bounded. Moreover, the only accumulation point of $E$ is the point $(0,\dots,0,2A)$ which is contained in $E$. Thus, $E$ is closed.

Next, we claim that
\begin{equation}\label{eq:dyadicnets}
|x-y| \geq 2^{-k}(k+1)^{-2} \qquad\text{for all $x \in \phi_k(\mathcal{G}_{k}^0)$ and all }y \in E\setminus\{x\}.
\end{equation}
Indeed, if $x,y\in\mathcal{G}_k^0$, then inequality (\ref{eq:dyadicnets}) is clear. Otherwise, $\dist(\mathcal{G}_{k}^0, E\setminus \mathcal{G}_{k}^0) \geq (k+1)^{-2}$, and thus, (\ref{eq:dyadicnets}) holds again.

Suppose in order to get a contradiction that there exists a $(1/s)$-H\"older continuous map $f:[0,1]^m\rightarrow\R^N$ such that $E\subset f([0,1])$. Let $H$ be the H\"older constant of $f$.  For each $k\geq 0$ and $x\in \mathcal{G}_{k}^0$, fix a point $w_{k,x}$ such that $f(w_{k,x}) = x$ and set
\[ B_{k,x} = B(w_{k,x}, \tfrac12 H^{-s}2^{-ks}(k+1)^{-2s}).\]
Inequality (\ref{eq:dyadicnets}) implies that the balls $B_{k,x}$ are mutually disjoint. Moreover, it is easy to see that each $B_{k,x}$ is contained in $[-1,2]^m$. Therefore,
\begin{align*}
1 \gtrsim_m\mathcal{H}^m([-1,2]^m) \geq \sum_{k=0}^{\infty}\sum_{x \in \mathcal{G}_{k}^0}\mathcal{H}^m(B_{k,x}) &\gtrsim_{H,s} \sum_{k=0}^{\infty}(2^k+1)^N\frac{2^{-skm}}{(k+1)^{2sm}} \simeq_{N} \sum_{k=0}^{\infty}\frac{2^{k(N-ms)}}{(k+1)^{2s}}.
\end{align*}
Since $N>ms$, the sum on the right hand side diverges and we reach a contradiction.

\subsection{Flat curves with finite $\mathcal{H}^s$ measure and no $(1/s)$-H\"older parametrizations}\label{sec:flatcurves-counterexample}

The following example shows that the assumption of lower $s$-regularity can not be dropped from Proposition \ref{thm:flat}.

\begin{prop}
For any $\beta_0 \in (0,1)$, there exists $s_0 \in (1,2)$ with the following property. For any $s\in (1,s_0)$ there exists a curve $E \subset \R^2$ such that
\begin{enumerate}
\item $\mathcal{H}^s(E) < \infty$ and
\item $\beta_{E}(Q) < \beta_0$ for all $Q\in\D(\R^N)$,
\end{enumerate}
but $E$ is not contained in any $(1/s)$-H\"older image of $[0,1]$.
\end{prop}

Before proceeding, we recall a well-known construction method for snowflakes in $\R^2$. Let $\textbf{p} = (p_0,p_1,\dots)$ be sequence of numbers in $[1/4,1/2)$. Let $\G_0$ be the segment $[0,1]\times\{0\}$, oriented from $(0,0)$ to $(1,0)$. Assume that we have constructed an oriented polygonal arc $\G_k$ with $4^k$ edges. Define $\G_{k+1}$ to be the polygonal arc constructed by replacing each edge $e$ of $\G_{k}$ by a rescaled and rotated copy of the oriented polygonal arc in Figure \ref{fig:figure1} with $p = p_k$, so that the new oriented arc lies to the left of $e$. A snowflake arc $\mathcal{S}_{\textbf{p}}$ is obtained by taking the limit of $\G_{k}$, just as in the construction of the usual von Koch snowflake.

\begin{figure}[ht]
\includegraphics[scale=0.6]{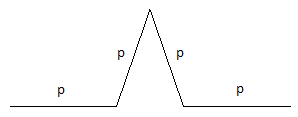}
\caption{}
\label{fig:figure1}
\end{figure}

\begin{rem}\label{rem:flatsnowflakes}
For any $\e>0$, there exists $p^*>1/4$ (small) such that if a snowflake is built with parameters $1/4 \leq p_k \leq p^*$ for all $k\geq 0$, then $\beta_{\G_k}(B(x,r)) \leq \e r$ for all $k\geq 0$, $x\in\G_k$, and $r>0$.
\end{rem}

Fix $\beta_0 \in (0,1)$. By the preceding remark, there exists $p^* \in (1/4,1/2)$ such that $\beta_{\mathcal{S}_{p_0}}(Q) < \beta_0$ for all $Q\in\D(\R^N)$. Set $\textbf{p}=(p^*,p^*,\dots)$, set $s_0 = -\log{4}/\log{p_0}$, and fix $s \in [1,s_0)$. It is well-known that there exists a $(1/s_0)$-bi-H\"older homeomorphism $\Phi : [0,1]\to \mathcal{S}_{\textbf{p}}$; e.g., see \cite{BoHei, RV}.

We now construct a self-similar Cantor set in $[0,1]$ in the following way. Let $I_{\emptyset} = [0,1]$. Assuming we have constructed $I_w = [a_w,b_w]$ for some $w\in\{1,2\}^n$, let
\[I_{w1} = [a_w, a_w + (b_w-a_w)2^{-s_0/s}] \quad\text{and}\quad I_{w2} = [b_w - (b_w-a_w)2^{-s_0/s}, b_w].\]
Define $E' = \bigcap_{n=0}^{\infty}\bigcup_{w\in\{1,2\}^n}I_w$. For each component $J$ of $[0,1]\setminus E'$, let $\g_J$ be the line segment joining the endpoints of $\Phi(J)$. Then define
\[ E = \Phi(E') \cup \bigcup_{J} \g_J,\]
where the union is taken over all components $J$ of $[0,1]\setminus E'$. Since $E'$ is $s/s_0$-regular and $\Phi$ is $(1/s_0)$-bi-H\"older,
\[\mathcal{H}^s(E) = \mathcal{H}^s(\Phi(E')) + \sum_{J}\mathcal{H}^s(\g_J) \leq C\mathcal{H}^{s/s_0}(E') < \infty .\]
Since $\Phi(E')\subset \mathcal{S}_{\textbf{p}}$ and $\gamma_J$ are line segments,
we have $\beta_{E}(Q) < \beta_0$ for all $Q\in\D(\R^N)$. Finally, by Theorem 2.1 in \cite{MM2000}, there does not exist a $(1/s)$-H\"older map $f:[0,1]\to \R^2$ whose image contains $\Phi(E')$ (and consequently $E$).

\subsection{Sharpness of exponent 1 in Theorem \ref{thm:main}}\label{sec:sharpnessof1}
To wrap up, we show that Theorem \ref{thm:main} does not hold if numbers $\tau_{s}(k,v,v')$ are replaced by $\tau_{s}(k,v,v')^p$ with $p>1$. When $s=1$, this follows from the necessary half of the Analyst's Traveling Salesman theorem. Thus, we may focus on the case $s>1$.

\begin{prop}
Let $p>1$, let $s>1$ be sufficiently close to $1$, and let $\alpha_0>0$ be sufficiently close to $0$. There exists a sequence of finite sets $\{(V_k,\rho_k)\}_{\geq 0}$ of numbers and finite sets in $\R^2$ satisfying (V0)--(V5) such that
\begin{equation}\label{eq:sharp}
\sum_{\substack{v\in V_k \\ \a_{k,v}\geq \a_0}}\rho_k^s r_0^s + \sum_{(v,v')\in \Flat(k)}\tau_s(k,v,v')^{p}\rho_k^sr_0^s < \infty
\end{equation}
but there does not exist a $(1/s)$-H\"older map $f:[0,1]\to \R^2$ such that $\bigcup_{k\geq 0}V_k \subset f([0,1])$.
\end{prop}

Let $s>1$ and $n_0 \in \N$ be constants to be specified below. Fix a number
\[ 0 < q < \min\{1/s, (p-1)/s\}. \]
For each $n\in \N$, let
\[ t_k = \sqrt{\frac{1}{4^{1/s}}\left ( 1 + \frac{1}{k+n_0} \right )^{2q} - \frac{1}{4}}.\]

Construct a sequence of polygonal arcs $\G_k$ as in \textsection\ref{sec:flatcurves-counterexample} with parameters
\[ p_k = 1/4 + t_k^2.\]
We may assume that numbers $p_k$ are in $[1/4,1/2)$ by taking $n_0$ to be sufficiently large. For each $k\geq 0$, we define a finite set $V_k \subset \G_k$ as follows. Define $V_0 := \{v_{0,1},v_{0,2}\}$, where $v_{0,1} = (0,0)$ and $v_{0,2} = (1,0)$. Suppose that for some $k\geq 0$ we have defined a set
\[ V_k = \{v_{k,1},\dots,v_{k,N_k}\},\qquad N_k = 2^{k}+1,\]
where points $v_{k,i}$ are enumerated according to the orientation of $\G_k$.
For each $i=1,\dots, 2^k+1$, set $v_{k+1,2i-1} = v_{k,i}$, and assign $v_{k+1,2i}$ to the point of $\G_{k+1}$ that lies between $v_{k+1,2i-1}$ and $v_{k+1,2i+1}$ and is equal distance to $v_{k+1,2i-1}$ and $v_{k+1,2i+1}$ (the peak of the triangle in Figure \ref{fig:figure1}). Define the quantities
\begin{equation}\label{eq:netparameters}
r_0 = 1, \ \  C^*=2, \ \ \xi_1 = 2^{-1/s},\ \ \xi_2 = \frac{1+2^{-1/s}}{2},\ \ \rho_0 = 1,\ \  \rho_k= 2^{-k/s}\frac{(k+1+n_0)^{q}}{(2+n_0)^{q}}.
\end{equation}
For each $k\geq 0$ and $v\in V_k$, define
\[ \a_{k,v} := \inf_{\ell} \sup_{x\in V_{k+1}\cap B(v,30A^*\rho_kr_0)} \frac{\dist(x,\ell)}{\rho_{k+1}r_0},\]
where the infimum is taken over all lines $\ell$ in $\R^2$ and $A^*$ is as in \textsection\ref{sec:flat}. Let $\ell_{k,v}$ be a line $\ell$, which realizes the number $\a_{k,v}$.

\begin{lem}\label{lem:propnets}
There exist choices of $s$ and $n_0$ so that the following properties hold.
\begin{enumerate}
\item For all $k\geq 0$ and $i\in\{1,\dots N_k\}$, we have $|v_{k,i} - v_{k,i+1}| = \rho_k$.
\item The sequence $\{(V_k,\rho_k)\}_{k\geq 0}$ satisfies (V0)--(V5) with the parameters given in (\ref{eq:netparameters}).
\item For all $k\geq 0$ and $v\in V_k$, we have $\a_{k,v} \leq \a_0$, where $\a_0$ is as in Definition 2.4.
\item For all $k\geq 0$ and $i\in\{1,\dots,N_k\}$,
\[ \Flat(k) = \{(v_{k,i},v_{k,i+1}) : i=1,\dots,2^{k}\}\quad\text{and}\quad V_{k+1,i}(v_{k,i},v_{k,i+1}) = \{v_{k,i},v_{k+1,2i},v_{k,i}\}.\]
\end{enumerate}
\end{lem}

\begin{proof}
For (1), we work by induction. The claim is true for $k=0$ by the choice of points $v_{0,1}$ and $v_{0,2}$. Assume the claim is true for some $k\geq 0$. By the Pythagorean theorem,
\[ |v_{k+1,2i-1}-v_{k+1,2i}| = |v_{k,i} - v_{k+1,2i-1}| =(4^{-1}+t_1^2)^{1/2}|v_{k,i}-v_{k,i+1}| = (4^{-1}+t_1^2)^{1/2}\rho_{k} = \rho_{k+1}.\]
In similar fashion, one can compute $|v_{k+1,2i}-v_{k+1,2i+1}|$ and the proof of (1) is complete.

Claim (3) is immediate from Remark \ref{rem:flatsnowflakes} by taking $s$ sufficiently close to $1$ and $n_0$ sufficiently large.

For (V0), we have
\[ \frac{\rho_{k+1}}{\rho_k} = 2^{-1/s}\left ( \frac{k+2+n_0}{k+1+n_0} \right )^q.\]
Clearly, $\rho_{k+1} > \xi_1\rho_{k}$. On the other hand, since $2^{-1/s} < \xi_2 < 1$, if $n_0$ is sufficiently large, then $\rho_{k+1} \leq \xi_2\rho_k$. Properties (V1), (V2), and (V5) are immediately satisfied by our construction. For (V4), fix a point $v_{k+1,2i} \in V_{k+1}\setminus V_k$. By (1), we have $|v_{k+1,2i} - v_{k+1,2i+1}| = \rho_{k+1}$ and (V4) is satisfied.

For (V3), claim (3), and claim (4), we apply induction on $k$. For $k=0$ (V3) is immediate by the choice of parameters. For claim (3), we note that $\a_{0,v}=0$ for all $v\in V_0$, since $V_0$ contains only 2 points. For the same reason, claim (4) is satisfied when $k=0$.

To show (V3), we note by (3) that the closest point of $V_{k+1}$ to $v_{k+1,2i}$ are the points $v_{k,i}$ and $v_{k,i+1}$. Therefore,
\[ \min_{v\in V_{k+1} \setminus \{v_{k+1,2i}\}} |v-v_{k+1,2i}| = |v_{k+1,2i} - v_{k,i}| = \rho_{k+1}.\]
Similarly, by (3), the closest point of $V_{k+1}$ to $v_{k+1,2i+1}=v_{k,i+1}$ are the points $v_{k+1,2i}$ and $v_{k+1, 2(i+1)}$ (or only one point of these two if $i =0$ or $i = 2^{k}\}$) and the above inequality also applies.

Finally, to show (4), we apply (3) and the arguments in the proof of (V3). Namely, if $v_{k,i} \in V_k$ with $k\in\{2,2^k-1\}$, then $\a_{v_{k,i}} < \a_0$ and $v_{k,i}$ lies between points $v_{k,i-1}$ and $v_{k,i+1}$. Therefore,
\[ \Flat(k) = \{(v_{k,i},v_{k,i+1}) : i=1,\dots,2^{k}\}.\]
Furthermore, the only point of $V_{k+1}$ lying between $v_{k,i}$ and $v_{k,i+1}$ is $v_{k+1,2i}$. Thus,
\[ V_{k+1,i}(v_{k,i},v_{k,i+1}) = \{v_{k,i},v_{k+1,2i},v_{k,i}\}. \qedhere\]
\end{proof}

We now show that there does not exist a $(1/s)$-H\"older map $f:[0,1] \to \R^2$ whose image contains $\bigcup_{k\geq 0}V_k$. Contrary to the claim, assume that such a map $f$ exists and let $H$ be its H\"older constant. For each $v_{k,i}\in V_k$, fix $w_{k,i} \in [0,1]$ such that $f(w_{k,i})=v_{k,i}$. Then
\[ |w_{k,i}-w_{k,j}| \gtrsim_{H,s} |v_{k,i}-v_{k,j}|^s \gtrsim_{n_0,q,s} 2^{-k}(k+1)^{sq}.\]
Therefore,
\begin{align*}
1 \geq 2^k \min_{i=1,\dots,2^{k}+1}|w_{k,i}-w_{k,j}| \gtrsim_{H,s,n_0,q} (k+1)^{sq},
\end{align*}
which diverges as $k\to\infty$ and we reach a contradiction.

It remains to check (\ref{eq:sharp}). By Lemma \ref{lem:propnets}, it suffices to show that
\begin{align*}
\sum_{k=0}^{\infty}\sum_{i=1}^{2^{k}} \tau_s(k,v_{k,i},v_{k,i+1})^p |v_{k,i}-v_{k,i+1}|^s < \infty.
\end{align*}
By the Mean Value Theorem,
\begin{align*}
\tau_s(k,v_{k,i},v_{k,i+1}) &= \frac{|v_{k,i}-v_{k+1,2i}|^s + |v_{k+1,2i}-v_{k,i+1}|^s-|v_{k,i}-v_{k,i+1}|^s}{|v_{k,i}-v_{k,i+1}|^s}\\
&=  2\frac{(k+2+n_0)^{sq} - (k+1+n_0)^{sq}}{(k+2+n_0)^{sq}}\\
&\lesssim_{n_0,s,q}\frac1{k+1}.
\end{align*}
Finally, since $sq-p<-1$,
\begin{align*}
\sum_{k=0}^{\infty}\sum_{i=1}^{2^{k}} \tau_s(k,v_{k,i},v_{k,i+1})^p |v_{k,i}-v_{k,i+1}|^s &\lesssim_{n_0,s,q}  \sum_{k=0}^{\infty} \frac{2^k}{ (k+1)^{p}}(2^{-k/s}(k+1)^q)^s\\
&= \sum_{k=0}^{\infty}(k+1)^{sq-p}<\infty.
\end{align*}

\addtocontents{toc}{\protect\vspace{10pt}}

\appendix

\section{Tours on connected, finite simple graphs}\label{sec:appendix-a}
A \emph{finite simple graph} $G = (V,E)$ in a Banach space $X$ is a finite set of points $V\subset X$ (called \emph{vertices} of $G$) along with a set $E\subseteq \{\{v,v'\} : \text{distinct } v,v' \in V\}$ (called \emph{edges} of $G$). We may identify edges $\{v,v'\}$ in the graph with the (unoriented) line segments $[v,v']$ in $X$. A graph is \emph{connected} if every pair of vertices in the graph can be joined by a sequence of edges in the graph.
The \emph{valence} of a vertex $v$ in $G$ is the number of edges in $G$ that contain $v$.

\begin{prop}\label{prop:graph}
Let $G$ be a connected, finite simple graph in $X$. Assume that every vertex in $G$ has valence at most 2. For any vertex $v_0$ in $G$ and any nondegenerate compact interval $\Delta$, there exists a collection $\mathcal{I}$ of open intervals, whose closures are mutually disjoint and contained in the interior of $\D$, and there exists a continuous map $g : \D \to G$ with the following properties.
\begin{enumerate}
\item The endpoints of $\D$ are mapped onto $v_0$.
\item For every vertex $v$ of $G$, there exists at least one component $J$ of $\D\setminus \bigcup\mathcal{I}$ such that $g(J)=v$. Conversely, for every component $J$ of $\D\setminus \bigcup\mathcal{I}$, there exists a vertex $v$ of $G$ such that $g(J)=v$.
\item Each interval in $\mathcal{I}$ is mapped linearly onto some edge $e$ of $G$. Conversely, for each edge $e$ of $G$, there exist exactly two intervals $I \in \mathcal{I}$ such that $g(I) = e$.
\item For any vertex $v$ in $G$, there exists a component $J$ of $g^{-1}(v)\cap\D\setminus \bigcup\mathcal{I}$ such that for any edge $e$ containing $v$ as an endpoint, there exists $I\in\mathcal{I}$ such that $\overline{I}\cap J \neq \emptyset$ and $g(I) = e$. 
\end{enumerate}
\end{prop}

\begin{proof} If we only desired properties (1)--(3), then we could prove the proposition without any restriction on the valency of the vertices by simple induction on the number of edges. The restriction on the valency of the vertices ensures the graph has one of two simple forms that make it easy to describe maps $g$ satisfying properties (1)--(4). Thus, let $G$ be a connected, finite simple graph in $X$, and assume that every vertex in $G$ has valence at most 2. The conclusion being trivial otherwise, we may assume that $G$ contains at least two vertices. There are two possibilities. In each case, we will construct the family $\mathcal{I}$ and the map $g$, but leave verification of properties (1) through (4) to the reader.

\emph{Case 1: Suppose that every vertex of $G$ has valence 2 (i.e.~ $G$ is a ``cycle").} Then we can find an enumeration $\{u^1,\dots,u^k\}$ of the vertices of $G$ so that the edges of $G$ are precisely $\{[u^i,u^{i+1}]: i=1,\dots,k\}$, where we set $u^{k+1}=u^1$. Without loss of generality, assume that $u^1 = v_0$. Let $\mathcal{I} = \{I_1,\dots, I_{2k}\}$ be open intervals, enumerated according to the orientation of $\D$, whose closures are mutually disjoint and contained in the interior of $\D$. Then there exists a continuous, surjective map $g : \D \to G$ such that
\begin{enumerate}
\item $g$ is linear on each $I_i$ and constant on each component of $\D\setminus \bigcup_{i=1}^{2k}I_i$;
\item for each $i\in \{1,\dots, k\}$, $g$ maps $I_i$ linearly onto $[u^i,u^{i+1}]$ and maps the left endpoint of $I_i$ onto $u^i$;
\item for each $i\in \{k+1,\dots, 2k\}$, $g$ maps $I_i$ linearly onto $[u^{i-k},u^{i-k+1}]$ and maps the left endpoint of $I_i$ onto $u^{i-k}$.
\end{enumerate}
That is, $g$ winds twice around the graph, starting and ending at $v_0=u^1$.

\emph{Case 2: Suppose that least one vertex of $G$ has valence 1 (i.e.~ $G$ is an ``arc")}. Then we can find an enumeration $\{u^1,\dots,u^k\}$ of the vertices of $G$ so that the edges of $G$ are precisely $\{[u^i,u^{i+1}]: i=1,\dots,k-1\}$. In this case, $u^1$ and $u^k$ have valence 1 and all other vertices have valence 2. Assume that $v_0 = u^l$ for some $1\leq l\leq k$. Let $\mathcal{I} = \{I_1,\dots, I_{2(k-1)}\}$ be open intervals, enumerated according to the orientation of $\D$, whose closures are mutually disjoint and contained in the interior of $\D$. Then there exists a continuous, surjective map $g : \D \to G$ such that
\begin{enumerate}
\item $g$ is linear on each $I_i$ and constant on each component of $\D\setminus \bigcup_{i=1}^{2(k-1)}I_i$;
\item for each $1\leq i\leq l-1$ (if any), $g$ maps $I_i$ linearly onto $[u^{l-i},u^{l-i+1}]$ and maps the left endpoint of $I_i$ onto $u^{l-i+1}$;
\item for each $i\in \{l,\dots, l+k-2\}$, $g$ maps $I_i$ linearly onto $[u^{i-l+1},u^{i-l+2}]$ and maps the left endpoint of $I_i$ onto $u^{i-l+1}$;
\item for each $l+k-1\leq i \leq 2(k-1)$ (if any), $g$ maps $I_i$ linearly onto $[u^{2k-2+l-i},u^{2k-1+l-i}]$ and maps the left endpoint of $I_i$ onto $u^{2k-1+l-i}$.
\end{enumerate} That is, $g$ walks along the graph from $v_0=u^l$ towards $u^1$, walks from $u^1$ to $u^k$, and walks from $u^k$ back to $u^l$.
\end{proof}

\section{From Lipschitz to H\"older parameterizations}\label{sec:lip-hold}

The following method of obtaining H\"older parameterizations from Lipschitz ones is well known, see e.g.~ \cite[Lemma VII.2.8]{SS}.  We include Lemma \ref{l:LipHold} and its proof here to have a clear statement about the dependence of the H\"older constant of the map $f$.

\begin{lem}\label{l:LipHold} Let $s> 1$, $M>0$, $0<\xi_1\leq \xi_2<1$, $\alpha>0$, $\beta>0$, and $j_0\in\Z$. Let $(X,|\cdot|)$ be a Banach space. Suppose that $\rho_j$ $(j\geq j_0)$ is a sequence of scales and  $f_j:[0,M]\rightarrow X$ $(j\geq j_0)$ is a sequence of Lipschitz maps satisfying \begin{enumerate}
\item $\rho_{j_0}=1$ and $\xi_1 \rho_j \leq \rho_{j+1}\leq \xi_2 \rho_j$ for all $j\geq j_0$,
\item $|f_j(x)-f_j(y)| \leq A_j|x-y|$ for all $j\geq j_0$, where $A_j \leq \alpha \rho_j^{1-s}$, and
\item $|f_j(x)-f_{j+1}(x)| \leq B_j$ for all $j\geq j_0$, where $B_j \leq \beta \rho_j$.\end{enumerate} Then
$f_j$ converges uniformly to a map $f:[0,M]\rightarrow X$ such that $$|f(x)-f(y)|\leq H |x-y|^{1/s}\quad\text{for all $x,y\in[0,M]$},$$ where $H$ is a finite constant depending only on $\max(M,1/M)$, $\xi_1$, $\xi_2$, $\alpha$, and $\beta$; see \eqref{e:geoH}.  \end{lem}

\begin{proof} Define $f:[0,M]\rightarrow X$ pointwise by $f(x)=f_{j_0}(x)+\sum_{k=j_0}^\infty (f_{k+1}(x)-f_k(x)).$ Then $f$ exists and is the uniform limit of the maps $f_j$ by (3), because $\sum_{k=j_0}^\infty B_k<\infty$. In fact, for all $j\geq j_0$ and $x\in[0,M]$,
\begin{equation}\label{e:geo1} |f(x)-f_j(x)| \leq  \sum_{k=j}^\infty |f_{j+1}(x)-f_j(x)| \leq \sum_{k=j}^\infty \beta \rho_{k} \leq  \frac{\beta}{1-\xi_2} \rho_j.
\end{equation}
Suppose that $x,y\in[0,M]$ with $x\neq y$. Then there is a unique integer $j\geq j_0$ such that
\begin{equation}\label{e:geo2} M\rho_{j+1}^s<|x-y|\leq M\rho_j^s.\end{equation}
By the triangle inequality, (2), and \eqref{e:geo1},
\begin{equation} \begin{split} \label{e:geo3} |f(x)-f(y)| &\leq |f_j(x)-f_j(y)|+|f(x)-f_j(x)|+|f(y)-f_j(y)|\\ &\leq \alpha \rho_j^{1-s}|x-y| + \frac{2\beta}{1-\xi_2}\rho_j.\end{split}\end{equation}
By the first inequality in \eqref{e:geo2} and (1), we have
\begin{equation}\label{e:geo4} \rho^{j} < \frac{1}{M^{1/s}\xi_1}|x-y|^{1/s}.\end{equation} Hence, by the second inequality in \eqref{e:geo2}, we have
\begin{equation} \label{e:geo5} \rho_j^{1-s}|x-y| \leq M\rho_j < \frac{M}{M^{1/s}\xi_1}|x-y|^{1/s}.\end{equation}
Combining \eqref{e:geo3}--\eqref{e:geo5} yields $|f(x)-f(y)|\leq h|x-y|^{1/s}$ for all $x,y\in [0,M]$, where
\begin{equation}\label{e:geo6} h= \frac{1}{M^{1/s}\xi_1}\left(\alpha M+\frac{2\beta}{1-\xi_2}\right).\end{equation}
If $M\geq 1$, then $1/M^{1/s}\leq 1$, while if $M\leq 1$, then $1/M^{1/s}\leq 1/M$. Thus, it follows that $|f(x)-f(y)|\leq H|x-y|^{1/s}$ for all $x,y\in [0,M]$, where
\begin{equation}\label{e:geoH} H= \frac{\max(1,1/M)}{\xi_1}\left(\alpha M+\frac{2\beta}{1-\xi_2}\right)=\frac{\alpha}{\xi_1}\max(1,M)+\frac{2\beta}{\xi_1(1-\xi_2)}\max(1,1/M)\end{equation}
depends only on $\max(M,1/M)$, $\xi_1$, $\xi_2$, $\alpha$, and $\beta$.
\end{proof}

\begin{rem}Lemma \ref{l:LipHold} is often used with geometric scales $\rho_j = \rho^j$ (i.e.~ $\xi_1=\xi_2=\rho$). However, separating the parameters $\xi_1$ and $\xi_2$ provides additional flexibility that can make constructing examples easier; see e.g.~ \S\ref{sec:sharpnessof1}.\end{rem}

\bibliography{tsp-refs}
\bibliographystyle{amsbeta}

\end{document}